\documentclass[11pt,dvipsnames,svgnames,table]{article}
\usepackage[T1]{fontenc}
\usepackage{graphicx,amsmath,amsthm,amsfonts,amssymb,microtype,thmtools,mathtools,anyfontsize,thm-restate,verbatim,tikz}
\usepackage{xcolor}
\usetikzlibrary{intersections}
\usetikzlibrary{external}
\usepackage{esvect}
\usepackage{wrapfig}
\usepackage{paralist, tabularx}
\usepackage{subcaption}
\usepackage{float}
\usepackage[shortlabels]{enumitem}
\setlist[itemize]{topsep=0ex,itemsep=0ex,parsep=0.4ex}
\setlist[enumerate]{topsep=0ex,itemsep=0ex,parsep=0.4ex}
\usepackage[unicode=true]{hyperref}
\hypersetup{
colorlinks=true,
breaklinks=true,
linkcolor={black},
citecolor={black},
urlcolor={blue!60!black},
pdftitle={Structure of $k$-matching-planar graphs},
pdfauthor={Kevin Hendrey, Nikolai Karol, David~R.~Wood}} 
\usepackage[noabbrev,capitalise]{cleveref}
\crefname{lem}{Lemma}{Lemmas}
\crefname{thm}{Theorem}{Theorems}
\crefname{cor}{Corollary}{Corollaries}
\crefname{prop}{Proposition}{Propositions}
\crefname{conj}{Conjecture}{Conjectures}
\crefname{open}{Open Problem}{Open Problems}
\crefname{question}{Question}{Questions}
\crefname{claim}{Claim}{Claims}
\crefformat{equation}{(#2#1#3)}
\Crefformat{equation}{Equation #2(#1)#3}

\newcommand{\header}[1]{\textbf{\textit{#1}}\textbf{:}}
\newcommand{\defn}[1]{\textcolor{Maroon}{\emph{#1}}}

\newcommand{\CPL}{Coloured Planarisation Lemma}
\newcommand{\DL}{Distance Lemma}

\usepackage[longnamesfirst,numbers,sort&compress]
{natbib}
\makeatletter
\def\NAT@spacechar{~}
\makeatother
\usepackage[margin=31mm]{geometry}
\DeclarePairedDelimiter{\ceil}{\lceil}{\rceil}

\renewcommand{\geq}{\geqslant}
\renewcommand{\leq}{\leqslant}

\DeclareMathOperator{\Int}{Int}

\DeclareMathOperator{\dist}{dist}

\DeclareMathOperator{\ltw}{ltw}
\DeclareMathOperator{\tw}{tw}
\DeclareMathOperator{\rtw}{rtw}

\DeclareMathOperator{\st}{st}
\DeclareMathOperator{\boxicity}{box}
\DeclareMathOperator{\scol}{scol}
\DeclareMathOperator{\wcol}{wcol}
\DeclareMathOperator{\level}{level}
\renewcommand{\thefootnote}{\fnsymbol{footnote}}
\numberwithin{equation}{section}
\theoremstyle{plain}
\newtheorem{thm}[equation]{Theorem}
\newtheorem{lem}[equation]{Lemma}

\newtheorem{cor}[equation]{Corollary}
\newtheorem{prop}[equation]{Proposition}
\newtheorem{obs}[equation]{Observation}
\newtheorem{question}[equation]{Question}
\newtheorem{claim}{Claim}[thm]
\newtheorem{subclaim}{Subclaim}[claim]
\theoremstyle{definition}

\newcommand{\PP}{\mathcal{P}}

\begin{document}

\author{Kevin Hendrey\footnotemark[2] \qquad 
Nikolai Karol\footnotemark[2] \qquad
David~R.~Wood\footnotemark[2]}

\footnotetext[2]{School of Mathematics, Monash   University, Melbourne, Australia  (\texttt{\{Kevin.Hendrey1,Nikolai.Karol,David.Wood\}@monash.edu}). Research of Hendrey is supported by the Australian Research Council. Research of Wood is supported by the Australian Research Council and by NSERC.}

\sloppy

\title{\bf\boldmath Structure of $k$-Matching-Planar Graphs}

\maketitle

\begin{abstract}
For $k \geqslant 0$, we define a simple topological graph $G$ (that is, a graph drawn in the plane such that every pair of edges intersect at most once, including endpoints) to be \textit{$k$-matching-planar} if for every edge $e \in E(G)$, every matching amongst the edges of $G$ that cross $e$ has size at most~$k$. The class of $k$-matching-planar graphs is a significant generalisation of many other existing beyond planar graph classes, including $k$-planar graphs. We prove that every simple topological $k$-matching-planar graph is isomorphic to a subgraph of the strong product of a graph with bounded treewidth and a path. This result qualitatively extends the planar graph product structure theorem of Dujmović, Joret, Micek, Morin, Ueckerdt, and Wood [\emph{J.~ACM} 2020] and recent product structure theorems for other beyond planar graph classes.  Using this result, we deduce that the class of simple topological $k$-matching-planar graphs has several attractive properties, such as bounded queue number, bounded nonrepetitive chromatic number, polynomial $p$-centred chromatic numbers, bounded boxicity, bounded strong and weak colouring numbers, and asymptotic dimension $2$. This makes the class of simple topological $k$-matching-planar graphs the broadest class of simple beyond planar graphs in the literature that has these attractive structural properties. All of our results about simple topological $k$-matching-planar graphs generalise to the non-simple setting, where the maximum number of pairwise crossing edges incident to a common vertex becomes relevant.

The paper introduces several tools and results of independent interest. We show that every simple topological $k$-matching-planar graph admits an edge-colouring with $\mathcal{O}(k^{3}\log k)$ colours such that monochromatic edges do not cross. We introduce the concept of weak shallow minors, which subsume and generalise shallow minors, a key concept in graph sparsity theory. A central element of the proof of our product structure theorem is that every simple topological $k$-matching-planar graph can be described as a weak shallow minor of the strong product of a planar graph with a small complete graph. 
We then develop new general-purpose tools to establish a product structure theorem for weak shallow minors of the strong product of a bounded genus graph with a small complete graph, from which our main product structure theorem follows. As a byproduct of our proof techniques, we establish upper bounds on the treewidth of graphs with well-behaved circular drawings that qualitatively generalise several existing results.
\end{abstract}

\renewcommand{\thefootnote}
{\arabic{footnote}}

\newpage\tableofcontents
\newpage

\section{Introduction}

Beyond planar graphs is a vibrant research topic within the graph drawing community that studies drawings of graphs in the plane, where crossings are controlled in some way (see the surveys \citep{DLM19,HT20}). One line of research on beyond planar graphs shows that certain structural properties of planar graphs also hold for specific beyond planar graph classes. Our goal is to prove such a structural result for the broadest possible class of beyond planar graphs. To this end, we consider a class of beyond planar graphs that generalises many other existing classes. Our main result establishes a product structure theorem for this class that generalises recent product structure theorems for planar graphs and other beyond planar classes.

\subsection{$k$-Matching-Planar Graphs} \label{introfirst}

A natural way to generalise planar graphs is to allow a bounded number of crossings per edge. We use the term `topological graph' to mean a drawing of a graph in the plane (see \cref{topologicalgraphs} for a detailed definition). For an integer $k \geqslant 0$, a topological graph is \defn{$k$-planar}~\citep{PachToth97} if every edge is involved in at most $k$ crossings. A graph is \defn{$k$-planar} if it is isomorphic to a topological $k$-planar graph. The class of $k$-planar graphs is a classical and well-studied example of beyond planar graphs; see \citep{GB07,KLM17,PachToth97,DEW17,HLRT20} for example.

This paper considers\footnote{\citet{AFPS14} considered topological graphs that contain no so-called $(k, 1)$-grid with distinct vertices, which are almost equivalent to $k$-matching-planar graphs (see \cref{sectionbackground,SectionColouring} for a detailed discussion).  \citet{MSSU24} considered \defn{$k$-independent crossing graphs}, which are equivalent to $k$-matching planar graphs.} the following substantial generalisation of $k$-planar graphs. For an integer $k \geqslant 0$, we define a topological graph $G$ to be \defn{$k$-matching-planar} if for every edge $e\in E(G)$, the matching number of the set of edges of $G$ that cross $e$ is at most $k$. Equivalently, this can be formulated by forbidding the configuration where an edge is crossed by $k + 1$ edges, no two of which share a common endpoint. A graph is \defn{$k$-matching-planar} if it is isomorphic to a topological $k$-matching-planar graph. Every $k$-planar graph is $k$-matching-planar, but not vice versa. For example, the complete bipartite graph  $K_{3,n}$ is $1$-matching-planar (see \cref{K3n}(a)), but in every drawing of $K_{3,n}$ some edge is crossed $\Omega(n)$ times, since $K_{3,n}$ has $\mathcal{O}(n)$ edges and crossing number $\Omega(n^2)$ \citep{Kleitman70}. More generally, $K_{2k + 2, n}$ is $k$-matching-planar for all $k \geqslant 0$ and $n \geqslant 1$. Thus, the class of $k$-matching-planar graphs is a significant generalisation of the class of $k$-planar graphs. 

  \begin{figure}[h]
  \centering
  (a) \scalebox{0.8}{\includegraphics{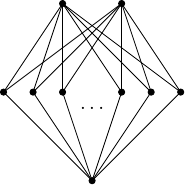}}
  \hspace*{20mm}
  (b) \scalebox{0.8}{\includegraphics{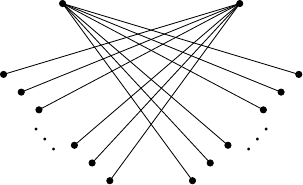}}
 \caption{(a) $K_{3, n}$ is $1$-matching-planar. 
(b) A topological $1$-matching-planar graph, where every edge crosses $n$ edges.}
      \label{K3n}
      \label{crossingstars}
  \end{figure}

While each edge of a topological $k$-planar graph is involved in a bounded number of crossings, this is not true for topological $k$-matching-planar graphs. For example, consider a topological graph that consists of two crossing stars, each with $n$ leaves (see \cref{crossingstars}(b)). Then every edge crosses $n$ edges, and this topological graph is $1$-matching-planar. Thus, every edge in a topological $k$-matching-planar graph can cross arbitrarily many other edges. This makes the study of $k$-matching-planar graphs attractive and more difficult compared to $k$-planar graphs.

\subsection{Product Structure Theory} \label{subsectionproductstructure}
 
Recently, there has been significant progress in understanding the global structure of planar graphs through the lens of graph products. Say a graph $H$ is \defn{contained} in a graph $G$ if $H$ is isomorphic to a subgraph of $G$. \citet{DJMMUW20} established that every planar graph is contained in the strong product of a graph with bounded treewidth and a path.

\begin{thm} [Planar Graph Product Structure Theorem \citep{DJMMUW20}] 
\label{GPST} Every planar graph is contained in $H \boxtimes P$ for some graph $H$ of treewidth at most $8$ and for some path $P$.    
\end{thm}

\cref{GPST} has been the key tool to resolve several major open problems regarding queue layouts~\citep{DJMMUW20}, nonrepetitive colourings~\citep{DEJWW20}, centred colourings~\citep{DFMS21}, adjacency labelling schemes~\citep{GJ22,BGP22,EJM23,DEGJMM21}, twin-width~\citep{BDHK24,JP22,KPS24}, comparable box dimension~\cite{DGLTU22}, infinite graphs~\citep{HMSTW}, and transducibility lower bounds~\citep{HJ25,HJ25a,GPP}. 
This breakthrough result led to a new direction in the study of sparse graphs, now called graph product structure theory, which aims to describe complicated graphs as subgraphs of strong products of simpler building blocks. Treewidth is the standard measure of how similar a graph is to a tree, and is of fundamental importance in structural and algorithmic graph theory (see \cref{LTWsubsection} for a formal definition and \citep{Bodlaender98, HW17, Reed03} for surveys about treewidth). The treewidth of a graph $G$ is denoted by \defn{$\tw(G)$}. Graphs with bounded treewidth are considered to be simple and are well understood. \cref{GPST} therefore reduces problems on a complicated class of graphs (planar graphs) to a simpler class of graphs (bounded treewidth).

Motivated by \cref{GPST}, \citet{BDJMW22} defined the \defn{row treewidth} of a graph $G$, denoted \defn{$\rtw(G)$}, to be the minimum treewidth of a graph $H$ such that $G$ is contained in $H \boxtimes P$ for some path $P$. \cref{GPST} implies that planar graphs have row treewidth at most $8$. \citet{UWY22} strengthened \cref{GPST} by improving the upper bound to $6$.

Several extensions of \cref{GPST} have been established. In the setting of minor-closed classes, it has been shown that graphs with bounded Euler genus \citep{DHHW22,DJMMUW20} and apex-minor-free graphs~\citep{DJMMUW20} have bounded row treewidth. Several non-minor-closed classes also have bounded row treewidth, including various beyond planar graph classes: $k$-planar graphs~\citep{DMW23,HW24,DHSW24}, fan-planar graphs~\citep{HW24}, $k$-fan-bundle-planar graphs~\citep{HW24}, squaregraphs~\citep{HJMW24}, $d$-map graphs~\citep{BDHK24,DMW23}, $h$-framed graphs~\citep{BDHK24}, and powers of bounded degree planar graphs~\citep{DMW23,HW24,DHSW24}. \citet{HJ24a} established analogous results representing graphs of bounded row treewidth as induced subgraphs of $H \boxtimes P$.

A topological graph is \defn{simple} if any two edges intersect in at most one point including endpoints. A \defn{geometric graph} is a topological graph in which every edge is a straight line segment. Every geometric graph is simple. Much of the existing graph drawing literature focuses on simple topological graphs or geometric graphs. Our primary result is a product structure theorem for simple topological $k$-matching-planar graphs, which qualitatively generalises \cref{GPST} and resolves a conjecture of \citet[Conjecture~21]{MSSU24}. In fact, we do not require simplicity. 

\begin{thm} \label{introRTW} 
Let $G$ be a topological $k$-matching-planar graph with no $t$ pairwise crossing edges incident to a common vertex. Then $\rtw(G) \leqslant f(k, t)$ for some function $f$. That is, $G$ is contained in $H \boxtimes P$ for some graph $H$ of treewidth at most $f(k, t)$ and for some path $P$.
\end{thm}

\cref{introRTW} provides the broadest known criterion for a beyond planar graph class to admit a product structure theorem (see \cref{RTWbounded} for a detailed discussion).

\subsection{Applications} \label{applications}

We now describe some applications of our main result, \cref{introRTW}.

\header{Labelling Schemes} Here the task is to assign labels to the vertices of a graph so that one can decide whether two vertices are adjacent by looking at their labels. \citet{DEGJMM21} used the Planar Graph Product Structure Theorem to show that $n$-vertex planar graphs have labelling schemes using $(1+o(1))\log_2 n$ bits, which is best possible and improves on a 30-year sequence of results. Equivalently, there is a graph $U$ on $n^{1+o(1)}$ vertices
that is \defn{universal} for the class of $n$-vertex planar graphs, meaning that every $n$-vertex planar graph is isomorphic to an induced subgraph of $U$. More generally, \citet{EJM23} constructed such a universal graph with $n^{1+o(1)}$  vertices and edges. Both results hold for any class with bounded row treewidth. \cref{introRTW} thus implies the following generalisation of their results.

\begin{thm}
\label{universal}
For any fixed integers $k \geqslant 0$ and $t \geqslant 2$ and for every integer $n \geqslant 1$ there is a  graph with $n^{1+o(1)}$ vertices and edges that is universal for the class of $n$-vertex topological $k$-matching-planar graphs with no $t$ pairwise crossing edges incident to a common vertex.
\end{thm}

\header{Queue Number} 
Heath, Leighton, and Rosenberg~\citep{HLR92,HR92} introduced queue number as a way to measure the power of queues to represent graphs\footnote{The \defn{queue number} of a graph $G$ is the minimum integer $k$ such that there is a vertex ordering $\sigma$ of $V(G)$ and a partition $E_1,\dots,E_k$ of $E(G)$, such that for each $i\in\{1,\dots,k\}$, no two edges in $E_i$ are nested with respect to $\sigma$. Here edges $uw, xy \in E(G)$ with $\sigma(u) < \sigma(w)$ and $\sigma(x) < \sigma(y)$  are \defn{nested} with respect to $\sigma$ if $\sigma(u) < \sigma(x) < \sigma(y) < \sigma(w)$ or $\sigma(x) < \sigma(u) < \sigma(w) < \sigma(y)$.}. \citet{DJMMUW20} proved that planar graphs have bounded queue number, resolving a long-standing conjecture of \citet{HLR92}. Upper bounds on queue number for graphs of given treewidth \citep{Wiechert17} and for graph products \citep{Wood05} imply that the queue number of every graph $G$ is at most $3 \cdot 2^{\rtw(G)} - 2$. Thus \cref{introRTW} implies the following:

\begin{thm} \label{introqueue} Let $G$ be a topological $k$-matching-planar graph with no $t$ pairwise crossing edges incident to a common vertex. Then the queue number of $G$ is at most some function~$f(k, t)$. 
\end{thm}

\header{Nonrepetitive Colourings} 
Nonrepetitive colourings were introduced by \citet{AGHR02}, and have since been widely studied (see the survey \citep{Wood21})\footnote{The \defn{nonrepetitive chromatic number} of a graph $G$ is the minimum number of colours in a vertex-colouring $\eta$ of $G$ such that there is no path $(v_1,v_2,\dots,v_{2t})$ in $G$ with $\eta(v_i)=\eta(v_{t+i})$ for each $i\in\{1,\dots,t\}$.}.  \citet{DEJWW20} proved that planar graphs have bounded nonrepetitive chromatic number, resolving a long-standing conjecture of \citet{AGHR02}. The proof uses their result that the nonrepetitive chromatic number of every graph $G$ is at most $4^{\rtw(G) + 1}$ (see \citep[Theorem $7$ and Corollary $9$]{DEJWW20}). Thus \cref{introRTW} implies the following:

\begin{thm} \label{intrononrepetitive} Let $G$ be a topological $k$-matching-planar graph with no $t$ pairwise crossing edges incident to a common vertex. Then the nonrepetitive chromatic number of $G$ is at most some function $f(k, t)$.
\end{thm}

\header{Centred Colourings}
\citet{NesOdM-GradI} introduced the concept of centred colourings\footnote{The \defn{$p$-centred chromatic number} of a graph $G$ is the minimum number of colours in a vertex-colouring $\eta$ of $G$ such that for every connected subgraph $X$ of $G$, $|\{\eta(v) : v \in V(X)\}| > p$ or there exists some $v \in V(X)$ such that $\eta(v) \neq \eta(w)$ for every $w \in V(X) \setminus \{v\}$.}, which are important within graph sparsity theory since they characterise graph classes with bounded expansion \citep{NesOdM-GradI}. The best known bound on the $p$-centred chromatic number of planar graphs is $\mathcal{O}(p^{3}\log p)$ due to~\citet{DFMS21}.  Combining the results of \citet{DFMS21} and 
\citet[Lemma~15]{PS21}, \citet{DMW23} observed that the $p$-centred chromatic number of every graph $G$ is 
$\mathcal{O}(p^{\rtw(G) + 1})$. Thus \cref{introRTW} implies the following:

\begin{thm} \label{introcentred}
Let $G$ be a topological $k$-matching-planar graph with no $t$ pairwise crossing edges incident to a common vertex. Then for each positive integer $p$, the $p$-centred chromatic number of $G$ is $\mathcal{O}(p^{f(k, t)})$ for some function $f$.
\end{thm}

\header{Intersection Graphs} Let $S$ be a convex polygon in the plane. Denote its area by \defn{$||S||$}. A \defn{homothetic copy} of $S$ is a convex polygon in the plane that can be obtained from $S$ by scaling and translation. For a real number $\alpha \in [0, 1]$, a collection $\{S_{a} : a \in A\}$ of homothetic copies of $S$ is \defn{$\alpha$-free} if for every $a \in A$, we have $||S_{a} \setminus \bigcup_{b \in A \setminus\{a\}}S_{b}||\geqslant \alpha \cdot ||S_{a}||$. 

 \citet*{MSSU24} analysed under which conditions the class of intersection graphs of $\alpha$-free homothetic copies of a regular $k$-gon has row treewidth bounded by a function of $k$. To this end, for integers $k \geqslant \ell \geqslant 4$, they defined \defn{$\Delta_{k}^{\ell}$} to be the convex hull of $\ell$ consecutive corners of a regular $k$-gon with area $1$, and 
 \begin{equation*}
    s(k) :=
    \begin{cases*}
        ||\Delta_k^{k/2+2}||                 & if $k \equiv 0 \pmod  4$ \\
        ||\Delta_k^{\lceil k/2 \rceil+1}||   & if $k \equiv 1 \pmod  4$ \\
        ||\Delta_k^{k/2+1}|| = \frac12       & if $k \equiv 2 \pmod  4$ \\
        ||\Delta_k^{\lceil k/2 \rceil+2}||   & if $k \equiv 3 \pmod  4$.
    \end{cases*}
\end{equation*}

\citet[Proposition~$19$]{MSSU24} showed that for any $\alpha \in [s(k), 1]$, the intersection graph of any collection of $\alpha$-free homothetic copies of a regular $k$-gon is isomorphic to a topological $26(k + 1)$-matching-planar graph where no two edges incident to a common vertex cross\footnote{\citet[Proposition~$19$]{MSSU24} used the notion of so-called canonical drawings and established that the canonical drawing of this intersection graph is topological $26(k + 1)$-matching-planar. Section~$4$ of their paper shows that no two edges incident to a common vertex cross in any canonical drawing.}. Thus we have the following corollary of \cref{introRTW}, which resolves a conjecture of \citet[Conjecture~$22$]{MSSU24}.

\begin{thm} \label{introIntersection} 
For an integer $k \geqslant 4$ and fixed $\alpha \in [s(k), 1]$, if $G$ is the intersection graph of a collection of $\alpha$-free homothetic copies of a regular $k$-gon, then $\rtw(G) \leqslant f(k)$ for some function $f$.
\end{thm}

Note that \citet[Conjecture~$20$]{MSSU24} conjectured that $s(k)$ is a tight threshold for bounded row treewidth; that is, for fixed $k \geqslant 4$ and $\alpha \in [0, s(k))$, the class of intersection graphs of $\alpha$-free homothetic copies of a regular $k$-gon has unbounded row treewidth.
This remains open.

\header{Layered Treewidth} 
Layered treewidth is a precursor to graph product structure theory,  independently introduced by \citet{DMW17} and \citet{Shahrokhi13} (see \cref{LTWsubsection} for a formal definition). The layered treewidth of a graph $G$ is denoted by \defn{$\ltw(G)$}. \citet{DMW17} showed that planar graphs have bounded layered treewidth.

\begin{thm} [\citep{DMW17}] \label{layeredtreewidthplanargraphs} Every planar graph has layered treewidth at most $3$.
\end{thm}

The following relation is well-known (see, for example, \citep[Section~2]{BDJMW22} for a proof).

\begin{lem} \label{LTWandRTW} For every graph $G$, $\ltw(G) \leqslant \rtw(G) + 1$.
    
\end{lem}

\cref{introRTW,LTWandRTW} imply the following:

\begin{thm} \label{introLTW} Let $G$ be a topological $k$-matching-planar graph with no $t$ pairwise crossing edges incident to a common vertex. Then $\ltw(G) \leqslant f(k, t)$ for some function $f$. 
\end{thm}

Although \cref{introLTW} follows directly from \cref{introRTW,LTWandRTW}, we give a separate proof of \cref{introLTW}, providing an asymptotically better bound on layered treewidth than our bound on row treewidth.

\header{Treewidth and Separators} 
Sergey Norin showed that 
$\tw(G)\leqslant 2\sqrt{\ltw(G)\,n} - 1$
for every graph $G$ with $n$ vertices  (see \citep[Lemma~10]{DMW17}). Thus \cref{introLTW} implies the following:

\begin{thm} \label{introseparators} Let $G$ be an $n$-vertex topological $k$-matching-planar graph with no $t$ pairwise crossing edges incident to a common vertex. Then $\tw(G) \in \mathcal{O}_{k, t}(\sqrt{n})$.
\end{thm}

\cref{introseparators} implies that every $n$-vertex topological $k$-matching-planar graph with no $t$ pairwise crossing edges incident to a common vertex has a balanced separator of order $\mathcal{O}_{k, t}(\sqrt{n})$~\citep{RS-II}. This generalises the classical result of \citet{LT79}, which says that every $n$-vertex planar graph has a balanced separator of size $\mathcal{O}(\sqrt{n})$.

\header{Local Treewidth}
\citet{Eppstein-Algo00} introduced the following definition under the guise of the ‘treewidth-diameter’ property. A graph class $\mathcal{G}$ has \defn{bounded local treewidth} if there is a function $f$ such that for every graph $G \in \mathcal{G}$, for every vertex $v \in V(G)$ and for every integer $r > 0$, the subgraph of $G$ induced by the set of vertices at distance at most $r$ from $v$ has treewidth at most $f(r)$.

\citet[Lemma~6]{DMW17} proved that every graph with layered treewidth $\ell$ and radius $r$ has treewidth at most $\ell(2r + 1) - 1$. This implies that every graph class with bounded layered treewidth has bounded local treewidth. By \cref{introLTW}, the class of topological $k$-matching-planar graphs with no $t$ pairwise crossing edges incident to a common vertex has bounded local treewidth. In other words, we have the following:

\begin{thm} \label{introradius}   
Let $G$ be a topological $k$-matching-planar graph with radius $r$ such that no $t$ edges incident to a common vertex pairwise cross. Then $\tw(G) \leqslant r \cdot f(k, t)$ for some function~$f$.
\end{thm}

\cref{introradius} is a qualitative generalisation of the following classical result of \citet{RS-III}.

\begin{thm} [\citep{RS-III}] \label{PlanarBoundedRadius} Every planar graph with radius $r$ has treewidth at most $3r + 1$.
    
\end{thm}

\header{Approximation Algorithms} As pointed out by \citet[Section~1]{Eppstein-Algo00}, Baker's method \citep{Baker94} shows that graph classes with bounded local treewidth admit linear-time approximation schemes for many NP-complete problems such as maximum independent set, minimum vertex cover, and minimum dominating set. So by \cref{introradius}, these results hold for topological $k$-matching-planar graphs with no $t$ pairwise crossing edges incident to a common vertex.

\header{Boxicity} 
The \defn{boxicity} of a graph $G$, denoted by \defn{$\boxicity(G)$}, is the minimum integer $d \geqslant 1$, such that $G$ is the intersection graph of axis-aligned boxes in $\mathbb{R}^{d}$. \citet{Thomassen86} showed that planar graphs have boxicity at most $3$. \citet{SW20} showed that $\boxicity(G) \leqslant 6\ltw(G) + 4$ for every graph $G$. Thus \cref{introLTW} implies:

\begin{thm} \label{introboxicity} Let $G$ be a topological $k$-matching-planar graph with no $t$ pairwise crossing edges incident to a common vertex. Then $\boxicity(G) \leqslant f(k, t)$ for some function $f$.
\end{thm}

\header{Generalised Colouring Numbers}
\citet{KY03} introduced the concepts of strong and weak colouring numbers. 
For a graph $G$ and an integer $s\geqslant 1$, $\scol_s(G)$ and $\wcol_s(G)$ denote the $s$-strong colouring number of $G$ and the $s$-weak colouring number of $G$ respectively. Colouring numbers are important because they characterise bounded expansion~\citep{Zhu09} and nowhere dense classes \citep{GKRSS18}, and have several algorithmic applications \cite{Dvorak13,GKS17}. Improving upon previous exponential upper bounds, \citet{HOQRS17} proved that $\scol_{s}(G) \leqslant 5s + 1$ and $\wcol_{s}(G) \leqslant \binom{s + 2}{2}(2s + 1)$ for every planar graph $G$. Van den Heuvel and Wood~\citep{vdHW18} proved that $\scol_{s}(G) \leqslant \ltw(G)(2s + 1)$ for every graph $G$. \citet{KY03} showed that $\wcol_{s}(G) \leqslant (\scol_{s}(G))^{s}$ for every graph~$G$. Hence, we have the following immediate corollary of \cref{introLTW}:

\begin{thm} \label{introstrongcolouring}
\label{kMatchingPlanarScol}
Let $G$ be a topological $k$-matching-planar graph with no $t$ pairwise crossing edges incident to a common vertex. Then $\scol_{s}(G) \leqslant f(k, t)(2s + 1)$ and $\wcol_{s}(G) \leqslant g(k, t, s)$ for some functions $f$ and $g$.
\end{thm}

Strong and weak colouring numbers upper bound numerous graph parameters of interest, including acyclic chromatic number \citep{KY03}, game chromatic number \citep{KT94,KY03}, Ramsey numbers \citep{CS93}, oriented chromatic number \citep{KSZ-JGT97}, arrangeability \citep{CS93}, odd chromatic number \citep{Hickingbotham23a},  and conflict-free chromatic number \citep{Hickingbotham23a}. Thus, by \cref{kMatchingPlanarScol}, all these parameters are bounded for topological $k$-matching-planar graphs with no $t$ pairwise crossing edges incident to a common vertex.

\header{Asymptotic Dimension} 
Asymptotic dimension is a measure of the large-scale shape of a metric space. First introduced by \citet{Gro93} for the study of geometric groups, it has since been studied within structural graph theory; see \citep{BD08} for a survey on asymptotic dimension. \citet{BBEGLPS24} proved that the class of planar graphs has asymptotic dimension $2$. The proof uses their stronger result that for every integer $k \geqslant 1$, the asymptotic dimension of the class of graphs of layered treewidth at most $k$ is $2$. Thus, the following corollary of \cref{introLTW} generalises their above result for planar graphs.

\begin{thm} \label{introdimension} The class of topological $k$-matching-planar graphs with no $t$ pairwise crossing edges incident to a common vertex has asymptotic dimension $2$.
    
\end{thm}

\cref{universal,introqueue,intrononrepetitive,introcentred,introLTW,introseparators,introradius,introboxicity,introstrongcolouring,introdimension}  provide the broadest known criteria for a beyond planar graph class to have the respective structural property (see \cref{RTWbounded} for a detailed discussion).

\subsection{Proof Highlights} \label{ProofHighlights}

The proof of \cref{introRTW} introduces a number of results and techniques of independent interest that we now summarise.

\header{Edge Colouring} We prove that the edges of certain topological $k$-matching-planar graphs can be coloured using a bounded number of colours such that monochromatic edges do not cross. An \defn{edge-colouring} of a graph $G$ is a function $\phi : E(G) \rightarrow \mathcal{C}$ for some set $\mathcal{C}$ whose elements are called \defn{colours}. For a positive integer $c$, if $|\mathcal{C}| = c$ then $\phi$ is a \defn{$c$-edge-colouring}. An \defn{ordered $c$-edge-colouring} is an edge-colouring $\phi : E(G) \rightarrow \{1,\dots,c\}$. We say that an edge-colouring of a topological graph is \defn{transparent} if no two edges of the same colour cross. The \defn{topological thickness} of a topological graph $G$ is the minimum positive integer $c$ such that there exists a transparent $c$-edge-colouring of $G$. This definition is related to the notion of geometric thickness, introduced by \citet{DEH00}. The \defn{geometric thickness} of a graph $G$ is the minimum integer $k$ such that $G$ is isomorphic to a geometric graph $H$ with topological thickness at most $k$ (see \citep{BMW06,DujWoo07,Duncan11,Eppstein04,DEH00,DEK-SoCG04}). 

We prove that every simple topological $k$-matching-planar graph has topological thickness $\mathcal{O}(k^{3}\log k)$. In fact, we prove the following qualitatively stronger result without requiring simplicity.

\begin{thm} \label{introtopologicalthickness} 
Let $G$ be a topological $k$-matching-planar graph with no $t$ pairwise crossing edges incident to a common vertex. Then the topological thickness of $G$ is at most some function $f(k, t)$.
    
\end{thm}

The \defn{crossing graph} of a topological graph $G$ is the graph \defn{$X_{G}$} with vertex set $E(G)$, where distinct $e,f\in E(G)$ are adjacent in $X_G$ if and only if $e$ and $f$ cross in $G$. By definition, the topological thickness of a topological graph $G$ equals the chromatic number of $X_{G}$. A graph $G$ is \defn{$d$-degenerate} if every subgraph of $G$ has minimum degree at most $d$. A greedy algorithm shows that every $d$-degenerate graph is $(d + 1)$-colourable. The crossing graph of a topological $k$-planar graph has maximum degree at most $k$ and is thus $k$-degenerate. Hence, every topological $k$-planar graph has topological thickness at most $k + 1$. On the other hand, if $G$ is $k$-matching-planar, then $X_{G}$ can be dense. For example, if $G$ consists of two crossing stars, each with $n$ leaves (shown in \cref{crossingstars}(b)),
then $G$ is a topological 1-matching-planar graph, but $X_{G}$ is isomorphic to $K_{n,n}$, which is dense with unbounded degeneracy. So \cref{introtopologicalthickness} says that a certain class of dense graphs has bounded chromatic number, which therefore is an interesting and non-trivial result.

\header{Coloured Planarisations} Associated with every topological graph $G$ is the planarisation $G'$, which is obtained from $G$ by placing a `dummy' vertex at every crossing point (see \cref{topologicalgraphs}). 
The planarisation $G'$ can be useful in proving that a structural property of planar graphs also holds for topological graphs with few crossings per edge. For example, for every edge $uv$ of a topological $k$-planar graph $G$, $\dist_{G'}(u, v) \leqslant k + 1$. This is the starting point for the proof by \citet{DMW23} and \citet{HW24} that $k$-planar graphs have bounded row treewidth. However, this distance property ceases to be true for topological graphs with many crossings per edge, and this makes the standard planarisation method unsuitable for our purposes. To address this issue, we introduce the notion of a coloured planarisation. Given a topological graph $G$ and a transparent ordered $c$-edge-colouring $\phi$ of $G$, the coloured planarisation $G^{\phi}$ is obtained from $G'$ by contracting certain edges. This enables us to prove an analogous `distance property' for coloured planarisations of certain topological graphs (\cref{distancelemma}). Combining this with other properties, we show that $G$ inherits certain structural properties of the planar graph $G^{\phi}$.

We believe that coloured planarisations are of independent interest and might be applicable for other problems about topological graphs with an unbounded number of crossings per edge.

\header{Weak Shallow Minors} As mentioned above, building on the work of \citet{DMW23}, \citet{HW24} proved product structure theorems for several beyond planar graph classes. Their key observation is that several beyond planar graph classes can be described as a shallow minor of the strong product of a planar graph with a small complete graph. Shallow minors are fundamental to graph sparsity theory (see the book of \citet{Sparsity}). Extending this idea, we introduce the concept of weak shallow minors, which subsume and generalise shallow minors. Generalising a result of \citet{DMW17} about shallow minors, we show that layered treewidth is well-behaved under weak shallow minors (\cref{ltwWeakShallowMinor}). We prove that every weak shallow minor of the strong product of a graph with bounded Euler genus and a small complete graph has bounded row treewidth (\cref{thm:RTWmain}). Interestingly, this statement is false if the `bounded Euler genus' assumption is relaxed. In particular, we construct graphs with arbitrarily large row treewidth that are weak shallow minors of graphs with row treewidth $2$ (\cref{bigrtw}). We thus consider our methods to be pushing the boundaries of graph product structure theory. We believe that the concept of weak shallow minors is of independent interest in graph sparsity theory.

\header{Proof Sketch} Here is a brief sketch of the proofs of \cref{introRTW,introLTW}. Let $G$ be a topological $k$-matching-planar graph with no $t$ pairwise crossing edges incident to a common vertex. By \cref{introtopologicalthickness}, there exists an ordered $c$-edge-colouring $\phi$ of $G$ for some integer $c$ bounded by a function of $k$ and $t$. We use some properties of coloured planarisations to establish that $G$ is a weak shallow minor of $G^{\phi} \boxtimes K_{\ell}$ for some small $\ell$. Using our above-mentioned results about the behaviour of row treewidth and layered treewidth under weak shallow minors, we establish the desired upper bounds on $\rtw(G)$ and $\ltw(G)$.

\subsection{Treewidth and Circular Graphs} \label{introtreewidthsection}

As a byproduct of our proof techniques, we prove upper bounds on the treewidth of circular graphs that are more general than the existing results. Here, a \defn{circular graph} is a geometric graph with its vertices positioned on a circle. Circular graphs (also known as \defn{circular} or \defn{convex} drawings of graphs) are well studied in the literature. For example, there is large literature on the \defn{book thickness} of a graph $G$ (also called \defn{page-number} or \defn{stack-number}), which is equivalent to the minimum integer $k$ such that $G$ is isomorphic to a circular graph with topological thickness $k$; see \citep{Eppstein04,Yann89,Yann20,BKKPRU20,Malitz94a,Malitz94b,DujWoo07,BBKR17}.

If a graph has a well-behaved circular drawing, must the structure of the graph be well-behaved? Circular graphs with no crossings are exactly the outerplanar graphs, which have treewidth at most $2$. Circular $k$-planar graphs (also known as \defn{outer $k$-planar} drawings) were first studied by \citet{WT07}, who proved that the treewidth of every circular $k$-planar graph is at most $3k + 11$; this bound was further improved to $\frac{3}{2}k + 2$ by \citet{FGKO024}. We prove the following result for circular $k$-matching-planar graphs, which is a qualitative generalisation of the above results. 

\begin{thm} \label{introcircularkcoverplanar} 
Every circular $k$-matching-planar graph has treewidth $\mathcal{O}(k^3\log^{2}k)$.
\end{thm}

A topological graph is \defn{min-$k$-planar}~\citep{BBBDD00LM24} if for any crossing edges $e$ and $f$, at least one of $e$ or $f$ is involved in at most $k$ crossings. Circular min-$k$-planar graphs are also known as \defn{outer min-$k$-planar} drawings. \citet{WT07} actually proved the following result, which is stronger than their above result for circular $k$-planar graphs.

\begin{thm} [\citep{WT07}] \label{knownresult} 
Every circular min-$k$-planar graph has treewidth at most $3k + 11$.

\end{thm}

\citet{FGKO024} slightly improved this upper bound from $3k + 11$ to $3k + 1$ for $k \geqslant 1$, which was further improved to $3\lfloor \frac{k}{2} \rfloor + 4$ by \citet{Pyzik25}.

We prove the following strengthening of \cref{introcircularkcoverplanar}, which also qualitatively generalises \cref{knownresult} (see \cref{result} for a detailed discussion).

\begin{thm} \label{introcirculardrawingscoloured} 
Let $G$ be a circular graph with a transparent ordered $c$-edge-colouring. Suppose that for any $i, j \in\{1, \dots, c\}$ with $i < j$ and for any edge $e$ of colour $i$, the matching number of the set of edges of colour $j$ that cross $e$ is at most $m$. Then $\tw(G) \leqslant 9mc(c - 1) + 3c - 1$.
\end{thm}

Note that \cref{introcircularkcoverplanar,introcirculardrawingscoloured} allow an unbounded number of crossings on every edge (see 
\cref{crossingstars}(b)), unlike the previous known results mentioned above. In particular, every circular $k$-planar graph or circular min-$k$-planar graph has edges that are involved in at most $k$ crossings. There is another relevant result in this direction due to \citet{HIMW24}, who proved that a circular graph $G$ satisfies $\tw(G) \leqslant 12t - 23$ if the crossing graph $X_{G}$ is $K_{t}$-minor-free and $t \geqslant 3$. Since $K_{t}$-minor-free graphs are $\mathcal{O}(t\sqrt{\log t})$-degenerate~\citep{Kostochka82,Kostochka84,Thomason84}, there must be edges involved in a bounded number of crossings in such graphs $G$.

We also prove a result related to \cref{introradius} that bounds the treewidth of (not necessarily circular) topological graphs with bounded radius and generalises \cref{PlanarBoundedRadius,introcircularkcoverplanar,introcirculardrawingscoloured} (see \cref{spanningtree}).

\subsection{Outline}

The remainder of the paper is organised as follows. In \cref{preliminaries}, we give basic definitions and review relevant background, including treewidth, layered treewidth, graph products, and topological graphs. We also introduce $k$-cover-planar graphs, which are closely related to $k$-matching-planar graphs. We provide a detailed overview of existing beyond planar graph classes and their relationships to $k$-matching-planar graphs. In \cref{tools}, we define the coloured planarisation and analyse its basic properties. We prove the so-called `Coloured Planarisation Lemma' and `Distance Lemma', which are used in the proofs of our main results providing upper bounds on row treewidth, layered treewidth, and treewidth. In \cref{result}, we prove our results upper-bounding the treewidth of certain beyond planar graphs. In \cref{SectionColouring}, we analyse edge-colourings of topological $k$-matching-planar graphs and prove \cref{introtopologicalthickness}. In \cref{WeakShallowMinorsSection}, we introduce the concept of weak shallow minors. We analyse how row treewidth and layered treewidth behave under weak shallow minors. We prove our main result (\cref{thm:RTWmain}) of this section, which says that every weak shallow minor of the strong product of a graph with bounded genus and a small complete graph admits a product structure theorem. \cref{FinalSection} combines the above material to finish the main proofs. 
In particular, we apply the Coloured Planarisation Lemma, the Distance Lemma and the results of \cref{SectionColouring} to show that certain topological $k$-matching-planar graphs are weak shallow minors of the strong product of a planar graph with a
small complete graph. We use this result and some results of \cref{WeakShallowMinorsSection} to prove \cref{introRTW,introLTW}. Finally, \cref{OpenProblems} concludes with open problems.

\section{Preliminaries} \label{preliminaries}

\subsection{Graph Basics} \label{graphbasics}

We consider simple finite undirected graphs $G$ with vertex set $V(G)$ and edge set $E(G)$. For any undefined graph-theoretic terminology, see \citep{Diestel5}.

A \defn{class} of graphs is a family of graphs that is closed under isomorphism.

The \defn{radius} of a connected graph $G$ is the minimum non-negative integer $r$ such that for some vertex $v \in V(G)$ and for every vertex $w \in V(G)$ we have $\dist_{G}(v, w) \leqslant r$. 

A \defn{matching} is a set of pairwise disjoint edges in a graph. Let $E \subseteq E(G)$ be a set of edges of a graph $G$. The \defn{matching number} of $E$, denoted \defn{$\mu(E)$}, is the size of a largest matching in $G$ that consists of edges in $E$.  A \defn{vertex cover} of $E$ is a set $U \subseteq V(G)$ of vertices such that every edge of $E$ is incident to $U$. The \defn{vertex cover number} of $E$, denoted \defn{$\tau(E)$}, is the minimum size of a vertex cover of $E$. It is folklore that:
\begin{equation} \label{folklore} \tag{1}
\mu(E) \leqslant \tau(E) \leqslant 2\mu(E).
\end{equation}

Let $G$ be a graph. For a set of vertices $V_{1} \subseteq V(G)$, the subgraph of $G$ \defn{induced by $V_{1}$}, denoted \defn{$G[V_{1}]$}, has vertex set $V_{1}$ and its edge set is the set of edges of $G$ with both endpoints in $V_{1}$.  
For a set of edges $E_{1} \subseteq E(G)$, the subgraph of $G$ \defn{induced by $E_{1}$} has edge set $E_1$ and vertex set the set of endpoints of edges in $E_{1}$.

Let $G$ be a graph and $t \geqslant 1$ be an integer. The \defn{$t$-th power} of $G$, denoted \defn{$G^t$}, is the graph with $V(G^t) := V(G)$ and $uv \in E(G^t)$ if and only if $\dist_{G}(u, v) \leqslant t$ and $u \neq v$.

For graphs $G$ and $H$, we say that $G$ is \defn{$H$-free} if $H$ is not isomorphic to a subgraph of $G$.

Let $G$ be a graph. We denote the chromatic number
of $G$ by \defn{$\chi(G)$}, and its clique number (the cardinality of its largest clique) by \defn{$\omega(G)$}. A class of graphs $\mathcal{G}$ is \defn{$\chi$-bounded} if there is a function $f : \mathbb{N} \to \mathbb{N}$ such that every graph $G \in \mathcal{G}$ satisfies $\chi(G) \leqslant f(\omega(G))$.

A \defn{walk} in a graph $G$ is a sequence
$(v_{1}, v_{2},\dots, v_{t})$ of vertices in $G$ such that $v_{i}v_{i + 1} \in E(G)$ for each $i \in \{1,\dots, t - 1\}$. A \defn{path} in a graph $G$ is a walk $(v_{1}, v_{2},\dots, v_{t})$ in $G$ such that $v_{i} \neq v_{j}$ for all distinct $i, j \in \{1,\dots,t\}$. Let $W = (v_{1}, v_{2}, \dots, v_{t})$ be a walk in a graph $G$. We say that $v_{1}$ and $v_{t}$ are the \defn{endpoints} of $W$. For any $i \in \{1,\dots,i - 1\}$, $v_{i}$ and $v_{i + 1}$ are \defn{consecutive} vertices in $W$.

A graph $S$ is a \defn{star} if it is isomorphic to $K_{1}$ or $K_{1, t}$ for some $t \geqslant 1$.  If $S$ is isomorphic to $K_{1}$ or $K_{1, 1}$, then a \defn{centre} of $S$ is an arbitrary vertex of $S$. If $S$ is isomorphic to $K_{1, t}$ for some $t \geqslant 2$, then the \defn{centre} of $S$ is the vertex of $S$ with degree $t$. A graph $G$ is a \defn{star-forest} if $G$ is a forest where every connected component is a star.

The \defn{arboricity} of a graph $G$ is the minimum number of edge-disjoint forests whose union is $G$. The \defn{star arboricity} of a graph $G$, denoted  \defn{$\st(G)$}, is the minimum number of edge-disjoint star-forests whose union is $G$.

The \defn{Euler genus} of a surface with $h$ handles and $c$ cross-caps is $2h+c$. The \defn{Euler genus} of a graph $G$ is the minimum Euler genus of a surface in which $G$ embeds without crossings.

\subsection{Treewidth, Layered Treewidth, Minors, and Graph Products}
\label{LTWsubsection}

For a graph $G$, a \defn{tree decomposition} is a pair $(T, B)$ such that:

\begin{itemize}
    \item $T$ is a tree and $B : V(T) \rightarrow 2^{V(G)}$ is a function,
    \item for every edge $vw \in E(G)$, there exists a node $t \in V(T)$ with $v, w \in B(t)$, and
    \item for every vertex $v \in V(G)$, the subgraph of $T$ induced by $\{t \in V(T) : v \in B(t)\}$ is a non-empty (connected) subtree of $T$.
\end{itemize}

The sets $B(t)$ where $t \in V(T)$ are called \defn{bags} of $(T, B)$. The \defn{width} of a tree decomposition $(T, B)$ is $\max\{|B(t)| : t \in V(T)\} - 1$. The \defn{treewidth} of $G$, denoted \defn{$\tw(G)$}, is the minimum width of a tree decomposition of $G$. Tree decompositions were introduced by \citet{RS-II}. Graphs of bounded treewidth are considered to be ‘easy’ and many problems can be solved for graphs of bounded treewidth. Numerous algorithmic problems can be solved in linear time on any graph class with bounded treewidth \citep{Courcelle90}.

A \defn{vertex-partition}, or simply \defn{partition}, of a graph $G$ is a set $\mathcal{P}$ of non-empty sets (called \defn{parts}) of vertices in $G$ such that each vertex of $G$ is in exactly one element of $\mathcal{P}$. A \defn{layering} of a graph $G$ is a partition $(V_{0}, V_{1},\dots, V_{s})$ of $G$ such that for every edge $vw \in E(G)$, if $v \in V_{i}$ and $w \in V_{j}$, then $|i - j| \leqslant 1$. Each set $V_{i}$ is called a \defn{layer}. The \defn{layered width} of a tree decomposition $(T, B)$ of a graph $G$ is the minimum integer $\ell$ such that, for some layering $(V_{0}, V_{1},\dots, V_{s})$ of $G$, each bag $B(t)$ contains at most $\ell$ vertices in each layer $V_{i}$. The \defn{layered treewidth} of a graph $G$ is the minimum layered width of a tree decomposition of $G$.

We now compare layered treewidth to row treewidth. \cref{LTWandRTW} says that $\ltw(G) \leqslant \rtw(G) + 1$ for every graph $G$. On the other hand, \citet{BDJMW22} showed that row treewidth cannot be upper bounded by any function of layered treewidth.

\begin{thm} [\citep{BDJMW22}] \label{LTWandRTWcomparison} For every integer $n \geqslant 1$, there is a graph with layered treewidth $1$ and row treewidth at least $n$.
\end{thm}

This says that row treewidth is a qualitatively stronger parameter than layered treewidth. Indeed, for many applications, row treewidth gives qualitatively stronger results than layered treewidth. For example, graphs of bounded row treewidth have bounded queue number~\citep{DJMMUW20}, but it is open whether graphs of layered treewidth 1 have bounded queue number~\citep{BDJMW22}.

Let $G$ and $H$ be graphs. $G$ is a \defn{minor} of $H$ if a graph isomorphic to $G$ can be obtained from $H$
by vertex deletion, edge deletion, and edge contraction. A \defn{model} of $G$ in $H$ is a function $\mu : V(G) \rightarrow 2^{V(H)}$ such that:
\begin{itemize}
    \item for each $v \in V(G)$, $\mu(v)$ is non-empty and the subgraph of $H$ induced by $\mu(v)$ is connected;

    \item $\mu(v) \cap \mu(w) = \emptyset$ for all distinct $v, w \in V(G)$; and

    \item for every edge $vw \in E(G)$, $ab \in E(H)$ for some $a \in \mu(v)$ and $b \in \mu(w)$.

\end{itemize}

The sets $\mu(v)$ are called \defn{branch sets} of $\mu$. It is folklore that $G$ is a minor of $H$ if and only if there exists a model of $G$ in $H$. It is well-known that if $G$ is a minor of $H$ then $\tw(G) \leqslant \tw(H)$ (see \citep{Bodlaender98} for an implicit proof).

The \defn{cartesian product} of graphs $G_{1}$ and $G_{2}$ is the graph $G_{1} \square G_{2}$ with vertex set $V(G_{1} \square G_{2}) := \{(a, v) : a \in V(G_{1}), v \in V(G_{2})\}$, where distinct vertices $(a, v)$ and $(b, u)$ are adjacent if: $ab \in E(G_{1})$ and $v = u$; or $a = b$ and $uv \in E(G_{2})$. The \defn{strong product} of graphs $G_{1}$ and $G_{2}$ is the graph $G_{1} \boxtimes G_{2}$ with vertex set $V(G_{1} \boxtimes G_{2}) := \{(a, v) : a \in V(G_{1}), v \in V(G_{2})\}$, where distinct vertices $(a, v)$ and $(b, u)$ are adjacent if: $ab \in E(G_{1})$ and $v = u$; or $a = b$ and $uv \in E(G_{2})$; or $ab \in E(G_{1})$ and $uv \in E(G_{2})$. We frequently make use of the well-known fact that $\tw(G \boxtimes K_{n}) \leqslant (\tw(G) + 1)n - 1$ for every graph $G$ and integer $n \geqslant 1$.

\subsection{Topological Graphs} \label{topologicalgraphs}

A \defn{topological graph} $G$ is a graph whose vertices are distinct points in the plane, where each edge $vw$ of $G$ is a non-self-intersecting curve between $v$ and $w$, such that:

\begin{itemize}

\item no edge passes through any vertex different from its endpoints,

\item  each pair of edges intersect at a finite number of points,

\item  no three edges internally intersect at a common point.

\end{itemize}

The language of `topological graph' \citep{AFPS14,PRT06,PT10a,Pach14,KPRT15,PST03a} and `geometric graph' \citep{TV-DCG99,Toth-JCTA00,Valtr-DCG98,Pinchasi08,AE-DCG89,Pach-BritishSurvey} is well-used in the literature.

A \defn{crossing} (or \defn{crossing point}) of distinct edges $e$ and $f$ in a topological graph is an internal intersection point of $e$ and $f$. A topological graph with no crossings is \defn{planar}. A topological graph $G$ is \defn{outerplanar} if $G$ is planar and 
every vertex of $G$ is on the outerface. A graph is \defn{planar} or \defn{outerplanar} if it is respectively isomorphic to a topological planar or outerplanar graph.

The \defn{planarisation} of a topological graph $G$, denoted \defn{$G'$}, is the topological planar graph obtained from $G$ by replacing each crossing with a `dummy' vertex of degree $4$.

\subsection{$k$-Cover-Planar Graphs}

 We now introduce a class of beyond planar graphs, so-called $k$-cover-planar graphs, and discuss their relationship with $k$-matching-planar graphs. For an integer $k \geqslant 0$, a topological graph $G$ is \defn{$k$-cover-planar} if for every edge $e\in E(G)$, the vertex cover number of the set of edges of $G$ that cross $e$ is at most $k$. A graph is \defn{$k$-cover-planar} if it is isomorphic to a topological $k$-cover-planar graph. This definition is closely related to $k$-matching-planar graphs, as shown by the following direct corollary of \cref{folklore}.

\begin{obs} \label{relation} Every (topological) $k$-cover-planar graph is a (topological) $k$-matching-planar graph. Every (topological) $k$-matching-planar graph is a (topological) $2k$-cover-planar graph. 
\end{obs}

By \cref{relation}, the results of \cref{introRTW,universal,introqueue,intrononrepetitive,introcentred,introLTW,introseparators,introtopologicalthickness,introcircularkcoverplanar,introradius,introboxicity,introstrongcolouring,introdimension} 
also hold for $k$-cover-planar graphs.

Every $k$-planar graph is $k$-cover-planar, but not vice versa. For example, as shown in \cref{K3n}(a), the complete bipartite graph  $K_{3,n}$ is $1$-cover-planar but not $k$-planar for sufficiently large $n$.  More generally, $K_{2k + 2, n}$ is $k$-cover-planar for all $k \geqslant 0$ and $n \geqslant 1$.

We present our main results in the language of $k$-matching-planar graphs since the `matching-planar' definition is more natural, and $k$-cover-planar graphs cannot be described by a single forbidden crossing configuration. Moreover, the vertex cover problem is NP-hard, whereas maximum matchings can be computed in polynomial time. Thus, one can determine in polynomial time whether a topological graph is $k$-matching-planar, unlike recognising $k$-cover-planarity.

Our main  motivation for introducing the concept of cover-planar graphs is the convenience of representing matching-planar graphs as cover-planar graphs using \cref{relation} in the proof of \cref{introtopologicalthickness}. 

\subsection{Related Beyond Planar Graphs} \label{sectionbackground}

We now give an overview of related beyond planar graph classes and discuss their relationships to $k$-matching-planar graphs. First, a simple topological graph $G$ is \defn{fan-planar}~\citep{KU22} if for each edge $e \in E(G)$, all the edges that cross $e$ are incident to a common vertex and no endpoint of $e$ is enclosed by $e$ and the edges that cross $e$. Equivalently, this can be formulated by forbidding three configurations (I, II, III) in \cref{fanplanar}, one of which is the configuration where $e$ is crossed by two edges not incident to a common vertex and the other two where $e$ is crossed by two edges incident to a common vertex in a way that encloses some endpoint of $e$. Note that for simple topological graphs, configurations II and III are well-defined.

\begin{figure}[h]
    \centering
    \begin{subfigure}[t]{0.24\textwidth}
    \centering
        \scalebox{0.8}{\includegraphics{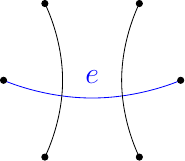}}
        \subcaption*{Configuration I}
    \end{subfigure}
    \begin{subfigure}[t]{0.24\textwidth}
    \centering
        \scalebox{0.8}{\includegraphics{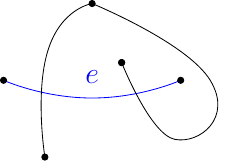}}
        \subcaption*{Configuration II}
    \end{subfigure}
        \begin{subfigure}[t]{0.24\textwidth}
        \centering
        \scalebox{0.8}{\includegraphics{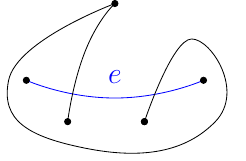}}
        \subcaption*{Configuration III}
    \end{subfigure}
            \begin{subfigure}[t]{0.24\textwidth}
        \centering
        \scalebox{0.8}{\includegraphics{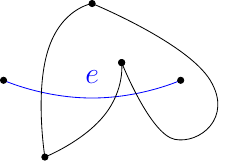}}
        \subcaption*{Configuration IV}
    \end{subfigure}
    \caption{Forbidden crossing configurations. Configuration I: $e$ is crossed by two edges that are not incident to a common vertex. Configuration II: $e$ is crossed by two edges that cross $e$ from different sides when directed away from a common endpoint. Configuration III: both endpoints of $e$ are in the bounded region determined by $e$ and two edges that cross $e$ and are incident to a common vertex. Configuration IV: $e$ is crossed by the edges of a triangle.}
    \label{fanplanar}
\end{figure}

Fan-planar graphs were introduced by \citet{KU22}. In their initial preprint~\citep{KU14}, only configurations I and II were forbidden. \citet{KKRS23} pointed out a missing case in the proof of the edge density upper-bound in~\citep{KU14}. This case was consequently fixed in the journal version \citep{KU22} by introducing forbidden configuration III in the definition of fan-planar graphs. \citet{CFKPS23} distinguish the case, where only configurations I and II are forbidden, and call the corresponding simple topological graphs \defn{weakly fan-planar} (see also \citep{CPS22}). They constructed a topological weakly fan-planar graph that is not isomorphic to a topological fan-planar graph, and hence configuration III is essential for the definition of fan-planar graphs.

\citet{Brandenburg20} considered the following extensions of fan-planar graphs. A topological graph is \defn{fan-crossing} if it is simple and does not allow configurations I and IV in \cref{fanplanar}. A topological graph is \defn{adjacency-crossing} if it is simple and does not allow configuration I. \citet{Brandenburg20} proved that every adjacency-crossing graph is isomorphic to a fan-crossing graph, and hence configuration IV is not necessary for the definition of fan-crossing graphs. He also proved that there exist fan-crossing graphs that are not isomorphic to a weakly fan-planar graph, and hence configuration II is essential for the definition of weakly fan-planar graphs.

Simple topological $1$-matching-planar graphs are exactly adjacency-crossing graphs, and simple topological $1$-cover-planar graphs are exactly fan-crossing graphs. Every fan-planar, weakly fan-planar, or fan-crossing graph is $1$-cover-planar and $1$-matching-planar. \citet*{CPS22} proved that every simple topological fan-planar graph has topological thickness at most $3$. \cref{introtopologicalthickness} generalises this result.

Most of the literature concerning fan-planar, weakly fan-planar, fan-crossing, and adjacency-crossing graphs considers simple topological graphs. A notable exception is the work of \citet{KKRS23}, who extended the definition of topological fan-planar graphs to the non-simple setting and proved the following result.

\begin{thm} [\citep{KKRS23}] \label{SimpleFan} Every non-simple topological fan-planar graph is isomorphic to a simple topological fan-planar graph.

\end{thm}

 We do not restrict ourselves to the simple case and analyse topological graphs that can be non-simple.

There are several extensions of $k$-planar graphs in the literature, notably $k$-gap-planar graphs~\citep{GapPlanar18}, min-$k$-planar graphs \citep{BBBDD00LM24}, $k$-quasi-planar graphs \citep{AAPPS-Comb97,FPS13,PRT06,SW13,Suk12}, and $k$-fan-bundle-planar graphs \citep{ABKKS18}. A topological graph is \defn{$k$-gap-planar} if every crossing can be charged to one of the two edges involved so that at most $k$ crossings are charged to each edge. Recall that a topological graph is \defn{min-$k$-planar} if for any crossing edges $e$ and $f$, at least one of $e$ or $f$ is involved in at most $k$ crossings. A topological graph is \defn{$k$-quasi-planar} if no $k$ edges pairwise cross. A graph is \defn{$k$-gap-planar}, \defn{min-$k$-planar}, or \defn{$k$-quasi-planar} if it is isomorphic to a topological $k$-gap-planar, a topological min-$k$-planar, or a topological $k$-quasi-planar graph, respectively. Every min-$k$-planar graph is $k$-gap-planar \citep{BBBDD00LM24}.

Consider the relationship between matching-planar graphs and quasi-planar graphs. By definition, topological $k$-matching-planar graphs can have an unbounded number of pairwise crossing edges, if they are incident to a common vertex. Hence, there exists no function $f$ such that every topological $k$-matching-planar graph is topological $f(k)$-quasi-planar. However, it is easily seen by \cref{relation} that topological $k$-matching-planar graphs with no $t$ pairwise crossing edges incident to a common vertex are topological $(2kt + 2)$-quasi-planar.

The class of $k$-fan-bundle-planar graphs was introduced by \citet{ABKKS18}. They studied edge density and algorithmic properties of $1$-fan-bundle-planar graphs. In \cref{subsectionbundle}, we give the definition of $k$-fan-bundle-planar graphs, show that every $k$-fan-bundle-planar graph is $2k$-matching-planar (see \cref{relationbundlematching}), and prove that for any fixed $k$ there are $1$-matching-planar graphs that are not $k$-fan-bundle-planar. Thus, the class of $k$-matching-planar graphs is a significant generalisation
of the class of $\lfloor \frac{k}{2} \rfloor$-fan-bundle-planar graphs.

We now compare $k$-matching planar graphs with the graph classes introduced by \citet*{AFPS14}. They defined a \defn{$(k,\ell)$-grid} in a topological graph $G$ to be a pair $(E_1,E_2)$ where $E_1, E_2 \subseteq E(G)$ and $|E_1| = k$ and $|E_2| = \ell$ and every edge in  $E_1$ crosses every edge in $E_2$. They considered the class of topological graphs with no $(k, \ell)$-grid. Let $G$ be a topological graph with no $(k,1)$-grid. Each edge of $G$ is crossed by at most $k$ other edges. 
Note that $G$ is $k$-cover-planar (the discussion after Lemma~4.1 in \citep{AFPS14} shows it is $2k$-cover-planar). \cref{introRTW,introLTW} imply that $G$ has bounded row treewidth and layered treewidth. On the other hand, the example in \cref{crossingstars}(b) is $1$-matching-planar but contains an $(n,n)$-grid. So in this sense, topological $k$-matching-planar graphs are more general than topological graphs with no $(k,1)$-grid. 

\citet{AFPS14} also considered $(k,\ell)$-grids $(E_1,E_2)$ `with distinct vertices', meaning 
that no two edges of $E_{1} \cup E_{2}$ are incident to a common vertex. The only difference between topological graphs that contain no $(k+1, 1)$-grid with distinct vertices and $k$-matching-planar graphs is that
the former may have an edge $e$ that is crossed by a matching of size $k+1$ provided that some edge of the matching shares an endpoint with $e$. Thus topological graphs with no $(k, 1)$-grid with distinct vertices are sandwiched between $(k - 1)$-matching-planar and $(k + 1)$-matching-planar graphs, and correspond exactly to $k$-matching-planar for simple topological graphs. \citet[Theorem~1.7]{AFPS14} proved a bound on the edge density of topological graphs with no $(k, 1)$-grid with distinct vertices; see \cref{edgeskmatchingplanar}.

\subsection{When is Row Treewidth Bounded?}
\label{RTWbounded}

The following question naturally arises: What is the most general known beyond planar graph class that has bounded layered treewidth or bounded row treewidth?

\citet{DEW17} showed that every $k$-planar graph has layered treewidth at most $6(k+1)$. Building on the work of \citet{DMW23}, \citet{HW24} proved that every $k$-planar graph has row treewidth at most $6(k + 1)^{2}\binom{k + 4}{3} - 1$, every fan-planar graph\footnote{Note that the proof of \citet{HW24} includes a non-trivial planarisation for fan-planar graphs that, like the coloured planarisation, addresses the issue of  some edges having many crossings.} has layered treewidth at most 45 and 
row treewidth at most $1619$, and every $k$-fan-bundle-planar graph has layered treewidth at most $24k+25$ and row treewidth at most $\binom{2k + 6}{3}6(2k + 3)^{2} - 1$. As explained above, the class of $k$-matching-planar graphs extends $k$-planar graphs, fan-planar graphs and $\lfloor \frac{k}{2} \rfloor$-fan-bundle-planar graphs.  
Indeed, every result in the literature bounding the row treewidth of a class of beyond planar graphs is subsumed by \cref{introRTW} for $k$-matching planar graphs (since the number of pairwise crossing edges incident to a common vertex can be bounded for fan-planar graphs by \cref{SimpleFan}, and for $k$-fan-bundle-planar graphs by \cref{relationbundlematching}).

On the other hand, some notable beyond planar graph classes have unbounded layered treewidth and unbounded row treewidth. In particular, \citet[Proposition~21]{HIMW24} constructed simple topological graphs whose crossing graph is a star-forest, with radius $1$ and arbitrarily large treewidth. Since the crossing graph is a star-forest, these graphs are $1$-gap-planar, min-$1$-planar, and have no ($2,2)$-grid (with or without distinct vertices). Hence, the class of simple topological $1$-gap-planar graphs has unbounded local treewidth, and therefore has unbounded layered treewidth and unbounded row treewidth (by \cref{LTWandRTW}). 
The same holds for  simple topological  min-$1$-planar graphs and simple topological graphs with no $(2,2)$-grid. This says that for $k,\ell\geq 2$, the class of topological graphs with no $(k,\ell)$-grid are broader than the class of $k$-matching-planar graphs.  So our main theorems (\cref{introRTW,introLTW}) cannot be generalised via excluded grids.

Quasi-planar graphs also have arbitrarily large layered treewidth and row treewidth. As explained by \citet[page~$5$]{DSW16},  there is an infinite family of bipartite expander graphs with geometric thickness $2$. By definition, every graph with geometric thickness $2$ is isomorphic to a geometric $3$-quasi-planar graph. Every $n$-vertex expander graph $G$ has treewidth $\Omega(n)$ (see \citep{GM09}). Since $\tw(G)\leqslant 2\sqrt{\ltw(G)n} - 1$ \citep[Lemma~10]{DMW17}, it follows that $\ltw(G) \in \Omega(n)$ also. So the class of geometric  $3$-quasi-planar $n$-vertex graphs has layered treewidth $\Omega(n)$ and row treewidth $\Omega(n)$ (by \cref{LTWandRTW})\footnote{Moreover, the class of quasi-planar graphs fails to have any of the applications listed in \cref{applications}. The key example is the 1-subdivision of $K_n$, which has geometric thickness 2 \citep{Eppstein04} and is thus 3-quasi-planar. On the other hand, the 1-subdivision of $K_n$ has boxicity  $\Theta(\log \log n)$ \citep{EJ13}, $\Omega(\sqrt{n})$ queue-number~\citep{DujWoo05}, and $\Omega(\sqrt{n})$ nonrepetitive chromatic number~\citep{Wood21}. Similarly, the class of graphs obtained from complete graphs by subdividing each edge at least once (which has geometric thickness 2 \citep{Eppstein04}) has unbounded asymptotic dimension.}.

All this is to say that the class of $k$-matching-planar graphs is a good candidate for the answer to the question at the start of \cref{RTWbounded} (and this remains true for simple topological graph classes).

\subsection{$k$-Fan-Bundle-Planar Graphs} \label{subsectionbundle}

We now define the class of $k$-fan-bundle-planar graphs, and show that it is subsumed by the class of $2k$-matching-planar graphs.

A \defn{fan-bundling} of a graph $G$ is an indexed  set $\mathcal{E}=(\mathcal{E}_v:v\in V(G))$ where $\mathcal{E}_v$ is a partition of the set of edges in $G$ incident to $v$. For each $v \in V(G)$, each element of $\mathcal{E}_v$ is called a \defn{fan-bundle}. For a fan-bundling $\mathcal{E}$ of $G$, let $G_\mathcal{E}$ be the graph with $V(G_\mathcal{E}):= V(G) \cup\{z_{B,v}:B\in \mathcal{E}_v,v\in V(G)\}$ and
$E(G_\mathcal{E}):= \{v z_{B,v}:B\in \mathcal{E}_v,v\in V(G)\}
\cup\{ z_{B_1,v}z_{B_2,w}:vw\in E(G),vw\in B_1\in\mathcal{E}_v,vw\in B_2\in\mathcal{E}_w\}$.
Here 
$V(G) \cap\{z_{B,v}:B\in \mathcal{E}_v,v\in V(G)\}=\emptyset$. 
For an integer $k \geqslant 0$, a graph $G$ is \defn{$k$-fan-bundle-planar} if for some fan-bundling $\mathcal{E}$ of $G$, the graph $G_\mathcal{E}$ is (isomorphic to) a topological graph such that each edge $z_{B_1,v}z_{B_2,w}\in E(G_\mathcal{E})$ is in no crossings, and each edge $vz_{B,v}\in E(G_\mathcal{E})$ is in at most $k$ crossings.

\begin{prop} \label{relationbundlematching} 
Every $k$-fan-bundle-planar graph is isomorphic to a topological $2k$-matching-planar graph such that no $2k + 2$ edges incident to a common vertex pairwise cross and any two edges have at most $2k$ crossing points in common.
\end{prop}

\begin{proof} Consider a $k$-fan-bundle-planar graph $G$. For the sake of convenience, we assume that the graph $G_{\mathcal{E}}$ is topological.

Let $\varepsilon > \delta > 0$ be real numbers. For each $v \in V(G_{\mathcal{E}})$, let $S_{v}^{\varepsilon} := \{p \in \mathbb{R}^{2} : \dist_{\mathbb{R}^{2}}(p, v) \leqslant \varepsilon\}$. For each vertex $w \in V(G)$ and fan-bundle $B \in \mathcal{E}_{w}$, $wz_{B,w}$ is an edge of $G_{\mathcal{E}}$ and a curve in the plane. Let $C_{B,w}^{\delta,\varepsilon} := \{p \in \mathbb{R}^{2} : \dist_{\mathbb{R}^{2}}(p, wz_{B,w}) \leqslant \delta\} \setminus (S_{w}^{\varepsilon} \cup S_{z_{B,w}}^{\varepsilon})$.

Choosing $\varepsilon$ and $\delta$ to be sufficiently small, we may assume that:

\begin{enumerate}
    \item for every edge $e = xy \in E(G_{\mathcal{E}})$, $e$ has exactly one intersection point with the boundary of $S_{v}^{\varepsilon}$ for each $v \in \{x, y\}$,

    \item $S_{v_{1}}^{\varepsilon} \cap S_{v_{2}}^{\varepsilon} = \emptyset$ for distinct vertices $v_{1}, v_{2} \in V(G_{\mathcal{E}})$,
    \item $S_{v}^{\varepsilon} \cap C_{B,w}^{\delta,\varepsilon} = \emptyset$ for each $v,w \in V(G_{\mathcal{E}})$ and fan-bundle $B \in \mathcal{E}_{w}$,
    \item $C_{B_{1}, v}^{\delta,\varepsilon} \cap C_{B_{2}, w}^{\delta,\varepsilon} = \emptyset$ for every pair of non-crossing edges $vz_{B_{1},v}$ and $wz_{B_{2},w}$ in $G_{\mathcal{E}}$.
\end{enumerate}

Consider an edge $e=uv$ of $G$. Say $e\in B_u\in\mathcal{E}_u$ and 
$e\in B_v\in\mathcal{E}_v$. So $uz_{B_u,u}, z_{B_u,u}z_{B_v,v}, z_{B_v,v}v \in E(G_{\mathcal{E}})$. For each $x \in \{u, v\}$, let $p_{e,x}$ be the intersection point of the edge $z_{B_u,u}z_{B_v,v}$ in $G_{\mathcal{E}}$ and the boundary of $S_{z_{B_x,x}}^{\varepsilon}$ given by property $1$ above. Draw a non-self-intersecting curve $\gamma_{e,x}$ between $x$ and $p_{e,x}$ in $S_{x}^{\varepsilon} \cup C_{B_x,x}^{\delta,\varepsilon} \cup S_{z_{B_x,x}}^{\varepsilon}$. Do this for every edge of $G$ such that for every $w \in V(G)$ and any two edges $e_{1}, e_{2} \in E(G)$ incident to $w$ that belong to the same fan-bundle, the curves $\gamma_{e_{1}, w}, \gamma_{e_{2}, w}$ do not intersect, except at $w$.

For each edge $e = uv \in E(G)$ with $e\in B_u\in\mathcal{E}_u$ and $e\in B_v\in\mathcal{E}_v$, the curves $\gamma_{e,u}$ and $\gamma_{e,v}$ and the subcurve of the edge $z_{B_u,u}z_{B_v,v}$ between $p_{e,u}$ and $p_{e,v}$ together form a curve $\gamma_{e}$ between $u$ and $v$. Note that $\gamma_{e}$ can be self-intersecting. This can happen if $uz_{B_{u},u}$ crosses $vz_{B_{v},v}$. Let $\gamma_{e}'$ be a non-self-intersecting curve with endpoints $u$ and $v$ in the region $\{d \in \mathbb{R}^{2} : \dist_{\mathbb{R}^{2}}(d, \gamma_{e}) \leqslant \delta_{1}\}$ for some sufficiently small $0 < \delta_{1} < \delta$ (if $\gamma_{e}$ is non-self-intersecting, let $\gamma_{e}' := \gamma_{e}$). We can choose these curves $\gamma_{e}'$ such that whenever $\gamma_{e_1}$ and $\gamma_{e_2}$ do not cross, $\gamma_{e_1}'$ and $\gamma_{e_2}'$ do not cross. By slightly perturbing the curves of $\{\gamma_{e}' : e \in E(G)\}$ without creating new crossings between these curves, we can ensure that no three curves internally intersect at a common point. For each edge $e \in E(G)$, identify $e$ with $\gamma_{e}'$. So now $G$ is a topological graph.

Consider an edge $e = uv \in E(G)$ with $e\in B_u\in\mathcal{E}_u$ and $e\in B_v\in\mathcal{E}_v$. For each $x \in \{u, v\}$, let $V_{x} := \{w \in V(G) : wz_{B, w} \text{ crosses } xz_{B_x, x} \text{ for some } B \in \mathcal{E}_w\}$. Since $G$ is $k$-fan-bundle-planar, $|V_u| \leqslant k$ and $|V_v| \leqslant k$. By construction, every edge $e' \in E(G)$ that crosses $e$ is incident to $V_u \cup V_v$. Thus $G$ is $2k$-cover-planar and $2k$-matching-planar by \cref{relation}. Let $E'$ be a set of pairwise crossing edges incident to $u$ such that $e \in E'$. So every edge of $E' \setminus \{e\}$ is incident to $u$ and to a vertex of $V_{u} \cup V_{v}$. Hence $|E| \leqslant 2k + 1$. Thus no $2k + 2$ edges of $G$ incident to a common vertex pairwise cross.

Consider two edges $e_{1} = uv, e_{2} = ab \in E(G)$ with $e_1 \in B_u \in \mathcal{E}_u$, $e_1\in B_v\in\mathcal{E}_v$, $e_2 \in B_a \in \mathcal{E}_a$, $e_2\in B_b\in\mathcal{E}_b$. Since $G$ is $k$-fan-bundle-planar, for each $x \in \{u, v\}$, the edge $xz_{B_x, x}$ has at most $k$ common crossing points with $az_{B_{a}, a} \cup bz_{B_{b}, b}$. By construction, $e_1$ and $e_2$ have at most $2k$ crossing points in common.
\end{proof}

To distinguish $k$-fan-bundle-planar graphs and $k$-matching-planar graphs, we now show that $K_{3,n}$ is not $k$-fan-bundle-planar for any fixed $k$ and large $n$, whereas $K_{3,n}$ is 1-matching-planar for all $n$, as shown in \cref{K3n}(a). The next proposition qualitatively generalises a result of \citet{ABKKS18} who showed that $K_{4,567}$ is not 1-fan-bundle-planar. 

\begin{prop}
\label{NotFanBundlePlanar} The graph $K_{3, n}$ is not $k$-fan-bundle-planar for every $n \geqslant (12k + 3)^{8}$. 
\end{prop}

\begin{proof} Let $m := 12k + 3$. Assume for the sake of contradiction that $G:=K_{3,m^8}$ is $k$-fan-bundle-planar. Let $\{X,Y\}$ be the bipartition of $G$ where $|X|=3$ and $|Y|=m^8$.
Say $X=\{x_1,x_2,x_3\}$. 
Let $\mathcal{E}$ be the fan-bundling of $G$, and $G_{\mathcal{E}}$ be the topological graph witnessing that $G$ is $k$-fan-bundle-planar.

Our goal is to find a $3k$-planar drawing of $K_{3,m}$. To do so, we re-embed the vertices of $G$. We first re-embed the vertices of $Y$. For each $y\in Y$, if $\mathcal{E}_{y}$ has a fan-bundle $B$ of size $2$ or $3$, then this fan-bundle is unique, and we re-embed $y$ at the location of this fan-bundle $z_{B, y}$. Otherwise, $\mathcal{E}_{y}$ has three singleton fan-bundles, and we keep the location of $y$.

Let $E_1$ be a set of $m^4$ edges incident to $x_1$ in $G$ such that either all the edges in $E_1$ are in distinct fan-bundles in $\mathcal{E}_{x_1}$ or all the edges in $E_1$ are in the same fan-bundle in $\mathcal{E}_{x_1}$.
Such a set exists because there are $m^8=(m^4)^2$ edges incident to $x_1$ in $G$.
Let $Y_1$ be the set of vertices in $Y$ incident to the edges in $E_1$.
If all the edges in $E_1$ are in the same fan-bundle $B_{1} \in \mathcal{E}_{x_1}$, then re-embed $x_1$ at the location of this fan-bundle $z_{B_{1}, x_{1}}$.
Let $E_2$ be a set of $m^2$ edges between $x_2$ and $Y_1$ in $G$ such that either all the edges in $E_2$ are in distinct fan-bundles in $\mathcal{E}_{x_2}$ or all the edges in $E_2$ are in the same fan-bundle in $\mathcal{E}_{x_2}$.
Such a set exists because there are $m^4=(m^2)^2$ edges between $x_2$ and $Y_1$ in $G$. Let $Y_2$ be the set of vertices in $Y_1$ incident to the edges in $E_2$.
If all the edges in $E_2$ are in the same fan-bundle $B_{2} \in \mathcal{E}_{x_{2}}$, then re-embed $x_2$ at the location of this fan-bundle $z_{B_{2}, x_{2}}$.
Let $E_3$ be a set of $m$ edges between $x_3$ and $Y_2$ in $G$ such that either all the edges in $E_3$ are in distinct fan-bundles in $\mathcal{E}_{x_3}$ or all the edges in $E_3$ are in the same fan-bundle in $\mathcal{E}_{x_3}$.
Such a set exists because there are $m^2$ edges between $x_3$ and $Y_2$ in $G$.
Let $Y_3$ be the set of vertices in $Y_1$ incident to the edges in $E_2$. So $|Y_3|=m$. 
If all the edges in $E_3$ are in the same fan-bundle $B_{3} \in \mathcal{E}_{x_{3}}$, then re-embed $x_3$ at the location of this fan-bundle~$z_{B_{3}, x_{3}}$.

Now $G_{\mathcal{E}}$ restricts to a drawing of the complete bipartite graph $K_{3,m}$ with bipartition $\{Y_3, X\}$ such that each edge is drawn in the union of at most three crossed edges of $G_{\mathcal{E}}$ and one uncrossed edge of $G_{\mathcal{E}}$.
Since each crossed edge of $G_{\mathcal{E}}$ is involved in at most $k$ crossings, this drawing is $3k$-planar.

We have established that $K_{3, 12k + 3}$ is $3k$-planar. This contradicts a result of \citet{ABKKS18} that says that $K_{3, 4k' + 3}$ is not $k'$-planar for every integer $k' \geqslant 0$. Thus $K_{3, (12k + 3)^{8}}$ is not $k$-fan-bundle-planar, and the result follows.
\end{proof}

\section{Coloured Planarisations} 
\label{tools}

This section introduces an auxiliary graph that is a useful tool in the proofs of our upper bounds on row treewidth, layered treewidth, and treewidth. In what follows, $G$ is a topological graph and $\phi$ is a transparent ordered $c$-edge-colouring of $G$. Recall that $G'$ is the planarisation of $G$ (see \cref{topologicalgraphs}). For any edge $e \in E(G)$, let \defn{$L_{e}$} be the path in $G'$ determined by the curve that $e$ describes in the plane.

Define the \defn{level} of a dummy vertex $d \in e_{1} \cap e_{2}$ to be $\level(d) := \min(\phi(e_{1}), \phi(e_{2}))$. For any $v \in V(G)$, let $\level(v) := 0$. Let \defn{$G^{\phi}$} be the topological planar graph obtained from $G'$ as follows: for each edge $e \in E(G)$ and for any two consecutive (along $e$) dummy vertices $d_{1}, d_{2} \in L_{e}$ such that $\level(d_{1}) = \level(d_{2}) = \phi(e)$, contract the edge $d_{1}d_{2}$ in $G'$. We say that $G^{\phi}$ is the \defn{coloured planarisation} of $G$. See \cref{firstexample,fragmentssections,colouredplanarisation} for examples of coloured planarisations. In these figures, the colours of the edges of $G'$ and $G^{\phi}$ are kept for better visual understanding, but formally speaking we do not define edge-colourings of $G'$ or $G^{\phi}$. The vertices of $G$ are grey, and the vertices of $V(G') \setminus V(G)$ and $V(G^{\phi}) \setminus V(G)$ are black. 

\begin{figure}[h]
    \centering
    \begin{subfigure}[t]{0.29\textwidth}
    \centering
        \scalebox{1}{\includegraphics{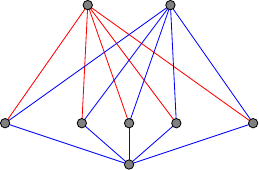}}
        \subcaption{$G$ and $\phi$}
        \label{figurea}
    \end{subfigure}
    \quad \quad
    \begin{subfigure}[t]{0.29\textwidth}
    \centering
        \scalebox{1}{\includegraphics{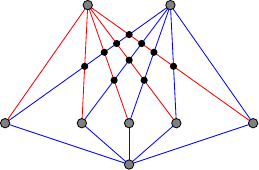}}
        \subcaption{$G'$}
        \label{figureb}
    \end{subfigure}
    \quad \quad
        \begin{subfigure}[t]{0.29\textwidth}
        \centering
        \scalebox{1}{\includegraphics{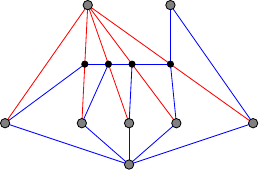}}
        \subcaption{$G^{\phi}$}
        \label{figurec}
    \end{subfigure}
    \caption{An example of a planarisation and a coloured planarisation. (a) A topological graph $G$ isomorphic to $K_{3, 5}$ with a transparent ordered $2$-edge-colouring $\phi$, where colours are: red $= 1$, blue $= 2$. (b) The planarisation $G'$ of $G$ where every dummy vertex has level $1$ and every vertex of $G$ has level $0$. (c) The coloured planarisation $G^{\phi}$ of $G$ obtained by contracting red edges of $G'$ not incident to $V(G)$. Every vertex of $V(G^{\phi}) \setminus V(G)$ has level $1$ and every vertex of $G$ has level $0$.     }
    \label{firstexample}
\end{figure}

Let \defn{$\psi$} $: V(G') \rightarrow V(G^{\phi})$ be the surjective function determined by the contraction operation in the construction of $G^\phi$. We emphasise that $G^{\phi}$ depends upon the ordering of the colours in the ordered $c$-edge-colouring $\phi$. Note that no edge incident to a vertex of $G$ is contracted in the construction of $G^{\phi}$. So $V(G) \subseteq V(G^{\phi})$ and $\psi(v) = v$ for each $v \in V(G)$.

Let $e \in E(G)$ be an arbitrary edge. The crossing points of $e$ and the edges of colour less than $\phi(e)$ split $e$ into subcurves, called the \defn{fragments} of $e$ (see \cref{fragmentssectionsa}). For each $e \in E(G)$, every fragment of $e$ naturally induces a subpath of $L_{e}$. Let $M$ be such a subpath. If $M$ consists of at least three vertices, then the subpath of $M$ obtained by deleting the endpoints of $M$ is called a \defn{section} of $L_{e}$ (see \cref{fragmentssectionsb}). By definition, every section of $L_{e}$ is non-empty.

\begin{figure}[h]
    \centering
    \begin{subfigure}[t]{0.44\textwidth}
    \centering
        \scalebox{1.13}{\includegraphics{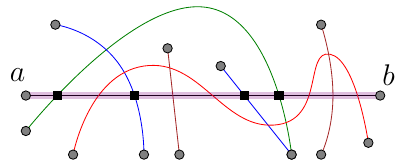}}
            \subcaption{$G$}
            \label{fragmentssectionsa}
    \end{subfigure} \quad \quad  
    \begin{subfigure}[t]{0.5\textwidth}
    \centering
        \scalebox{1.13}{\includegraphics{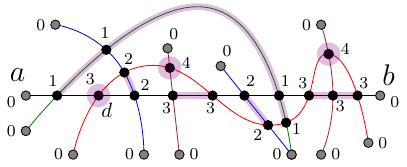}}
        \subcaption{$G'$}
        \label{fragmentssectionsb}
    \end{subfigure}
    \caption{An example of fragments and sections. (a) A topological graph $G$ with a transparent ordered $5$-edge-colouring $\phi$, where colours are: green $= 1$, blue $= 2$, black $= 3$ (only the edge $ab$ is black), red $= 4$, and brown $= 5$. The edge $ab$ is split by the crossing points (marked as squares) of $ab$ and the edges of smaller colours into five fragments (highlighted in purple). (b) The planarisation $G'$ of $G$. Each vertex is labelled by its level. The edges of sections and $1$-vertex sections of $G'$ are highlighted in purple. There are three sections of $L_{ab}$, one of which consists of a single dummy vertex labelled $d$.}
    \label{fragmentssections}
\end{figure}

Let $S_{1}$ be a section of $L_{e_{1}}$ and $S_{2}$ be a section of $L_{e_{2}}$, where $e_{1}, e_{2} \in E(G)$, such that $S_{1} \neq S_{2}$. If $e_{1} = e_{2}$ then $S_{1}$ and $S_{2}$ are disjoint. Otherwise, $e_{1} \neq e_{2}$. If $S_{1} \cap S_{2} \neq \emptyset$ then there exists a dummy vertex $d \in S_{1} \cap S_{2}$, and hence $d \in e_{1} \cap e_{2}$. Since $\phi$ is transparent, $\phi(e_{1}) \neq \phi(e_{2})$. Without loss of generality, $\phi(e_{1}) < \phi(e_{2})$. By definition, $d$ is not a vertex of a section of $L_{e_{2}}$, a contradiction. Thus, sections of $G'$ are pairwise disjoint.

For each section $S$ of $G'$, there exists exactly one edge $e \in E(G)$ such that $L_{e}$ contains $S$ and $\phi(e)$ is equal to the common level of vertices of $S$. The coloured planarisation $G^{\phi}$ is obtained from the planarisation $G'$ by contracting every edge in every section of $G'$ (see \cref{colouredplanarisation}). For $v \in V(G)$, $\psi^{-1}(v) = \{v\}$. Since sections are pairwise disjoint, for $x \in V(G^{\phi}) \setminus V(G)$, $\psi^{-1}(x)$ is the vertex set of a section of $G'$. Note that $(\psi^{-1}(x) : x \in V(G^{\phi}))$ is a partition of $G'$.

\begin{figure}[h]
    \centering
        \scalebox{1.13}{\includegraphics{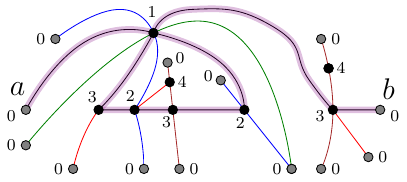}}
    \caption{The coloured planarisation $G^{\phi}$ of the graph $G$ with the transparent ordered $5$-edge-colouring $\phi$ from \cref{fragmentssectionsa}. Each vertex is labelled by its level.
    The edges between consecutive vertices of the walk $W_{ab}$ in $G^{\phi}$ are highlighted in purple.  
    }
    \label{colouredplanarisation}
\end{figure}

For each vertex $x \in V(G^{\phi})$, define the \defn{level} of $x$, denoted $\level(x)$, to be the common level of the vertices in $\psi^{-1}(x)$. Observe that the vertices of level $0$ are exactly the vertices of $G$.

Consider any edge $e = uv \in E(G)$. Let $u = w_{0},\dots, w_{r} = v$ be the path $L_{e}$ in $G'$. Let \defn{$W_{e}$} be the walk in $G^{\phi}$ obtained from $(\psi(w_{0}), \psi(w_{1}),\dots, \psi(w_{r}))$ by identifying consecutive identical vertices.

We now establish several basic properties of coloured planarisations.

\begin{lem} \label{1} For each $uv \in E(G)$, we have $W_{uv} \setminus \{u, v\} \subseteq V(G^{\phi}) \setminus V(G)$.
    
\end{lem}

\begin{proof}

Since $L_{uv} \cap V(G) = \{u, v\}$, we have $W_{uv} \cap V(G) = \{u, v\}$.
\end{proof}

\begin{lem} \label{2} For each $e \in E(G)$, the level of each vertex in $W_{e}$ is at most $\phi(e)$.
    
\end{lem}

\begin{proof}

By definition, the level of each vertex in $L_{e}$ is at most $\phi(e)$. Hence, the level of each vertex in $W_{e}$ is at most $\phi(e)$.
\end{proof}

\begin{lem} \label{walklength} Let $e \in E(G)$ be an edge involved in at most $t$ crossings with the edges of colour less than $\phi(e)$. Then the length of $W_{e}$ is at most $2(t + 1)$.
    
\end{lem}

\begin{proof}
    The path $L_{e}$ in $G'$ is split by the dummy vertices of level less than $\phi(e)$ into at most $t + 1$ subpaths. Every such subpath does not contain a dummy vertex of level less than $\phi(e)$. Therefore, the length of $W_{e}$ is at most $2(t + 1)$.
\end{proof}

\begin{lem} \label{3}
    Let $x \in V(G^{\phi}) \setminus V(G)$. Then there exists exactly one edge $e \in E(G)$ such that $\phi(e) = \level(x)$ and $x \in W_{e}$. Moreover, $L_{e}$ contains $\psi^{-1}(x)$.
\end{lem}

\begin{proof}
    By assumption, $\psi^{-1}(x)$ is the vertex set of a section of $G'$. Hence, there exists exactly one edge $e \in E(G)$ such that $\phi(e) = \level(x)$ and $L_{e}$ contains $\psi^{-1}(x)$. Then $x \in W_{e}$. Since $\phi$ is transparent, no edge of $G$ of colour $\phi(e)$ crosses $e$. Hence, $e$ is the only edge of $G$ of colour $\level(x)$ that contains a dummy vertex of $\psi^{-1}(x)$. Thus there is no edge $e_{1} \in E(G) \setminus \{e\}$ such that $\phi(e_{1}) = \level(x)$ and $x \in W_{e_{1}}$.
\end{proof}

\begin{lem} \label{thencross} Let $x \in V(G^{\phi}) \setminus V(G)$ be a vertex and $e \in E(G)$ be an edge such that $\level(x) = \phi(e)$ and $x \in W_{e}$. Let $S$ be the section of $G'$ such that $\psi^{-1}(x) = V(S)$. Let $g \in E(G)$ be an edge such that $\phi(g) > \level(x)$. Then $x \in W_{g}$ if and only if $g$ crosses the fragment of $e$ in $G$ that corresponds to $S$. In particular, if $x \in W_{g}$ then $e$ and $g$ cross.
\end{lem}

\begin{proof} By \cref{3}, $L_{e}$ contains $\psi^{-1}(x)$. Let $\gamma$ be the fragment of $e$ that corresponds to $S$. If $g$ crosses $\gamma$ then $L_{g}$ contains a dummy vertex of $\psi^{-1}(x)$, and hence $x \in W_{g}$.

If $x \in W_{g}$ then both $L_{g}$ and $S$ contain a dummy vertex of $\psi^{-1}(x)$, and hence $g$ crosses $\gamma$. Since $\gamma$ is a fragment of $e$, this implies that $e$ and $g$ cross.
\end{proof}

\begin{lem} \label{5} For any edge $e$ of $G$, no two vertices with level $\phi(e)$ are consecutive in $W_{e}$.
\end{lem}

\begin{proof}

Assume for the sake of contradiction that some consecutive vertices $x, y$ in $W_{e}$ have level $\phi(e)$. By definition of $W_{e}$, $x \neq y$. Since $\level(x) = \level(y) = \phi(e)$, $\psi^{-1}(x)$ and $\psi^{-1}(y)$ are the vertex sets of some distinct sections $S_{1}$ and $S_{2}$ of $e$. By definition of $W_{e}$, there exist two dummy vertices $d_{x} \in S_{1}$, $d_{y} \in S_{2}$ such that $\psi(d_{x}) = x$, $\psi(d_{y}) = y$, and $d_{x}$, $d_{y}$ are consecutive vertices in the path $L_{e}$. By definition of sections, no two dummy vertices of distinct sections of $L_{e}$ are consecutive in $L_{e}$, a contradiction.
\end{proof}

\begin{lem} \label{6} For each $x \in V(G^{\phi})$, there exists $v \in V(G)$ such that $\dist_{G^{\phi}}(x, v) \leqslant c - 1$.
\end{lem}

\begin{proof}

Let $y \in V(G^{\phi}) \setminus V(G)$ be an arbitrary vertex. By \cref{3}, there exists an edge $e \in E(G)$ such that $\phi(e) = \level(y)$ and $y \in W_{e}$. By \cref{5}, there exists a vertex $z \in W_{e}$ such that $\level(z) \neq \level(y)$ and $yz \in E(G^{\phi})$. Since the level of each vertex in $L_{e}$ is at most $\phi(e)$, we have $\level(z) \leqslant \phi(e) = \level(y)$. So $\level(z) < \level(y)$. Hence, each vertex $y \in V(G^{\phi}) \setminus V(G)$ has a neighbour in $G^{\phi}$ of level less than $\level(y)$.

By definition, the level of each vertex in $G'$ is at most $c - 1$. Therefore, the level of each vertex in $G^{\phi}$ is at most $c - 1$. Hence, $\level(x) \leqslant c - 1$. By the observation above, there exists a path $x = x_{0}, x_{1}, \dots, x_{r} = v$ in $G^{\phi}$ such that $\level(v) = 0$ and $\level(x_{i + 1}) < \level(x_{i})$ for each $i \in \{0,\dots,r - 1\}$. The vertices of level $0$ are exactly the vertices of $G$, so $v \in V(G)$. Therefore, the length of this path is at most $\level(x) \leqslant c - 1$. Thus $\dist_{G^{\phi}}(x, v) \leqslant c - 1$, as desired.
\end{proof}

We now prove the Coloured Planarisation Lemma, which is a crucial ingredient in the proofs of our upper bounds on row treewidth, layered treewidth, and treewidth in \cref{result,FinalSection}. 
The proofs of these results consider models of graphs in $H \boxtimes K_{t}$, where $H$ is planar. Throughout, we assume that $V(K_{t}) = \{1, \dots, t\}$.

\begin{lem} [\CPL] \label{generallemma} Suppose that a topological graph $G$ has a transparent ordered $c$-edge-colouring $\phi$ such that for any $i, j \in \{1,\dots,c\}$ with $i < j$, for any edge $e$ of colour $i$ and for any fragment $\gamma$ of $e$, the matching number of the set of edges of colour $j$ that cross $\gamma$ is at most $m$.

Then there exists a positive integer $t$ and a model $\mu$ of $G$ in $G^{\phi} \boxtimes K_{t}$ such that:

\begin{enumerate}[(a)]
    \item \label{CPLa} $t \leqslant 1 + 5(c - 1)m$,
    \item \label{CPLb} if $G$ is circular, then $t \leqslant 1 + 3(c - 1)m$,
    \item \label{CPLc} for each $v \in V(G)$ and $x \in V(G^{\phi})$, if $(x, i) \in \mu(v)$ for some $i \in \{1, \dots, t\}$, then $x \in W_{vw} \setminus \{v, w\}$ for some edge $vw \in E(G)$ or $x = v$.
\end{enumerate}

\begin{proof}

For each $i \in \{1, \dots, c\}$, let $G_{i}$ be the subgraph of $G$ induced by the set of edges of colour $i$. Since $\phi$ is transparent, $G_{i}$ is planar. Let $s := \max\{\st(G_{1}), \dots, \st(G_{c})\}$. \citet{HMS-DM96} proved that every planar graph has star arboricity at most $5$, so $s \leqslant 5$.

Let $\mathcal{C} := \{(i, j) : i \in \{1,\dots,c\}, \text{ } j \in \{1,\dots,s\}\}$. By definition of $s$, the edges of $G_{i}$ can be coloured with colours $(i, 1), \dots, (i, s)$ such that the subgraph of $G_{i}$ induced by the set of edges of any new colour is a star-forest. So there exists a transparent $sc$-edge-colouring $\phi' : E(G) \rightarrow \mathcal{C}$ such that:

\begin{itemize}

\item for any $i \in \{1, \dots, c\}$ and any $e \in E(G_{i})$, $\phi'(e) = (i, j)$ for some $j \in \{1, \dots, s\}$, and

\item for any $(i, j) \in \mathcal{C}$, the subgraph of $G$ induced by $\{e \in E(G) : \phi'(e) = (i, j)\}$ is a star-forest.

\end{itemize}

For any $(i, j) \in \mathcal{C}$, let $G_{i, j}$ be the subgraph of $G$ induced by $\{e \in E(G) : \phi'(e) = (i, j)\}$. By definition of $\phi'$, $G_{i, j}$ is a star-forest. Fix a centre of each component of $G_{i,j}$. Observe that, for any edge $ab \in E(G_{i, j})$, exactly one of the endpoints of $ab$, say $a$, is the centre of a component of $G_{i, j}$. Then we say that $a$ is the \defn{dominant endpoint} of $ab$. Thus every edge of $G$ has exactly one dominant endpoint.

Recall that, for $e \in E(G)$, $L_{e}$ is the path in the planarisation $G'$ of $G$ associated with $e$ and $W_{e}$ is the walk in the coloured planarisation $G^{\phi}$ of $G$. For any $x \in V(G^{\phi}) \setminus V(G)$ and any $(i, j) \in \mathcal{C}$ with $i \geqslant \level(x)$, let $B_{x}^{(i, j)}$ be the set of dominant endpoints of the edges $e \in E(G)$ such that $\phi'(e) = (i, j)$ and $x \in W_{e}$. For any $v \in V(G)$, let $B_{v} := \{v\}$. For any $x \in V(G^{\phi}) \setminus V(G)$, let 

$$B_{x} := \bigcup_{(i, j) \in \mathcal{C} \text{ } : \text{ } i \geqslant \level(x)} B_{x}^{(i, j)}.$$

Let $t := \max\{|B_{x}| : x \in V(G^{\phi})\}$. We now define a model $\mu$ of $G$ in $G^{\phi} \boxtimes K_{t}$. For each $x \in V(G^{\phi})$, let $\lambda_{x} : B_{x} \rightarrow \{1, \dots, |B_{x}|\}$ be an injective function. For each $v \in V(G)$, define $\mu(v) := \{(x, \lambda_{x}(v)) : v \in B_{x}\}$. Note that $(v, 1) \in \mu(v)$. In the next two claims, we prove that $\mu$ is a model of $G$ in $G^{\phi} \boxtimes K_{t}$.

\begin{claim} \label{almostdecompositionone}
For each $v \in V(G)$, $\mu(v)$ is non-empty and $(G^{\phi} \boxtimes K_{t})[\mu(v)]$ is connected. 
\end{claim}

\begin{proof}
Since $(v, 1) \in \mu(v)$, the set $\mu(v)$ is non-empty. We now show that $(G^{\phi} \boxtimes K_{t})[\mu(v)]$ is connected. Let $(x, i) \in \mu(v)$ for some $x \in V(G^{\phi}) \setminus \{v\}$ and $i \in \{1, \dots, |B_{x}|\}$. By definition of $\mu$, we have $v \in B_{x}$. Since $B_{u} = \{u\}$ for each $u \in V(G)$ and $x \neq v$, this implies that $x \in V(G^{\phi}) \setminus V(G)$. Since $v \in B_{x}$, there is an edge $vw \in E(G)$ such that $v$ is the dominant endpoint of $vw$, and $x \in W_{vw} \setminus \{v, w\}$, and $v \in B_{x}^{\phi'(vw)}$. Let $y \in W_{vw} \setminus \{v, w\}$ be a vertex and $(i_{0}, j_{0}) := \phi'(vw)$. By definition of $\phi'$, $i_{0} = \phi(vw)$. By \cref{2}, $i_{0} \geqslant \level(y)$. By \cref{1}, $y \in V(G^{\phi}) \setminus V(G)$. By definition of $B_{y}^{\phi'(vw)}$, we have $v \in B_{y}^{\phi'(vw)}$, and hence $v \in B_{y}$. For every such vertex $y$, we have $(y, \lambda_{y}(v)) \in \mu(v)$. Consequently, there is a walk in $G^{\phi} \boxtimes K_{t}$ with endpoints $(x, i)$ and $(v, 1)$ such that every vertex of the walk belongs to $\mu(v)$. Thus $(G^{\phi} \boxtimes K_{t})[\mu(v)]$ is connected.
\end{proof}

\begin{claim} \label{almostdecompositiontwo} For all distinct $v, w \in V(G)$, $\mu(v) \cap \mu(w) = \emptyset$. For every edge $vw \in E(G)$, $ab \in E(G^{\phi} \boxtimes K_{t})$ for some $a \in \mu(v)$ and $b \in \mu(w)$.
\end{claim}

\begin{proof}

First, let $v, w \in V(G)$ be distinct. By construction, if $(x, i) \in \mu(v)$ for some $x \in V(G^{\phi})$ and $i \in \{1, \dots, t\}$, then $i = \lambda_{x}(v)$. Similarly, if $(x, i) \in \mu(w)$, then $i = \lambda_{x}(w)$. Since $\lambda_{x}$ is injective, $\mu(v) \cap \mu(w) = \emptyset$.

Now assume that $vw \in E(G)$. Without loss of generality, $v$ is the dominant endpoint of $vw$. Let $x_{0} \in W_{vw} \setminus \{w\}$ be the neighbour of $w$ in $W_{vw}$ such that $x_{0}$ and $w$ are consecutive in $W_{vw}$. If $x_{0} = v$ then $\{v\} = B_{x_{0}}$. Otherwise, $x_{0} \neq v$ and by \cref{1}, we have $x_{0} \in W_{vw} \setminus V(G)$. By \cref{2}, $\phi(vw) \geqslant \level(x_{0})$. By construction,  $v \in B_{x_0}^{\phi'(vw)}$, and hence $v \in B_{x_0}$. Let $a := (x_{0}, \lambda_{x_{0}}(v))$ and $b := (w, 1)$. Since $v \in B_{x_{0}}$ and $\{w\} = B_{w}$, we have $a \in \mu(v)$ and $b \in \mu(w)$. Since $x_{0}w \in E(G^{\phi})$, we have $ab \in E(G^{\phi} \boxtimes K_{t})$, as desired.
\end{proof}

By \cref{almostdecompositionone,almostdecompositiontwo}, $\mu$ is a model of $G$ in $G^{\phi} \boxtimes K_{t}$.

We now show an upper bound on $t$. Fix some $x \in V(G^{\phi}) \setminus V(G)$. By \cref{3}, there exists exactly one edge $e \in E(G)$ such that $\phi(e) = \level(x)$ and $x \in W_{e}$. By construction of $\phi'$, we have $(\level(x), j_{0}) = \phi'(e)$ for some $j_{0} \in \{1, \dots, s\}$. By construction, $|B_{x}^{(\level(x), j_{0})}| = 1$ and $B_{x}^{(\level(x), j)} = \emptyset$ for any $j \in \{1, \dots, s\} \setminus \{j_{0}\}$. Consequently, $|B_{x}^{(\level(x), 1)}| + \dots + |B_{x}^{(\level(x), s)}| = 1$.

Now, fix some $i \in \{\level(x) + 1, \dots, c\}$ and $j \in \{1, \dots, s\}$. Recall that $\psi$ is the surjective function determined by the contraction operation in the construction of $G^\phi$. Since $x \in V(G^{\phi}) \setminus V(G)$, $\psi^{-1}(x)$ is the vertex set of a section $S$ of $G'$. By \cref{3}, $L_{e}$ contains $S$.  Let $\gamma$ be the fragment of $e$ that corresponds to $S$. Let $E$ be the set of edges $g \in E(G)$ that cross $\gamma$ and $\phi'(g) = (i, j)$. By construction of $\phi'$, we have $\phi(g) = i > \level(x) = \phi(e)$ for each $g \in E$. By \cref{thencross}, $x \in W_{g}$ for some $g \in E(G)$ such that $\phi'(g) = (i, j)$ if and only if $g \in E$. By construction, $B_{x}^{(i, j)}$ consists of the dominant endpoints of the edges in $E$. By assumption, the matching number of $E$ is at most $m$. Therefore, $E$ is contained in the union of at most $m$ components of $G_{i, j}$. So $|B_{x}^{(i, j)}| \leqslant m$. Since $x \in V(G^{\phi}) \setminus V(G)$, we have $\level(x) \geqslant 1$. Thus,
\begin{align*}
|B_{x}| &\leqslant \sum\limits_{i \geqslant \level(x)} \sum\limits_{j \in \{1,\dots,s\}} |B_{x}^{(i, j)}| \\
 &= 
\sum\limits_{j \in \{1,\dots,s\}}  |B_{x}^{(\level(x), j)}| \hspace{3mm}  
+ \sum\limits_{i > \level(x)} \sum\limits_{j \in \{1,\dots,s\}} |B_{x}^{(i, j)}| \\ 
&\leqslant  1 + (c - \level(x))sm \\
 &\leqslant 1 + (c - 1)sm.
\end{align*}

For any $v \in V(G)$, $|B_{v}| = 1$. Thus $t \leqslant 1 + (c - 1)sm \leqslant 1 + 5(c - 1)m$. This shows property~\ref{CPLa}.

If $G$ is circular then $G_{i}$ is outerplanar for any $i \in \{1,\dots,c\}$. \citet{HMS-DM96} proved that every outerplanar graph has star arboricity at most $3$. So in this case $s \leqslant 3$ and $t \leqslant 1 + (c - 1)sm \leqslant 1 + 3(c - 1)m$. This shows property \ref{CPLb}.

Let $v \in V(G)$ and $x \in V(G^{\phi})$ be two vertices such that $(x, i) \in \mu(v)$ for some $i \in \{1, \dots, t\}$.  By construction of $\mu$, $v \in B_{x}$. If $x \in V(G)$ then $x = v$ because $B_{u} = \{u\}$ for each $u \in V(G)$. Otherwise, $x \in V(G^{\phi}) \setminus V(G)$. By construction of $B_{x}$, there exists an edge $vw \in E(G)$ such that $x \in W_{vw} \setminus \{v, w\}$. This shows property \ref{CPLc}. Thus $\mu$ satisfies the conditions of the lemma.
\end{proof}
\end{lem}

The next lemma bounds the distance between two vertices in the coloured planarisation and is used in the proof of our upper bounds on row treewidth and layered treewidth in \cref{FinalSection}. The proof relies on the following definitions about walks. Let $W$ be a walk in a graph $G$ with distinct endpoints. Let $u$ be one of the endpoints of $W$. Then we can enumerate the vertices of $W$ such that $W = (v_{1}, \dots, v_{t})$ and $u = v_{1}$. Let $a$ be a vertex of $W$ such that $a \neq u$. Let $i \in \{2, \dots, t\}$ be the index such that $a = v_{i}$ and $a \neq v_{j}$ for any $j \in \{1,\dots,i - 1\}$. Then we say that $v_{i - 1}$ is the neighbour of $a$ \defn{towards} $u$ in $W$. Since $v_{1} \neq v_{t}$, the neighbour of $a$ towards $u$ in $W$ is unambiguously defined by $a$, $W$ and $u$. Let $b$ be a vertex of $W$. We say that $b$ is \defn{between} $a$ and $u$ in $W$ if there exists $j \in \{1, \dots, i\}$ such that $b = v_{j}$. In particular, $b$ is between $u$ and $a$ if $b \in \{a, u\}$. Observe that the neighbour of $a$ towards $u$ in $W$ is between $a$ and $u$ in $W$. If $b$ is between $a$ and $u$ in $W$ and a vertex $x$ is between $b$ and $u$ in $W$ then $x$ is between $a$ and $u$ in~$W$.

Recall that, for $e \in E(G)$, $W_{e}$ is the walk associated with $e$ in the coloured planarisation $G^{\phi}$ of $G$, where $\phi$ is a transparent ordered $c$-edge-colouring of a topological graph $G$.

\begin{lem} [\DL] \label{distancelemma}
Suppose that a topological graph $G$ has a transparent ordered $c$-edge-colouring $\phi$ such that for any $e \in E(G)$, the vertex cover number of the set of edges of colour less than $\phi(e)$ that cross $e$ is at most $k$. Then, for any $e = uw \in E(G)$ and any $x \in W_{e} \setminus \{u, w\}$, we have $\dist_{G^{\phi}}(u, x) \leqslant \frac{2^{c + 1}k^{c} - 2k - 1}{2k - 1}$.
\end{lem}

\begin{proof}

Let $h(i) := \frac{2^{i + 1}k^{i} - 2k - 1}{2k - 1}$ for any $i \geqslant 1$. Observe that $h(1) = 1$ and $h(i) = 2k(h(i - 1) + 1) + 1$ for any $i \geqslant 2$. By induction, we prove that for any $e = uw \in E(G)$ and any $x \in W_{e} \setminus \{u, w\}$, $\dist_{G^{\phi}}(x, u) \leqslant h(\phi(e))$. \cref{distancelemma} follows from this because $h(\phi(e)) \leqslant h(c) = \frac{2^{c + 1}k^{c} - 2k - 1}{2k - 1}$.

By \cref{1}, since $x \in W_{e} \setminus \{u, w\}$, we have $x \in W_{e} \setminus V(G)$. Consider the base case with $\phi(e) = 1$. By \cref{walklength}, the length of $W_{e}$ is at most $2$. Then $xu, xw \in E(G^{\phi})$, and hence $\dist_{G^{\phi}}(x, u) = 1 = h(1) = h(\phi(e))$, as desired.

Now assume that $\phi(e) = i$ for some $i \in \{2,\dots,c\}$. By \cref{2}, $\level(x) \leqslant i$. By assumption, there exists a set $X_{e} \subseteq V(G)$ such that $|X_{e}| \leqslant k$ and every edge of colour less than $\phi(e)$ that crosses $e$ is incident to $X_{e}$. For any $v \in X_{e}$, let $E_{v}$ be the set of edges of $G$ of colour at most $i - 1$ that are incident to $v$ and cross $e$. For any $v \in X_{e}$, let $V_{v} := \{y \in W_{e} \setminus \{u, w\} : y \in W_{g} \text{ for some } g \in E_{v}\}$. For any $A \subseteq X_{e}$, define $E_{A} := \bigcup_{v \in A}E_{v}$ and $V_{A} := \bigcup_{v \in A}V_{v}$.

\begin{claim} \label{layeredtreewidthclaim} For any $j \in \{0,\dots,|X_{e}|\}$, there exists a vertex $x_{j} \in W_{e} \setminus \{w\}$ and a set $A_{j} \subseteq X_{e}$ such that:

\begin{itemize} 

\item $\dist_{G^{\phi}}(x, x_{j}) \leqslant 2j(h(i - 1) + 1)$,

\item $|A_{j}| = j$,

\item no vertex of $V_{A_{j}}$ is between $x_{j}$ and $u$ in $W_{e}$.

\end{itemize}
    
\end{claim}

\begin{proof}
    We prove this claim by induction on $j$.

    Consider the base case with $j = 0$. \cref{layeredtreewidthclaim} is trivial for $j = 0$, $x_{0} := x$, and $A_{0} := \emptyset$.

    Now assume that $j \in \{1, \dots, |X_{e}|\}$. By the inductive hypothesis (for \cref{layeredtreewidthclaim}), there exists $x_{j - 1} \in W_{e} \setminus \{w\}$ and a set $A_{j - 1} \subseteq X_{e}$ such that all three properties in \cref{layeredtreewidthclaim} are satisfied for $x_{j - 1}$ and $A_{j - 1}$. Since $j \leqslant |X_{e}|$ and $|A_{j - 1}| = j - 1$, the set $X_{e} \setminus A_{j - 1}$ is non-empty.

    If $x_{j - 1} = u$ then $\dist_{G^{\phi}}(x, u) \leqslant 2(j - 1)(h(i - 1) + 1) \leqslant 2j(h(i - 1) + 1)$. In this case, let $x_{j} := u$ and $A_{j} := A_{j - 1} \cup \{a\}$ for any $a \in X_{e} \setminus A_{j - 1}$. All three properties in \cref{layeredtreewidthclaim} are satisfied for this choice of $x_{j}$ and $A_{j}$.

    Otherwise, $x_{j - 1} \neq u$. Let $y_{j}$ be the neighbour of $x_{j - 1}$ towards $u$ in $W_{e}$.
    
    If $y_{j} = u$ then $\dist_{G^{\phi}}(x, u) \leqslant 2(j - 1)(h(i - 1) + 1) + 1 \leqslant 2j(h(i - 1) + 1)$.  In this case, let $x_{j} := u$ and $A_{j} := A_{j - 1} \cup \{a\}$ for any $a \in X_{e} \setminus A_{j - 1}$. All three properties in \cref{layeredtreewidthclaim} are satisfied for this choice of $x_{j}$ and $A_{j}$. 
    
    Otherwise, $y_{j} \neq u$. By \cref{2}, $\level(x_{j - 1}) \leqslant \phi(e) = i$ and $\level(y_{j}) \leqslant \phi(e) = i$. By \cref{5}, $x_{j - 1}$ or $y_{j}$ has level less than $i$. Let $z_{j} \in \{x_{j - 1}, y_{j}\}$ be such a vertex, so $\level(z_{j}) < i$. By definition of $z_{j}$, $\dist_{G^{\phi}}(x_{j - 1}, z_{j}) \leqslant 1$ and $z_{j}$ is between $x_{j - 1}$ and $u$ in $W_{e}$. Note that $z_{j} \in W_{e} \setminus \{u, w\}$, and hence $z_{j} \in W_{e} \setminus V(G)$ by \cref{1}. By \cref{3}, there exists an edge $e_{j} \in E(G)$ such that $\phi(e_{j}) = \level(z_{j}) < i = \phi(e)$ and $z_{j} \in W_{e_{j}}$. By \cref{thencross}, $e_{j}$ crosses $e$ in $G$.
    By assumption, $e_{j}$ is incident to $X_{e}$. Let $v_{j} \in X_{e}$ be an endpoint of $e_{j}$. Then $e_{j} \in E_{v_{j}}$ and $z_{j} \in V_{v_{j}}$. By the inductive hypothesis (for \cref{layeredtreewidthclaim}), no vertex of $V_{A_{j - 1}}$ is between $x_{j - 1}$ and $u$ in $W_{e}$. Since $z_{j}$ is between $x_{j - 1}$ and $u$ in $W_{e}$, this implies that $v_{j} \notin A_{j - 1}$.

    By the inductive hypothesis (for \cref{distancelemma}), $\dist_{G^{\phi}}(z_{j}, v_{j}) \leqslant h(\phi(e_{j})) \leqslant h(i - 1)$. Since $z_{j} \in V_{v_{j}}$, the set $V_{v_{j}}$ is non-empty. By definition, $u \notin V_{v_{j}}$. Let $a_{j}$ be the first vertex of $W_{e}$ starting at $u$ that is in $V_{v_{j}}$. So no vertex of $V_{v_{j}} \setminus \{a_{j}\}$ is between $a_{j}$ and $u$ in $W_{e}$. Moreover, $a_{j} \neq u$. By the inductive hypothesis (for \cref{distancelemma}), $\dist_{G^{\phi}}(v_{j}, a_{j}) \leqslant h(i - 1)$.

    Let $A_{j} := A_{j - 1} \cup \{v_{j}\}$. Since $v_{j} \notin A_{j - 1}$ and $|A_{j - 1}| = j - 1$, we have $|A_{j}| = j$.

    Let $x_{j}$ be the neighbour of $a_{j}$ towards $u$ in $W_{e}$. Since $z_{j}$ is between $x_{j - 1}$ and $u$ in $W_{e}$ and $a_{j}$ is between $z_{j}$ and $u$ in $W_{e}$, $a_{j}$ is between $x_{j - 1}$ and $u$ in $W_{e}$. Then, by the choice of $x_{j}$, $x_{j}$ is between $x_{j - 1}$ and $u$ in $W_{e}$. Then, since no vertex of $V_{A_{j - 1}}$ is between $x_{j - 1}$ and $u$ in $W_{e}$, no vertex of $V_{A_{j - 1}}$ is between $x_{j}$ and $u$ in $W_{e}$. By the choice of $a_{j}$ and $x_{j}$, no vertex of $V_{v_{j}}$ is between $x_{j}$ and $u$ in $W_{e}$. Thus, no vertex of $V_{A_{j}}$ is between $x_{j}$ and $u$ in $W_{e}$.

    By combining the above distance inequalities, we obtain that
    \begin{align*}
    \dist_{G^{\phi}}(x, x_{j}) &\leqslant \dist_{G^{\phi}}(x, x_{j - 1}) + \dist_{G^{\phi}}(x_{j - 1}, z_{j}) + \dist_{G^{\phi}}(z_{j}, v_{j}) \\ &\phantom{fish}+\dist_{G^{\phi}}(v_{j}, a_{j}) + \dist_{G^{\phi}}(a_{j}, x_{j}) \\ &\leqslant 2(j - 1)(h(i - 1) + 1) + 1 + h(i - 1) \\
    &\phantom{fish}+ h(i - 1) + 1 = 2j(h(i - 1) + 1).\qedhere
    \end{align*} 
    \end{proof}

By \cref{layeredtreewidthclaim} (setting $j = |X_{e}|$), there exists a vertex $r \in W_{e} \setminus \{w\}$ such that $\dist_{G^{\phi}}(x, r) \leqslant 2|X_{e}|(h(i - 1) + 1) \leqslant 2k(h(i - 1) + 1) < h(i)$ and (since $A_{|X_{e}|} = X_{e}$) no vertex of $V_{X_{e}}$ is between $r$ and $u$ in $W_{e}$. If $r = u$ then we are done. Otherwise, let $r_{0}$ be the neighbour of $r$ towards $u$ in $W_{e}$. Then $\dist_{G^{\phi}}(x, r_{0}) \leqslant 2k(h(i - 1) + 1) + 1 = h(i)$. If $r_{0} = u$ then we are done. Otherwise, $r_{0} \neq u$, so $r_{0}, r \in W_{e} \setminus \{u, w\}$ and hence $r_{0}, r \in W_{e} \setminus V(G)$ by \cref{1}.

By \cref{2}, $\level(r) \leqslant \phi(e) = i$ and $\level(r_{0}) \leqslant \phi(e) = i$. By \cref{5}, $r$ or $r_{0}$ has level less than $i$. Let $z \in \{r, r_{0}\}$ be such a vertex, so $\level(z) < i$ and $z$ is between $r$ and $u$ in $W_{e}$.  By \cref{3}, there exists an edge $g \in E(G)$ such that $\phi(g) = \level(z) < i = \phi(e)$ and $z \in W_{g}$. By \cref{thencross}, $g$ crosses $e$, and hence $g$ is incident to $X_{e}$. Therefore, $z \in V_{X_{e}}$, which contradicts \cref{layeredtreewidthclaim}.

We have shown that for any $e = uw \in E(G)$ and any $x \in W_{e} \setminus \{u, w\}$, $\dist_{G^{\phi}}(x, u) \leqslant h(\phi(e))$. Since $h(\phi(e)) \leqslant h(c) = \frac{2^{c + 1}k^{c} - 2k - 1}{2k - 1}$, the result follows.
\end{proof}

\section{Treewidth Bounds}
\label{result}

This section proves \cref{introcircularkcoverplanar,introcirculardrawingscoloured}, which provide upper bounds on the treewidth of certain circular graphs. We extend these results and show an upper bound on the treewidth of certain (not necessarily circular)  topological graphs with bounded radius. The proofs use coloured planarisations (\cref{tools}) and the Coloured Planarisation Lemma (\cref{generallemma}). We start with the following result, which immediately implies \cref{introcirculardrawingscoloured}.

\begin{thm} \label{circulardrawings} Let $G$ be a circular graph with a transparent ordered $c$-edge-colouring $\phi$. Suppose that for any $i, j \in \{1,\dots,c\}$  with $i < j$, for any edge $e$ of colour $i$ and for any fragment $\gamma$ of $e$, the matching number of the set of edges of colour $j$ that cross $\gamma$ is at most $m$. Then $\tw(G) \leqslant 9mc(c - 1) + 3c - 1$.
\end{thm}

\begin{proof}

By \cref{6}, the coloured planarisation $G^{\phi}$ of $G$ is $c$-outerplanar\footnote{A topological outerplanar graph is called \defn{$1$-outerplanar}. A topological planar graph is \defn{$c$-outerplanar} if the topological planar graph obtained by deleting the vertices on the outerface is $(c - 1)$-outerplanar.}. \citet{Bodlaender88} proved that every $c$-outerplanar graph has treewidth at most $3c - 1$, so $\tw(G^{\phi}) \leqslant 3c - 1$. By the Coloured Planarisation Lemma (\cref{generallemma}\ref{CPLb}), 
$G$ is a minor of $G^{\phi} \boxtimes K_{1 + 3(c - 1)m}$. Thus $\tw(G) \leqslant \tw(G^{\phi} \boxtimes K_{1 + 3(c - 1)m}) \leqslant (\tw(G^{\phi}) + 1)(1 + 3(c - 1)m) - 1 \leqslant 9mc(c - 1) + 3c - 1$.
\end{proof}

We now explain why \cref{introcirculardrawingscoloured} is a generalisation of \cref{knownresult}, which provides an upper bound on the treewidth of circular min-$k$-planar graphs. Let $G$ be a circular min-$k$-planar graph. Let $E(G) = E_{1} \cup E_{2}$, where $E_{1}$ is the set of edges that are involved in at least $k + 1$ crossings and $E_{2}$ is the set of edges that are involved in at most $k$ crossings. Since $G$ is circular min-$k$-planar, no two edges of $E_{1}$ cross. Greedily colour the edges of $E_{2}$ using colours $1, \dots, k + 1$ so that no two edges of $E_2$ of the same colour cross. Colour all the edges of $E_{1}$ using colour $k + 2$. So no two edges of the same colour cross. By construction, for any $i, j \in \{1, \dots, k + 2\}$ with $i < j$ and for any edge $e$ of colour $i$, the matching number of the set of edges of colour $j$ that cross $e$ is at most $k$. Hence, \cref{introcirculardrawingscoloured} gives the bound $\tw(G) \in \mathcal{O}(k^{3})$. Thus \cref{introcirculardrawingscoloured} implies that circular min-$k$-planar graphs have bounded treewidth, as shown in \cref{knownresult} (which gives a better bound on treewidth).

We now show that \cref{introcirculardrawingscoloured} is in fact a qualitative generalisation of \cref{knownresult} by considering complete bipartite graphs $K_{2,n}$. Let $H$ be a circular graph isomorphic to $K_{2, n}$ for any $n \geqslant 2$. Let $V(H) = \{a, b\} \cup X$, where all the vertices of $X$ are adjacent to both $a$ and $b$, and $ab \notin E(H)$. Colour all the edges of $H$ incident to $a$ by $1$ and colour all the edges of $H$ incident to $b$ by $2$.
Then monochromatic edges do not cross and for any edge $e$ of colour $1$, the matching number of the set of edges of colour $2$ that cross $e$ is at most $1$. So \cref{introcirculardrawingscoloured} is applicable with $m = 1$ and $c = 2$.
On the other hand, we now show that $K_{2, 2k + 3}$ is not isomorphic to a circular min-$k$-planar graph. Let $J$ be a circular min-$k$-planar graph isomorphic to $K_{2, 2k + 3}$. Let $a, b$ be two vertices such that every vertex of $V(J) \setminus \{a, b\}$ is adjacent to both $a$ and $b$. The vertices $a$ and $b$ split the circle into two arcs. One of these arcs contains at least $k + 2$ vertices. Let the order of the vertices in this arc be $a, v_{1}, \dots, v_{s}, b$, where $s \geqslant k + 2$. Then the edge $av_{s}$ crosses all the edges $bv_{1}, \dots, bv_{s - 1}$ and the edge $bv_{1}$ crosses all the edges $av_{2}, \dots, av_{s}$. So $av_{s}$ and $bv_{1}$ cross and each of these edges crosses at least $k + 1$ edges. So $K_{2, 2k + 3}$ is not isomorphic to a circular min-$k$-planar graph. Hence, \cref{knownresult} is not applicable for $K_{2, n}$ with large $n$. Thus,  \cref{introcirculardrawingscoloured} is a qualitative generalisation of \cref{knownresult}. 

Circular graphs are closely related to topological graphs of bounded radius, since one may add a dominant vertex outside the circle without introducing new crossings. Consider the class $\mathcal{G}_{c, m}$ of topological graphs that have a transparent ordered $c$-edge-colouring such that for any $i, j \in \{1,\dots,c\}$ with $i < j$ and for any edge $e$ of colour $i$, the matching number of the set of edges of colour $j$ that cross $e$ is at most $m$. \cref{introcirculardrawingscoloured} suggests that $\mathcal{G}_{c, m}$ might have bounded (as a function of $c$ and $m$) local treewidth. However, this is not true even for $m = 1$ and $c = 2$. For example, consider a geometric planar $(n \times n)$-grid\footnote{The \defn{$(n \times n)$-grid} is the graph with vertex set $\{1, \dots, n\} \times \{1, \dots, n\}$ where vertices $(v_{1}, v_{2})$ and $(u_{1}, u_{2})$ are adjacent whenever $|v_{1} - u_{1}| + |v_{2} - u_{2}| = 1$.}. Add a dominant vertex $v$ in the outerface that is adjacent to every vertex of the grid, let $G$ be the geometric graph obtained. So $G$ has radius $1$. Colour all the edges of the grid by $1$ and colour all the edges incident to $v$ by $2$. For every edge $e \in E(G)$ of colour $1$, all the edges of colour $2$ that cross $e$ are incident to $v$, so $G \in \mathcal{G}_{2, 1}$. But $\tw(G) = n + 1$ since the treewidth of the $(n \times n)$-grid is $n$ for $n \geqslant 2$ (see \citep[Lemma~20]{HW17} for a proof). Thus $\mathcal{G}_{2, 1}$ (and, as a consequence, $\mathcal{G}_{c, m}$ for every $c \geqslant 2$ and $m \geqslant 1$) does not have bounded local treewidth and, as a corollary, bounded layered treewidth and row treewidth (by \cref{LTWandRTW}).

We obtain the following upper bound on the treewidth of graphs in $\mathcal{G}_{c, m}$ that satisfy an additional property. Note that \cref{spanningtree} is a qualitative generalisation of \cref{circulardrawings,introcirculardrawingscoloured} and the classical result of \citet{RS-III} about the treewidth of planar graphs with bounded radius (\cref{PlanarBoundedRadius}).

\begin{thm} \label{spanningtree} Suppose that a topological graph $G$ has a transparent ordered $c$-edge-colouring $\phi$ such that:

\begin{itemize}

\item for any $i, j \in \{1,\dots,c\}$ with $i < j$, for any edge $e$ of colour $i$ and for any fragment $\gamma$ of $e$, the matching number of the set of edges of colour $j$ that cross $\gamma$ is at most $m$. 

\item $G$ has a spanning tree $T$ of radius $r$ such that every edge $e \in E(T)$ is involved in at most $t$ crossings with the edges of $G$ of colour less than $\phi(e)$.

\end{itemize}

Then $\tw(G) \in \mathcal{O}(((t + 1)r + c)cm)$. In particular, $\tw(G) \leqslant (6(t + 1)r + 3c - 1)(1 + 5(c - 1)m) - 1$.
\end{thm}

\begin{proof}

For each $e \in E(T)$, let $E_{e} \subseteq E(G^{\phi})$ be the set of edges between consecutive vertices of $W_{e}$. By \cref{walklength}, the length of $W_{e}$ is at most $2(t + 1)$. Let $E := \bigcup_{e \in E(T)} E_{e}$. Let $G_{T}$ be the subgraph of $G^{\phi}$ induced by $E$. Since $T$ has radius $r$, $G_{T}$ has radius at most $2(t + 1)r$. Since $V(T) = V(G)$, we have $V(G) \subseteq V(G_{T})$. By \cref{6}, for any $x \in V(G^{\phi})$, $\dist_{G^{\phi}}(x, v) \leqslant c - 1$ for some $v \in V(G_{T})$. By triangle inequality, $G^{\phi}$ has radius at most $2(t + 1)r + c - 1$. By \cref{PlanarBoundedRadius}, $\tw(G^{\phi}) \leqslant 6(t + 1)r + 3c - 2$. By the Coloured Planarisation Lemma (\cref{generallemma}\ref{CPLa}), $G$ is a minor of $G^{\phi} \boxtimes K_{1 + 5(c - 1)m}$. Thus $\tw(G) \leqslant \tw(G^{\phi} \boxtimes K_{1 + 5(c - 1)m}) \leqslant (\tw(G^{\phi}) + 1)(1 + 5(c - 1)m) - 1 \leqslant (6(t + 1)r + 3c - 1)(1 + 5(c - 1)m) - 1$.
\end{proof}

We have the following bound on the treewidth of circular $k$-matching-planar graphs.

\begin{cor} \label{outerkcoverplanar} Let $G$ be a circular $k$-matching-planar graph, where $k \geqslant 1$. Then $\tw(G) \in \mathcal{O}(k^3\log^{2}k)$. In particular, $\tw(G) \leqslant 9kc(c - 1) + 3c - 1$, where $c = 2(k + 1)\log_{2}(k + 1) + 2(k + 1)\log_{2}(\log_{2}(k + 1)) + 10k + 10$. 
\end{cor}

\begin{proof}

By assumption, no $k + 2$ edges of $G$ pairwise cross. 
A result of \citet{Davies22a} (about $\chi$-boundedness of circle graphs) is equivalent to saying that every circular graph with 
no $k + 2$ pairwise crossing edges has topological thickness at most $c$. 
Thus, $G$ has topological thickness at most $c$. The result follows from \cref{introcirculardrawingscoloured} (or \cref{circulardrawings}).
\end{proof}

Note that \cref{outerkcoverplanar} implies \cref{introcircularkcoverplanar}.

\section{Edge Colouring $k$-Matching-Planar Graphs} \label{SectionColouring}

This section proves \cref{introtopologicalthickness}, which bounds the topological thickness of certain topological $k$-matching-planar graphs and is an essential ingredient in the proofs of \cref{introRTW,introLTW} in \cref{FinalSection}.

The starting point is a bound on the edge density of $k$-cover-planar graphs. Although the definition of $k$-cover-planar graphs is introduced in this paper, a similar concept was briefly mentioned by \citet{AFPS14}. In particular, Rom Pinchasi proved the following bound on the number of edges in $k$-cover-planar graphs, where $d_{k} :=  \frac{3(k + 1)^{k + 1}}{k^k}$ for each integer $k \geqslant 0$ (see \citep[Lemma~4.1]{AFPS14}). Note that $d_{k} < 3e(k + 1)$. We include the proof for completeness. 

\begin{lem} [Rom Pinchasi; see \citep{AFPS14}] \label{edgeskcoverplanar} Every $k$-cover-planar graph on $n$ vertices has at most $d_{k}n$ edges.
\end{lem}
\begin{proof} 
Let $G$ be a topological $k$-cover-planar graph with $m := |E(G)|$. 
For each edge $uv\in E(G)$, let $X_{uv}$ be the set of edges of $G$ that cross $uv$, and are not incident to $\{u, v\}$. Since $G$ is $k$-cover-planar, $\tau(X_{uv}) \leqslant k$. Let $C_{uv}$ be a vertex cover of $X_{uv}$ with minimum size, so $|C_{uv}| = \tau(X_{uv}) \leqslant k$ and $\{u, v\} \cap C_{uv} = \emptyset$. Choose each vertex of $G$ independently with probability $p:=\frac{1}{k+1}$. Let $H$ be the subgraph of $G$ where $V(H)$ is the set of chosen vertices, and $E(H)$ is the set of edges $uv$ in $G$ such that $u$ and $v$ are chosen, but no vertex in $C_{uv}$ is chosen. Let $n^{*}$ and $m^{*}$ be the expected value of $|V(H)|$ and $|E(H)|$ respectively. By definition, $n^{*} = pn$. The probability that an edge $uv\in E(G)$ is in $H$ equals $p^{2}(1 - p)^{|C_{uv}|} \geqslant p^{2}(1 - p)^{k}$. Thus $m^{*} \geqslant p^{2}(1 - p)^{k}m$. Two edges in $H$ may cross only if they are incident to a common vertex. By the Hanani–Tutte Theorem, $H$ is planar  (see \citep{Tutte70} for example). Therefore, $p^{2}(1 - p)^{k}m \leqslant m^{*} \leqslant 3n^{*} = 3pn$, implying $m \leqslant \frac{3}{p(1 - p)^{k}}n =d_{k}n$.
\end{proof}

\cref{edgeskcoverplanar} and \cref{relation} immediately imply the following.

\begin{lem} [\citep{AFPS14}] \label{edgeskmatchingplanar} Every $k$-matching-planar graph on $n$ vertices has at most $d_{2k}n$ edges.
\end{lem}

As an aside, note that \cref{edgeskmatchingplanar} is useful for proving lower bounds. For example, suppose that $K_n$ is $k$-matching-planar. By \cref{edgeskmatchingplanar}, $\binom{n}{2}\leq d_{2k} n < 3e(2k+1)n$, implying $k\in\Omega(n)$. That is, in every topological $K_n$ there is an edge crossed by a matching of $\Omega(n)$ edges. This argument holds for any graph with $n$ vertices and $\Omega(n^2)$ edges.

We use \cref{edgeskmatchingplanar} to bound the arboricity and star arboricity of $k$-matching-planar graphs.

\begin{lem} \label{stararboricity} 
Every $k$-matching-planar graph $G$ has arboricity at most  $\lceil 2d_{2k} \rceil$ and star arboricity at most $2\lceil 2d_{2k} \rceil$.
\end{lem}

\begin{proof}
Let $n := |V(G)|$. By \cref{edgeskmatchingplanar}, $G$ has at most $d_{2k}n \leqslant \lceil 2d_{2k} \rceil (n - 1)$ edges assuming $n \geqslant 2$.  Every induced subgraph of $G$ is $k$-matching-planar. So by the Nash-Williams arboricity theorem \citep{NW-JLMS64}, $G$ is the union of $\lceil 2d_{2k} \rceil$ forests. Every forest is the union of two star-forests \citep{AA-DM89}. Thus $G$ is the union of $2\lceil 2d_{2k} \rceil$ star-forests.
\end{proof}

\cref{stararboricity} implies that to bound the topological thickness of a general topological $k$-matching-planar graph, it suffices to bound the topological thickness of a topological $k$-matching-planar star-forest. To do so, we employ the following definitions. A graph $J$ is called a \defn{string graph} if it is the intersection graph of a collection of continuous curves in the plane; that is, for each vertex $v \in V(J)$, there is a curve $\alpha_{v}$ in the plane such that distinct vertices $v, w$ are adjacent in $J$ if and only if $\alpha_{v}\cap \alpha_{w}\neq\emptyset$. Let $G$ be a topological graph. An edge $e$ of $G$ \defn{crosses} a component $S$ of $G$ if $e$ crosses an edge of~$S$. Distinct components $S_{1}$ and $S_{2}$ of $G$ \defn{cross} if an edge of $S_{1}$ crosses an edge of $S_{2}$. The \defn{component-crossing-graph} of $G$, denoted by \defn{$H_G$}, is the graph where the vertices of $H_{G}$ are the components of $G$, and two vertices of $H_{G}$ are adjacent if and only if the corresponding components of $G$ cross.

\begin{lem} \label{colouringk-matching-planargraphsstarforests} The topological thickness of every topological $k$-matching-planar star-forest $G_{0}$ such that no two edges incident to a common vertex cross is $\mathcal{O}(k^{2}\log k)$.
\end{lem}

\begin{proof}

The result is trivial if $k = 0$, so we assume that $k \geqslant 1$.

\begin{claim} \label{firstclaimcolouring} $H_{G_0}$ is $K_{12k^{2} + 3k + 2}$-free.
\end{claim}

\begin{proof}
    Let $t := 12k^{2} + 3k + 2$.
    Assume for the sake of contradiction that $K_{t}$ is contained in~$H_{G_{0}}$. Let $G$ be a minimal subgraph of $G_0$ such that the component-crossing-graph $H_{G}$ of the components of $G$ is isomorphic to $K_t$.

    Let $S_{1}, \dots, S_{t}$ be the components of $G$. Let $e \in E(S_{1})$ be an arbitrary edge. By minimality, there exists a component $S^{e} \in \{S_{2}, \dots, S_{t}\}$ such that $e$ crosses $S^{e}$, but no other edge of $S_{1}$ crosses $S^{e}$ (otherwise $H_{G-e}$ is isomorphic to $H_{G}$).

     Since $G_0$ is $k$-matching-planar, every edge of $S_{1}$ crosses at most $k$ of $S_{2}, \dots, S_{t}$. Since $H_{G}$ is isomorphic to $K_{t}$, the star $S_{1}$ has at least $\lceil\frac{t - 1}{k}\rceil = 12k + 4$ edges. Let $v$ be the centre of~$S_{1}$. Let $e_{1}, \dots, e_{12k + 4}$ be $12k + 4$ edges of $S_{1}$ in the counterclockwise order around $v$. By definition, $S_{1}$, $S^{e_{1}}, \dots, S^{e_{12k + 4}}$ are distinct.

    Let $a$ be the crossing point of $e_{1}$ and $S^{e_{1}}$ such that there are no crossing points of $e_{1}$ and $S^{e_{1}}$ between $a$ and $v$ (along $e_{1}$). As illustrated in \cref{colouring1},  let $\gamma_{1}$ be the subcurve of $e_{1}$ between $v$ and $a$ (green curve in \cref{colouring1}). Similarly,  let $b$ be the crossing point of $e_{6k + 3}$ and $S^{e_{6k + 3}}$ such that there are no crossing points of $e_{6k + 3}$ and $S^{e_{6k + 3}}$ between $b$ and $v$ (along $e_{6k + 3}$). Let $\gamma_{2}$ be the subcurve of $e_{6k + 3}$ between $v$ and $b$ (red curve in \cref{colouring1}). It follows from the definitions of $S^{e_{1}}$, $S^{e_{6k + 3}}$, $a$ and $b$ that no edge of $S_{1} \cup S^{e_{1}} \cup S^{e_{6k + 3}}$ crosses $\gamma_{1} \cup \gamma_{2}$.

    Since $a$ belongs to an edge of $S^{e_{1}}$, $b$ belongs to an edge of $S^{e_{6k + 3}}$, and $S^{e_{1}}$ crosses $S^{e_{6k + 3}}$, there exists a non-self-intersecting curve $\gamma_{3}$ with endpoints $a$ and $b$ such that: 

    \begin{itemize}
        \item $\gamma_{3} \subseteq S^{e_{1}} \cup S^{e_{6k + 3}}$,
        \item $\gamma_{3} \cap S^{e_{1}}$ is a subset of at most two edges of $S^{e_{1}}$ ($\gamma_{3} \cap S^{e_{1}}$ is blue in \cref{colouring1}),
        \item $\gamma_{3} \cap S^{e_{6k + 3}}$ is a subset of at most two edges of $S^{e_{6k + 3}}$ ($\gamma_{3} \cap S^{e_{6k + 3}}$ is purple in \cref{colouring1}).
    \end{itemize}

    Let $\alpha := \gamma_{1} \cup \gamma_{2} \cup \gamma_{3}$. Since  no edge of $S^{e_{1}} \cup S^{e_{6k + 3}}$ crosses $\gamma_{1} \cup \gamma_{2}$ and $\gamma_{3} \subseteq S^{e_{1}} \cup S^{e_{6k + 3}}$,  $\alpha$ is a Jordan curve. Let $F_{1}$ be the interior region in the plane bounded by $\alpha$ and $F_{2}$ be the exterior region in the plane bounded by $\alpha$.

 \begin{figure}
     \centering
          \scalebox{0.8}{\includegraphics{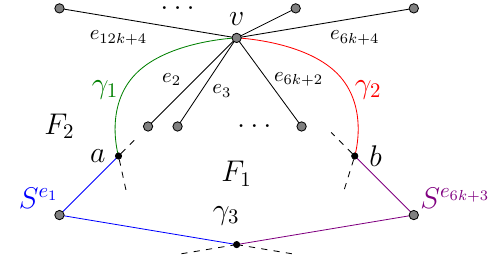}}
     \caption{Proof of \cref{firstclaimcolouring}. The vertices of $G$ are grey, the crossing points are black.}
     \label{colouring1}
 \end{figure}

    Let $E_{1} := \{e_{2}, \dots, e_{6k + 2}\}$, $E_{2} := \{e_{6k + 4}, \dots, e_{12k + 4}\}$, $\mathcal{S}_{1} := \{S^{e_{2}}, \dots, S^{e_{6k + 2}}\}$, and $\mathcal{S}_{2} := \{S^{e_{6k + 4}}, \dots, S^{e_{12k + 4}}\}$. Since $\gamma_{3} \subseteq S^{e_{1}} \cup S^{e_{6k + 3}}$, no edge of $E_{1} \cup E_{2}$ crosses $\gamma_{3}$. Since no two edges of $S_{1}$ cross and $\gamma_{1} \cup \gamma_{2} \subseteq e_{1} \cup e_{6k + 3}$, no edge of $E_{1} \cup E_{2}$ crosses $\alpha$. Without loss of generality, we can assume that for each $i \in \{1, 2\}$, the edges of $E_{i}$ lie in $F_{i}$. That is, each edge $e \in E_{i}$ lies in the interior of $F_{i}$ except for the endpoint $v$. Every edge $e$ of $S_{1}$ crosses $S^{e}$. Therefore, for every edge $e \in E_{i}$, there is a point in the interior of $F_{i}$ that belongs to $S^{e}$. Thus, for each star $S \in \mathcal{S}_{i}$, there exists a point of $S$ that lies in the interior of $F_{i}$. Since $H_{G}$ is complete, every star of $\mathcal{S}_{1}$ crosses every star of $\mathcal{S}_{2}$. Then there are at least $\min(|\mathcal{S}_{1}|, |\mathcal{S}_{2}|) = 6k + 1$ components of $G$ that cross $\alpha$. Since $G_0$ is $k$-matching-planar and there is a set of at most six edges of $G_0$ whose union contains $\alpha$, at most $6k$ components of $G$ cross $\alpha$, which is the desired contradiction.
\end{proof}

\begin{claim} \label{Kttstars} $H_{G_{0}}$ is $K_{16k^{2} + 3k + 1, 16k^{2} + 3k + 1}$-free.
\end{claim}

\begin{proof}
    The proof is analogous to the proof of \cref{firstclaimcolouring}. Let $t := 16k^{2} + 3k + 1$. Assume for the sake of contradiction that $K_{t, t}$ is contained in $H_{G_0}$, and let $G$ be a minimal subgraph of $G_0$ such that $K_{t,t}$ is contained in the component-crossing-graph $H_G$ of the components of~$G$. Let $\mathcal{T}_{1} := \{S_{1}, \dots, S_{t}\}$ and $\mathcal{T}_{2}$ be two sets of components of $G$ such that $|\mathcal{T}_{1}| = |\mathcal{T}_{2}| = t$, $\mathcal{T}_{1} \cap \mathcal{T}_{2} = \emptyset$, $\mathcal{T}_{1} \cup \mathcal{T}_{2}$ is the set of all components of $G$, and every star of $\mathcal{T}_{1}$ crosses every star of $\mathcal{T}_{2}$. By \cref{firstclaimcolouring}, $H_{G}$ is not isomorphic to $K_{2t}$. Without loss of generality, we can assume that the stars $S_{1}$ and $S_{2}$ do not cross.

    Let $e \in E(S_{1})$ be an arbitrary edge.  By minimality, there exists a star $S^{e} \in \mathcal{T}_{2}$ such that $e$ crosses $S^{e}$, but no other edge of $S_{1}$ crosses $S^{e}$ (otherwise $H_{G-e}$ is isomorphic to $H_G$).

    Since $G_0$ is $k$-matching-planar, every edge of $S_{1}$ crosses at most $k$ stars of $\mathcal{T}_{2}$. Since $S_{1}$ crosses every star of $\mathcal{T}_{2}$ and $|\mathcal{T}_{2}| = t$, the star $S_{1}$ has at least $\lceil\frac{t}{k}\rceil = 16k + 4$ edges. Let $v$ be the centre of $S_{1}$. Let $e_{1}, \dots, e_{16k + 4}$ be $16k + 4$ edges of $S_{1}$ in the counterclockwise order around $v$. By definition, the stars $S^{e_{1}}, \dots, S^{e_{16k + 4}}$ are distinct.

    Let $a$ be the crossing point of $e_{1}$ and an edge of $S^{e_{1}}$ such that there are no crossing points of $e_{1}$ and an edge of $S^{e_{1}}$ between $a$ and $v$ (along $e_{1}$). As illustrated in \cref{colouring2}, let $\gamma_{1}$ be the subcurve of $e_{1}$ between $v$ and $a$ (green curve in \cref{colouring2}). Similarly,  let $b$ be the crossing point of  $e_{8k + 3}$ and an edge of $S^{e_{8k + 3}}$ such that there are no crossing points of $e_{8k + 3}$ and an edge of $S^{e_{8k + 3}}$ between $b$ and $v$ (along $e_{8k + 3}$). Let $\gamma_{2}$ be the subcurve of $e_{8k + 3}$ between $v$ and $b$ (red curve in \cref{colouring2}). By definition of $S^{e_{1}}$, $S^{e_{8k + 3}}$, $a$ and $b$, no edge of $S_{1} \cup S^{e_{1}} \cup S^{e_{8k + 3}}$ crosses $\gamma_{1} \cup \gamma_{2}$. Since the stars $S_{1}$ and $S_{2}$ do not cross, no edge of $S_{1} \cup S_{2} \cup S^{e_{1}} \cup S^{e_{8k + 3}}$ crosses $\gamma_{1} \cup \gamma_{2}$.

    \begin{figure}
        \centering
        \scalebox{0.8}{\includegraphics{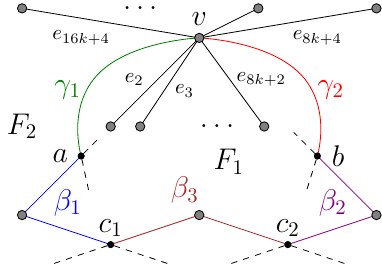}}
        \caption{Proof of \cref{Kttstars}. The vertices of $G$ are grey, the crossing points are black.}
        \label{colouring2}
    \end{figure}

    \begin{subclaim} \label{curvesubclaim} There exists a non-self-intersecting curve $\gamma_{3}$ with endpoints $a$ and $b$ such that $\gamma_{3} \subseteq S^{e_{1}} \cup S^{e_{8k + 3}} \cup S_{2}$ and for each $S \in \{S^{e_{1}}, S^{e_{8k + 3}}, S_{2}\}$, $\gamma_{3} \cap S$ is a subset of at most two edges of $S$.
    \end{subclaim}

    \begin{proof}

    If $S^{e_{1}}$ and $S^{e_{8k + 3}}$ cross then, by an argument similar to that used in the proof of \cref{firstclaimcolouring}, there exists a curve $\gamma_{3}$ that is a subset of $S^{e_{1}} \cup S^{e_{8k + 3}}$ and satisfies the conditions of this subclaim. 
    
    Now assume that $S^{e_{1}}$ and $S^{e_{8k + 3}}$ do not cross. Let $c_{1}$ be the crossing point of $S^{e_{1}}$ and $S_{2}$ such that there are no crossing points of $S^{e_{1}}$ and $S_{2}$ between $a$ and $c_{1}$ along the edges of $S^{e_{1}}$. Let $\beta_{1}$ be the curve with endpoints $a$ and $c_{1}$ that is a subset of at most two edges of $S^{e_{1}}$ (blue curve in \cref{colouring2}). Thus $\beta_{1}$ is not involved in crossings with $S^{e_{8k + 3}} \cup S_{2}$.

    Similarly, let $c_{2}$ be the crossing point of $S^{e_{8k + 3}}$ and $S_{2}$ such that there are no crossing points of $S^{e_{8k + 3}}$ and $S_{2}$ between $b$ and $c_{2}$ along the edges of $S^{e_{1}}$. Let $\beta_{2}$ be the curve with endpoints $b$ and $c_{2}$ that is a subset of at most two edges of $S^{e_{8k + 3}}$ (purple curve in \cref{colouring2}). Thus $\beta_{2}$ is not involved in crossings with $S^{e_{1}} \cup S_{2}$. In particular, $\beta_{1} \cap \beta_{2} = \emptyset$.

    Let $\beta_{3}$ be the curve with endpoints $c_{1}$ and $c_{2}$ that is a subset of at most two edges of $S_{2}$ (brown curve in \cref{colouring2}). By construction, $\beta_{3}$ does not cross $\beta_{1} \cup \beta_{2}$. Thus $\gamma_{3} := \beta_{1} \cup \beta_{2} \cup \beta_{3}$ is suitable.
    \end{proof}

    Let $\gamma_{3}$ be the subcurve given by \cref{curvesubclaim} and $\alpha := \gamma_{1} \cup \gamma_{2} \cup \gamma_{3}$. Since $\gamma_{3} \subseteq S^{e_{1}} \cup S^{e_{8k + 3}} \cup S_{2}$ and no edge of $S^{e_{1}} \cup S^{e_{8k + 3}} \cup S_{2}$ crosses $\gamma_{1} \cup \gamma_{2}$, $\alpha$ is a Jordan curve. Let $F_{1}$ be the interior region in the plane bounded by $\alpha$ and let $F_{2}$ be the exterior region in the plane bounded by $\alpha$.

    Let $E_{1} := \{e_{2}, \dots, e_{8k + 2}\}$, $E_{2} := \{e_{8k + 4}, \dots, e_{16k + 4}\}$, $\mathcal{S}_{1} := \{S^{e_{2}}, \dots, S^{e_{8k + 2}}\}$, and $\mathcal{S}_{2} := \{S^{e_{8k + 4}}, \dots, S^{e_{16k + 4}}\}$.  Since $\gamma_{3} \subseteq S^{e_{1}} \cup S^{e_{8k + 3}} \cup S_{2}$, no edge of $E_{1} \cup E_{2}$ crosses $\gamma_{3}$. Since no two edges of $S_{1}$ cross and $\gamma_{1} \cup \gamma_{2} \subseteq e_{1} \cup e_{8k + 3}$, no edge of $E_{1} \cup E_{2}$ crosses $\alpha$. Without loss of generality, we can assume that for each $i \in \{1, 2\}$, the edges of $E_{i}$ lie in $F_{i}$. That is, each edge $e \in E_{i}$ lies in the interior of $F_{i}$ except for the endpoint $v$. Every edge $e$ of $S_{1}$ crosses~$S^{e}$. Therefore, for every edge $e \in E_{i}$, there is a point in the interior of $F_{i}$ that belongs to $S^{e}$. Thus, for each star $S$ of $\mathcal{S}_{i}$, there exists a point of $S$ that lies in the interior of $F_{i}$. 

    Suppose that for each $i \in \{1, 2\}$, there exists a star $T_{i} \in \mathcal{S}_{i}$ that does not cross $\alpha$. Then $T_{i}$ lies in the interior of $F_{i}$. Since $T_{i} \in \mathcal{T}_{2}$, $T_{i}$ crosses every star of $\mathcal{T}_{1} \setminus \{S_{1}\}$. Note that $\mathcal{T}_{1} \setminus \{S_{1}\} \neq \emptyset$ because $k \geqslant 1$. Since each star of $\mathcal{T}_{1} \setminus \{S_{1}\}$ crosses $T_{1}$ and $T_{2}$, this implies that every star of $\mathcal{T}_{1} \setminus \{S_{1}\}$ crosses $\alpha$.  Since $G_0$ is $k$-matching-planar and there is a set of at most eight edges of $G_0$ whose union contains $\alpha$, at most $8k$ components of $G$ cross $\alpha$. Thus $|\mathcal{T}_{1}| - 1 \leqslant 8k$, a contradiction to $|\mathcal{T}_{1}| = t = 16k^{2} + 3k + 1$. 

    So there exists $i \in \{1, 2\}$ such that every star of $\mathcal{S}_{i}$ crosses $\alpha$. Then $|\mathcal{S}_{i}| \leqslant 8k$, a contradiction to $|\mathcal{S}_{i}| = 8k + 1$. Thus $H_{G_0}$ is $K_{t, t}$-free.
\end{proof}

We now complete the proof of \cref{colouringk-matching-planargraphsstarforests}. For $\varepsilon > \delta > 0$, for each vertex $v \in V(G_0)$ and edge $xy \in E(G_0)$, let $B_{v}^{\varepsilon} := \{p \in \mathbb{R}^{2} : \dist_{\mathbb{R}^{2}}(p, v) \leqslant \varepsilon\}$ and $C_{xy}^{\delta, \varepsilon} := \{p \in \mathbb{R}^{2} : \dist_{\mathbb{R}^{2}}(p, xy) \leqslant \delta\} \setminus (B_{x}^{\varepsilon} \cup B_{y}^{\varepsilon})$. Choosing $\varepsilon$ and $\delta$ to be sufficiently small, we may assume that:

\begin{itemize}
    \item $B_{v_{1}}^{\varepsilon} \cap B_{v_{2}}^{\varepsilon} = \emptyset$ for each pair of distinct vertices $v_{1}, v_{2}$ of $G_0$,
    \item $B_{v}^{\varepsilon} \cap C_{xy}^{\delta, \varepsilon} = \emptyset$ for each vertex $v$ and edge $xy$ of $G_0$,
    \item $C_{x_{1}y_{1}}^{\delta, \varepsilon} \cap C_{x_{2}y_{2}}^{\delta, \varepsilon} = \emptyset$ for every pair of non-crossing edges $x_{1}y_{1}$, $x_{2}y_{2}$ of $G_0$.
\end{itemize}

For each component $S$ of $G_0$, let $A_{S}^{\varepsilon, \delta} := (\bigcup_{v \in V(S)}B_{v}^{\varepsilon}) \cup (\bigcup_{e \in E(S)}C_{e}^{\delta, \varepsilon})$ and $\alpha_{S}$ be the boundary of $A_{S}^{\varepsilon, \delta}$. Observe that $\alpha_{S}$ is a Jordan curve. Thus, for every pair $S_{1}$, $S_{2}$ of distinct components of $G_0$, $\alpha_{S_{1}} \cap \alpha_{S_{2}} = \emptyset$ if and only if $S_{1}$ and $S_{2}$ do not cross.

Let $J$ be the string graph that corresponds to the set of curves $\{\alpha_{S} : S \text{ is a component of } G_0\}$. By \cref{Kttstars}, $J$ is $K_{16k^{2} + 3k + 1, 16k^{2} + 3k + 1}$-free. \citet{Lee17} proved that every $K_{t,t}$-free string graph is $\mathcal{O}(t\log t)$-degenerate. This implies that $\chi(J) \in \mathcal{O}(k^{2}\log k)$. For each component $S$ of $G_0$, colour $\alpha_{S}$ by one of $\mathcal{O}(k^{2}\log k)$ colours such that for any two components $S_{1}$ and $S_{2}$ of $G_0$, the curves $\alpha_{S_{1}}$ and $\alpha_{S_{2}}$ do not cross if they have the same colour. Colour each edge of $S$ by the colour of $\alpha_{S}$. Thus we obtain a transparent $\mathcal{O}(k^{2}\log k)$-edge-colouring of $G_0$.
\end{proof}

We now generalise from star-forests to general graphs.

\begin{thm} \label{colouringkcoverplanar} Let $G$ be a topological $k$-matching-planar graph such that for every vertex $v \in V(G)$, the set of edges incident to $v$ can be coloured with at most $s$ colours such that monochromatic edges do not cross. Then the topological thickness of $G$ is $\mathcal{O}(sk^{3}\log k)$.
\end{thm}

\begin{proof}
By \cref{stararboricity}, $G$ is the union of $2 \lceil 2d_{2k} \rceil$ star-forests. By assumption, $G$ is the union of a set $\mathcal{Q}$ of $2s \lceil 2d_{2k} \rceil \leqslant 2s \lceil 6e(2k + 1)  \rceil \leqslant 34s(2k + 1)$ star-forests, such that for each star-forest $F\in \mathcal{Q}$, no two edges in $F$ incident to a common vertex cross. The result follows from \cref{colouringk-matching-planargraphsstarforests} by taking a product colouring.
\end{proof}

\cref{colouringkcoverplanar} implies that the topological thickness of simple topological $k$-matching-planar graphs is $\mathcal{O}(k^{3}\log k)$. We wish to push the statement of \cref{colouringkcoverplanar} to the most general setting possible and prove  \cref{introtopologicalthickness}. To do this, we apply a result of \citet{RW19} about $\chi$-boundness of outerstring graphs. An \defn{outerstring graph} is the intersection graph of a collection of curves in a closed half-plane such that each curve $\alpha$ has exactly one point on the boundary of the half-plane and that point is an endpoint of $\alpha$.

\begin{lem} \label{starouterstring} A graph is outerstring if and only if it is the crossing graph of a topological star.
\end{lem} 

\begin{proof}

We first show that the crossing graph of a topological star $S$ is outerstring. Let $v$ be a centre of $S$. Let $D$ be a disc of radius $\varepsilon > 0$ centred at $v$. Choosing $\varepsilon$ to be sufficiently small, we may assume that no two edges of $S$ cross in $D$ and each edge of $S$ has exactly one intersection point with the boundary of $D$. Apply a Möbius transformation so that the boundary of $D$ maps to the boundary of a half-plane. The edges of $S$ in $\mathbb{R}^{2} \setminus \Int(D)$ transform into curves, and the crossing graph of $S$ is the intersection graph of these curves, and hence it is an outerstring graph.

Now we show that an outerstring graph $G$ is the crossing graph of a topological star. Let $\{\gamma_{v} : v \in V(G)\}$ be a collection of curves in a closed half-plane $B$ that corresponds to $G$ and $L$ be the boundary of $B$ such that every curve $\gamma_{v}$ has an endpoint $a_{v}$ in $L$. For sufficiently small $\varepsilon > 0$, redraw each curve $\gamma_{v}$ in the region $(\{p \in \mathbb{R}^{2} : \dist_{\mathbb{R}^{2}}(p, \gamma_{v}) < \varepsilon\} \cap \Int(B)) \cup \{a_{v}\}$ without creating new crossings and keeping the endpoint $a_{v}$ in $L$ such that: (i) every new curve has distinct endpoints, (ii) every new curve is non-self-intersecting, (iii) no three curves internally intersect at a common point, and (iv) all curves are pairwise distinct. Contract $L$ to a point and the curves transform into the edges of a topological star. Thus $G$ is isomorphic to the crossing graph of this topological star.
\end{proof}

Although the class of string graphs is not $\chi$-bounded~\citep{Pawlik14}, 
\citet{RW19} proved that the class of outerstring graphs is $\chi$-bounded. Specifically, they proved that $\chi(G) \in 2^{\mathcal{O}(2^{\omega(G)(\omega(G) - 1) / 2})}$ for every outerstring graph $G$. Applying \cref{starouterstring}, we conclude the following.

\begin{lem} \label{starcolouring} Every topological star with no $t$ pairwise crossing edges has topological thickness $2^{\mathcal{O}(2^{(t - 1)(t - 2) / 2})}$.
\end{lem}

\cref{colouringkcoverplanar} and \cref{starcolouring} imply the following result, which implies \cref{introtopologicalthickness}.

\begin{thm} \label{colouringtheorem} Every topological $k$-matching-planar graph with no $t$ pairwise crossing edges incident to a common vertex has topological thickness $(k + 1)^{3}\log_{2}(k + 2)2^{\mathcal{O}(2^{(t - 1)(t - 2) / 2})}$.
\end{thm}

\section{Weak Shallow Minors} \label{WeakShallowMinorsSection}

This section introduces weak shallow minors, which subsume and generalise shallow minors. The main result of this section (\cref{thm:RTWmain}) is a product structure theorem for weak shallow minors of the strong product of a graph with bounded Euler genus and a small complete graph. We use \cref{thm:RTWmain} to establish a product structure theorem for certain topological $k$-matching-planar graphs in \cref{FinalSection} (\cref{introRTW}).

We start with definitions. A model $\mu$ of a graph $G$ in a graph $H$ is \defn{$r$-shallow} if for each $v \in V(G)$, the radius of $H[\mu(v)]$ is at most $r$. A graph $G$ is an \defn{$r$-shallow minor} of a graph $H$ if there exists an $r$-shallow model of $G$ in $H$.

Let $H$ be a graph and $A \subseteq V(H)$. The \defn{weak diameter} of $A$ in $H$ is the maximum distance in $H$ between the vertices of $A$; that is, $\max\{\dist_{H}(u,v) : u,v \in A\}$. Weak diameter is an important concept in coarse graph theory~\citep{Hickingbotham25,Distel25,GP25,NSS25}, asymptotic dimension~\citep{BBEGLPS24,Distel23,Liu25}, and graph colouring~\citep{CL25,DN23}. We use the following variant of this definition. The \defn{weak radius} of $A$ in $H$ is the minimum non-negative integer $r$ such that for some  $v \in V(H)$ and for every $a \in A$ we have $\dist_{H}(v, a) \leqslant r$. Such a vertex $v$ is called an \defn{origin} of $A$. Weak diameter and weak radius are within a multiple of $2$ of each other.

We introduce the following definition\footnote{\citet[Observation~$6$]{Hickingbotham25} used a concept that is similar to weak shallow minors in relation to quasi-isometry of graphs.}.  A model $\mu$ of a graph $G$ in a graph $H$ is \defn{weak $r$-shallow} if for each $v \in V(G)$, the weak radius of $H[\mu(v)]$ in $H$ is at most $r$. We say that $G$ is a \defn{weak $r$-shallow minor} of $H$ if there exists a weak $r$-shallow model of $G$ in $H$. Every $r$-shallow minor of $H$ is a weak $r$-shallow minor of $H$. But the converse does not hold. For example, if $W_n$ is the $n$-vertex wheel, then $K_4$ is a weak 1-shallow minor of $W_n$ for every $n \geqslant 4$, but $K_4$ is not an $r$-shallow minor of $W_n$ for any fixed value of $r$ and sufficiently large $n$.

Intuitively speaking, if $G$ is a shallow minor of a graph $H$, then $G$ can be obtained from $H$ by contracting disjoint balls of bounded radius. So in some sense, $G$ inherits the structure of~$H$. It is natural to ask under what circumstances do weak shallow minors behave similarly.

\subsection{Weak Shallow Minors and Layered Treewidth}

\citet[Lemma~9]{DMW17} showed that shallow minors inherit bounded layered treewidth. In particular, for every graph $H$ and every $r$-shallow minor $G$ of $H$, $\ltw(G) \leqslant (4r + 1)\ltw(H)$. We generalise this result by showing that weak shallow minors inherit bounded layered treewidth. Our proof is based on the approach of \citet{DMW17}.

\begin{lem} \label{ltwWeakShallowMinor} For any graph $H$ and any weak $r$-shallow minor $G$ of $H$, $$\ltw(G) \leqslant (4r + 1)\ltw(H).$$
\end{lem}

\begin{proof}

Let $\ell := \ltw(H)$. So there is a tree decomposition $(T, B_{1})$ of $H$, and a layering $(V_{0}, V_{1}, \dots )$ of $H$, such that |$B_{1}(t) \cap V_{i}| \leqslant \ell$ for each $t \in V(T)$ and $i \geqslant 0$. Let $\mu$ be a weak $r$-shallow model of $G$ in $H$. For each $h \in V(H)$, let $X_{h} := \{v \in V(G) : h \in \mu(v)\}$. Since $\mu$ is a model, $|X_{h}| \leqslant 1$. Define $B_{2} : V(T) \rightarrow 2^{V(G)}$ by $B_{2}(t) := \bigcup_{h \in B_{1}(t)}X_{h}$ for each $t \in V(T)$. 

We now show that $(T, B_{2})$ is a tree decomposition of $G$. 
First, consider $vw \in E(G)$. Since $\mu$ is a model, $h_{1}h_{2} \in E(H)$ for some $h_{1} \in \mu(v)$ and $h_{2} \in \mu(w)$. Hence, there exists $t \in V(T)$ such that $h_{1}, h_{2} \in B_{1}(t)$. By construction, $v \in X_{h_{1}}$ and $w \in X_{h_{2}}$. Thus $v, w \in B_{2}(t)$. Second, consider $v \in V(G)$. Since $H[\mu(v)]$ is connected and for each $h \in \mu(v)$, $T[\{t \in V(T) : h \in B_{1}(t)\}]$ is a connected subtree of $T$, $T[\{t \in V(T) : v \in B_{2}(t)\}]$ is connected.

For each $v \in V(G)$, fix an origin $h_{v}$ of $\mu(v)$. So $\dist_{H}(h_{v}, a) \leqslant r$ for every $a \in \mu(v)$.
Since $\mu$ is a model, for each edge $vw \in E(G)$, we have $\dist_{H}(h_{v}, h_{w}) \leqslant 2r + 1$. So if $h_{v} \in V_{i}$ and $h_{w} \in V_{j}$ then $|i - j| \leq 2r + 1$. For each $i \geqslant 0$, let $V_{i}':= \{v \in V(G): h_{v} \in V_{(2r + 1)i} \cup \dots \cup V_{(2r + 1)(i + 1) - 1}\}$.  Hence, a partition of $G$ obtained from $(V_{0}', V_{1}', \dots)$ by excluding empty sets $V_{i}'$ is a layering of $G$.

We now bound $|B_{2}(t) \cap V'_{i}|$ for each $t \in V(T)$ and $i\geqslant 0$. Consider a vertex $v \in B_{2}(t) \cap V'_{i}$. So $h_{v} \in V_{j}$ for some $j\in \{ (2r + 1)i,\dots, (2r + 1)(i + 1) - 1\}$. By definition of $B_{2}$, there is a vertex $h' \in B_{1}(t) \cap \mu(v)$. By definition of $h_{v}$, we have $\dist_{H}(h_{v}, h') \leqslant r$. So $h' \in V_{j - r} \cup \dots \cup V_{j + r}$, implying $h' \in V_{(2r + 1)i - r} \cup \dots \cup V_{(2r + 1)(i + 1) - 1 + r}$. Therefore, $h'$ belongs to one of $((2r + 1)(i + 1) - 1 + r) - ( (2r + 1)i - r - 1)) = 4r + 1$ these layers. Since $h' \in B_{1}(t)$ and |$B_{1}(t) \cap V_{s}| \leqslant \ell$ for each $s \in \{(2r + 1)i - r, \dots, (2r + 1)(i + 1) - 1 + r\}$, there are at most $(4r + 1)\ell$ such vertices $h'$. Each such vertex $h'$ contributes at most one vertex (from $X_{h'}$) to $B_{2}(t) \cap V_{i}'$. So $|B_{2}(t) \cap V'_i| \leqslant (4r + 1)\ell$.  Thus $\ltw(G) \leqslant (4r + 1)\ell$.
\end{proof}

\subsection{Weak Shallow Minors and Row Treewidth}

\citet[Theorem~7]{HW24} showed that shallow minors inherit bounded row treewidth, in the sense that there is a function $f$ such that if a graph $G$ is an $r$-shallow minor of a graph $H$, then $\rtw(G) \leqslant f(\rtw(H),r)$. In light of \cref{ltwWeakShallowMinor}, it is natural to ask if a similar property holds for weak shallow minors. 

\begin{question} \label{rtwquestion}
 Does there exist a function $f$ such that if a graph $G$ is a weak $r$-shallow minor of a graph $H$, then $\rtw(G) \leqslant f(\rtw(H),r)$? 
\end{question}

We now set out to show that (perhaps surprisingly) the answer to \cref{rtwquestion} is ``no'' even when $\rtw(H) = 2$ and $r = 1$. 
The proof relies on the fact that the class of graphs of layered treewidth $1$ have unbounded row treewidth (\cref{LTWandRTWcomparison}). We start by characterising graphs of layered treewidth $1$. 

\begin{lem}
\label{ltw1}
A graph $G$ has layered treewidth $1$ if and only if there is a tree $T$ and a path $P$ such that $G$ can be obtained from $T \square P$ by first contracting edges of the form $(x,i)(y,i)$ where $xy \in E(T)$ and $i \in V(P)$; then deleting all remaining edges of the same form, and then deleting some vertices and edges. 
\end{lem}

\begin{proof}
First suppose that $\ltw(G) = 1$. So $G$ has a tree decomposition $(T, B)$ and a layering $(V_{1},V_{2}, \dots, V_{n})$ such that $|B(x) \cap V_{i}| \leqslant 1$ for each $x \in V(T)$ and $i \in \{1, \dots, n\}$. Consider $T \square P$ where $P$ is the path $(1,2,\dots,n)$. For each vertex $v$ of $G$, if $v \in  V_{i}$ and $xy \in E(T)$ with $v \in B(x) \cap B(y)$, then contract the edge $(x,i)(y,i)$ in $T\square P$. After these contractions, each vertex of $G$ is mapped to a single vertex. Delete the remaining edges of the form $(x,i)(y,i)$ where $xy \in E(T)$ and $i \in V(P)$. If $B(x) \cap V_i = \emptyset$ then delete vertex $(x,i)$. Now there is a $1$-$1$ map between $V(G)$ and the remaining vertices. For each edge $vw$ of $G$, there is a bag $B(x)$ containing both $v$ and $w$. Since $|B(x) \cap V_i| \leqslant 1$, $v$ and $w$ must be on distinct layers. So $B(x) \cap V_i = \{v\}$ and $B(x)\cap V_{i+1} = \{w\}$ for some $i\in\{1,\dots,n-1\}$ and node $x \in V(T)$. In the above construction, the edge $(x,i)(x,i+1)$ survives, $(x,i)$ is mapped to $v$, and $(x,i+1)$ is mapped to $w$. So $vw$ is present. Any unused edges can be deleted. 

Now suppose that $G$ can be obtained from $T \square P$ (for some tree $T$ and path $P = (1, \dots, n)$) by first contracting edges of the form $(x,i)(y,i)$ where $xy \in E(T)$ and $i \in V(P)$; then deleting all remaining edges of the same form, and then deleting some vertices and edges. Since the above contractions are of edges of the form $(x,i)(y,i)$ where $xy \in E(T)$ and $i \in V(P)$, each vertex of $G$ projects to a single vertex of $P$. Let $V_i$ be the set of vertices in $G$ that project to $i\in V(P)$. So $(V_1,\dots,V_n)$ is a layering of $G$. We now define a bag assignment $B : V(T) \rightarrow 2^{V(G)}$. For each node $x \in V(T)$, if $v$ is the vertex of $G$ mapped to the vertex obtained from $(x,i)$ after contractions, then put $v$ in the bag $B(x)$. For each vertex $v$ of $G$, the subgraph of $T$ induced by $\{x\in V(T): v\in B(x)\}$ is a connected subtree of $T$. Consider an edge $vw$ of $G$. Since non-contracted edges of the form $(x,i)(y,i)$ where $xy \in E(T)$ and $i \in V(P)$ are deleted, $v$ projects to $i\in V(P)$ and $w$ projects to $i+1\in V(P)$ for some $i\in\{1,\dots,n-1\}$. By definition of $T\square P$, there is a node $x\in V(T)$ such that $(x,i)$ is in the subtree of $T\times\{i\}$ corresponding to $v$, and $(x,i+1)$ is in the subtree of $T\times\{i+1\}$ corresponding to $w$. By construction, $v,w\in B(x)$. So $(T, B)$ is a tree decomposition of $G$. By construction, $|B(x) \cap V_i | \leqslant 1$ for each $x \in V(T)$ and $i \in \{1,\dots,n\}$. Thus $\ltw(G) = 1$. 
\end{proof}

A graph $J$ is an \defn{apex-forest} if $J - A$ is a forest for some $A \subseteq V(J)$ with $|A| \leqslant 1$.

\begin{lem}
\label{ltw1weakradius}
For every graph $G$ with layered treewidth at most $1$, there is an apex-forest $J$ and there is a path $P$, such that $G$ is a weak $1$-shallow minor of $J\square P$. 
\end{lem}

\begin{proof}
By \cref{ltw1}, there is a tree $T$ and a path $P$ such that $G$ can be obtained from $T \square P$ by first contracting edges of the form $(x,i)(y,i)$ where $xy \in E(T)$ and $i \in V(P)$; then deleting all remaining edges of the same form, and then deleting some vertices and edges. These operations define a model $\mu$ of $G$ in $T\square P$, such that each branch set of $\mu$ projects to a single vertex of $P$. Let $J$ be the apex-forest obtained from $T$ by adding a dominant vertex. Since each branch set of $\mu$ projects to a single vertex of $P$, its weak radius in $J \square P$ is at most~$1$. Thus $\mu$ is a weak $1$-shallow model of $G$ in $J \square P$. 
\end{proof}

\cref{ltw1weakradius} and \cref{LTWandRTWcomparison} together imply the following.

\begin{cor}
\label{bigrtw}
For every integer $n$ there is a graph $G$ with layered treewidth $1$ and row treewidth at least $n$, such that $G$ is a weak $1$-shallow minor of $J \square P$ for some apex-forest $J$ and path $P$.
\end{cor}

Since every apex-forest has treewidth at most $2$, \cref{bigrtw} shows that the answer to \cref{rtwquestion} is ``no'', even with $\rtw(H) = 2$ and $r = 1$.

\subsection{Weak Shallow Minors and Euler Genus}
\label{WeakShallowMinorsEulerGenus}

While the answer to \cref{rtwquestion} is ``no'' in general, the following theorem shows that the answer is ``yes'' in an important case, which we use to prove our product structure theorem for $k$-matching-planar graphs (\cref{introRTW}).

\begin{thm}\label{thm:RTWmain}
Let $r, g \geqslant 0$ and $c \geqslant 1$ be integers. Let $H$ be a graph of Euler genus $g$ and $G$ be a weak $r$-shallow minor of $H \boxtimes K_c$. Then 
$$\rtw(G)\leqslant (4r + 1)c((2(8r+1)c+3)(2g + 7)^{(6r+2)(2g + 5) - 4} - 1) - 1.$$
\end{thm}

The remainder of this section is devoted to proving \cref{thm:RTWmain}. We start with definitions. Let $T$ be a tree rooted at a node $r$. A node $a \in V(T)$ is a \defn{$T$-ancestor} of $x \in V(T)$ (and $x$ is a \defn{$T$-descendant} of $a$) if $a$ is contained in the path in $T$ with endpoints $r$ and $x$. If in addition $a \neq x$, then $a$ is a \defn{strict $T$-ancestor} of $x$. Every node of $T$ is a $T$-ancestor and a $T$-descendant of itself. A non-empty path $(x_{1}, \dots, x_{p})$ in $T$ is \defn{vertical} if for all $i \in \{1, \dots , p\}$ we have $\dist_{T}(x_{i}, r) = \dist_{T}(x_{1}, r) + i - 1$. The \defn{closure} of $T$ is the graph $J$ such that $V(J) := V(T)$ where $vw \in E(J)$ if and only if one of $v$ or $w$ is a strict $T$-ancestor of the other.

\begin{lem}
\label{FindKst} Let $r\geq 0$ and $s, t \geq 1$ be integers. 
Let $X_1,\dots,X_m$ be pairwise disjoint connected subgraphs of a graph $G$, where $m\geq (2rs+1)(t + s - 1) + 1$. Let 
$Y_1,\dots,Y_s$ be pairwise disjoint connected subgraphs of $G$, each with radius at most $r$. Assume that $V(X_i \cap Y_a)\neq\emptyset$ for each $i\in\{1,\dots,m\}$ and $a\in\{1,\dots,s\}$. Then $K_{s,t}$ is a minor of $G$.    
\end{lem}

\begin{proof}  We may assume that each $Y_a$ is a tree rooted at a vertex $y_a$, where each vertex in $Y_a$ is at distance at most $r$ from $y_a$. For each $i\in\{1,\dots,m\}$ and $a\in\{1,\dots,s\}$, fix a vertex $v_{i,a}$ in $X_{i}\cap Y_a$ at minimum distance from $y_a$ in $Y_a$. 

Let $H$ be the digraph with $V(H):=\{1,\dots,m\}$, where for distinct $i,j\in\{1,\dots,m\}$, we have $(i,j)\in E(H)$ if and only if, for some $a\in\{1,\dots,s\}$, some strict $Y_{a}$-ancestor of $v_{i,a}$ is in $X_j$. Each vertex $v_{i,a}$ has at most $r$ strict $Y_{a}$-ancestors. Thus, each vertex in $H$ has outdegree at most $rs$. Let $H'$ be the undirected graph underlying $H$. So $|E(H')|\leq |E(H)|\leq rsm$ and $H'$ has average degree at most $2rs$. By Turán's Theorem \citep{Turan41}, $H'$ has an independent set $I$ of size $\ceil{\frac{m}{2rs+1}}\geq t + s$.

For each $a\in\{1,\dots,s\}$, let $Y'_a$ be the subgraph of $Y_{a}$ induced by the union, taken over $i\in I$, of the $v_{i,a}y_a$-path in $Y_a$ excluding $v_{i,a}$. Since $X_{1}, \dots, X_{m}$ are pairwise disjoint, there exists at most one index $i_{a} \in \{1, \dots, m\}$ such that $v_{i_{a}, a} = y_{a}$. If there is no such index, define $i_{a} := 0$. For each $i \in I \setminus \{i_{a}\}$, we have $v_{i, a} \neq y_{a}$. So $Y'_a$ is non-empty and connected because $|I| \geqslant t + s \geqslant 2$. 

Suppose that $Y'_a$ contains a vertex $v$ in $X_i$, for some $a\in\{1,\dots,s\}$ and $i\in I$. By construction, $v$ is a strict $Y_{a}$-ancestor of $v_{j,a}$, for some $j\in I$. If $i=j$ then $v$ contradicts the choice of $v_{i,a}$. If $i \neq j$ then $(j,i)\in E(H)$, contradicting that $I$ is an independent set in $H'$. Hence $Y'_a$ is disjoint from $X_i$, for each $a \in \{1, \dots, s\}$ and $i\in I$. By construction, for each $a\in\{1,\dots,s\}$ and $i\in I \setminus \{i_{a}\}$, the parent of $v_{i,a}$ in $Y_a$ is in $Y'_a$. So $v_{i,a}$, which is in $X_i$, has a neighbour in $Y'_a$. Thus $V(Y_1'),\dots,V(Y_s')$ and $(V(X_i):i\in I \setminus \bigcup_{h \in \{1, \dots, s\}}\{i_{h}\})$ form a model of $K_{s,|I| - s}$ in~$G$. Since $|I| \geqslant t + s$, $K_{s, t}$ is a minor of $G$.   
\end{proof}

\begin{lem} \label{lem:closure} Let $T$ be a rooted tree and $H$ be a spanning subgraph of the closure of $T$. Let $B : V(T) \rightarrow 2^{V(H)}$ be defined as follows. For each $v \in V(T)$, let $B(v)$ be the set consisting of $v$ and all vertices $w \in V(H)$ such that $w$ is a strict $T$-ancestor of $v$ and $wx \in E(H)$ for some $T$-descendant $x$ of $v$. Then $(T,B)$ is a tree decomposition of $H$.    
\end{lem}

\begin{proof} First, consider an edge $vw \in E(H)$. Since $H$ is a subgraph of the closure of $T$, one of $v$ or $w$ is a strict $T$-ancestor of the other. Without loss of generality, $w$ is a strict $T$-ancestor of $v$. By definition, $v, w \in B(v)$ because $v$ is a $T$-descendant of itself.

Second, consider a vertex $w \in V(H)$. Consider any vertex $v \in V(T)$ such that $v \neq w$ and $w \in B(v)$. By definition, $w$ is a strict $T$-ancestor of $v$ and $wx \in E(H)$ for some $T$-descendant $x$ of $v$. Let $P_{vw}$ be the vertical path in $T$ with endpoints $v$ and $w$. For every vertex $v' \in P_{vw} \setminus \{w\}$, $x$ is a $T$-descendant of $v'$, and hence $w \in B(v')$. So all the vertices of $P_{vw}$ are in $\{t \in V(T) : w \in B(t)\}$. Thus $T[\{t \in V(T) : w \in B(t)\}]$ is connected.
\end{proof}

A \defn{path decomposition} of a graph is a tree decomposition $(T,B)$ where $T$ is a path. Define $P_{n}$ to be the graph with $V(P_n) := \{1, \dots, n\}$ and $E(P_n) := \{\{1, 2\}, \{2, 3\}, \dots, \{n - 1, n\}\}$.

\begin{lem}\label{lem:pathwidth}
Let $(P_n, B)$ be a path decomposition of a graph $H$ of width at most $t$. For each vertex $v\in V(H)$, let $\ell_v$ be the minimum index such that $v\in B(\ell_v)$, and let $X_v\subseteq V(H)$ be a set of vertices such that: $(i)$ $v\in X_v$, $(ii)$ $\ell_v\leq \ell_w$ for each $w\in X_v$, and $(iii)$ $H[X_v]$ is connected.  Let $s_{1}$ and $s_{2}$ be positive integers. Assume that for some vertex $v\in V(H)$ there are at least $(s_{1} + 2)(t + 1)^{s_{2}}$ distinct vertices $w\in V(H)$ such that $\ell_w\leq \ell_v$ and $H[X_w\cup X_v]$ is connected. Then there are subsets $S_1$ and $S_2$ of $V(H)$ such that $|S_{1}| \geqslant s_{1}$, $|S_{2}| \geqslant s_{2}$, and for each $w\in S_1$, we have $S_2\subseteq X_w$.
\end{lem}

\begin{proof}
    Let $Z$ be the set of vertices $w\in V(H)$ such that $\ell_w\leq \ell_v$ and $H[X_w\cup X_v]$ is connected. So $|Z| \geqslant (s_1 + 2)(t + 1)^{s_2}$.
    For each $w\in Z$, consider the set $I_{w}$ of indices $i \in \{1, \dots, n\}$ such that $B(i) \cap X_w \neq \emptyset$. Since $H[X_{w}]$ is connected and $w \in X_{w} \cap B(\ell_{w})$, $I_{w}$ is an interval in $(1, \dots, n)$ that contains $\ell_w$. By definition of $X_{v}$, we have $X_{v} \cap (B(1) \cup \dots \cup B(\ell_{v} - 1)) = \emptyset$. Since $H[X_{w} \cup X_{v}]$ is connected and by the edge-property of the path decomposition $(P_n, B)$, we have $\ell_{v} \in I_{w}$. So $I_{w}$ is an interval in $(1, \dots, n)$ that contains both $\ell_w$ and $\ell_v$. Thus $X_w$ forms a hitting set for the bags $B(\ell_w),B(\ell_w+1),\dots, B(\ell_v)$.
   
    For each $w \in Z$, let $Z_w\subseteq X_w$ be a minimal hitting set for the bags $B(\ell_w),B(\ell_w+1),\dots, B(\ell_v)$.
    Label the vertices of $Z_w$ by $z_{w,1},z_{w,2},\dots, z_{w,|Z_w|}$ so that $\ell_{z_{w,i}}\leq \ell_{z_{w,j}}$ whenever $i\geq j$. By definition of $Z_{w}$, there exists $i' \in \{1, \dots, |Z_{w}|\}$ such that $z_{w, i'} \in B(\ell_{v})$. Suppose for the sake of contradiction that $i' \neq 1$. Then, since $\ell_{z_{w, i'}} \leqslant \ell_{z_{w, 1}}$, we have that $z_{w, i'}$ hits all the bags of $(B(\ell_w),B(\ell_w+1),\dots, B(\ell_v))$ that are hit by $z_{w, 1}$. Thus $Z_{w} \setminus \{z_{w, 1}\}$ is also a hitting set for the bags $B(\ell_w),B(\ell_w+1),\dots, B(\ell_v)$, a contradiction to the minimality of $Z_{w}$. Thus $z_{w, 1} \in B(\ell_{v})$. By a similar inductive argument that uses the minimality of $Z_{w}$,  we have that $z_{w,i+1}\in B(\ell_{z_{w,i}}-1)$ for each $i \in \{1, \dots, |Z_{w}| - 1\}$.

    For each positive integer $c$, let $\mathcal{S}_c$ be the set of sequences $(v_1,\dots ,v_c)$ of vertices of $H$ such that $v_1\in B(\ell_v)$ and for each $i\in \{1,\dots,c - 1\}$ we have $v_{i+1}\in B(\ell_{v_i}-1)$. Let $\mathcal{S}'_c$ be the set of sequences $(v_1,\dots ,v_{c + 1})$ of vertices of $H$ such that $(v_1,\dots, v_c)\in \mathcal{S}_c$ and $v_{c+1}\in B(\ell_{v_{c}})$. By the observations above, $(z_{w,1},z_{w,2},\dots, z_{w,|Z_w|}) \in \mathcal{S}_{|Z_{w}|}$ for each $w \in Z$. Since $z_{w, |Z_{w}|} \in X_{w}$, we have $\ell_{z_{w, |Z_{w}|}} \geqslant \ell_{w}$. Since $Z_w \cap B(\ell_{w}) \neq \emptyset$ and by definition of the ordering $z_{w,1},z_{w,2},\dots ,z_{w,|Z_w|}$, we have $z_{w, |Z_{w}|} \in B(\ell_{w})$ and $\ell_{z_{w, |Z_{w}|}} = \ell_{w}$. Therefore $w \in B(\ell_{w})=B(\ell_{z_{w},|Z_w|})$.    
    Thus $(z_{w,1},z_{w,2},\dots, z_{w,|Z_w|},w) \in \mathcal{S}'_{|Z_w|}$ because $(z_{w,1},z_{w,2},\dots, z_{w,|Z_w|}) \in \mathcal{S}_{|Z_{w}|}$.

    For positive integers $c\geq c_{0}$, $\mathcal{S}_{c_{0}}$ is exactly the set of prefixes of sequences in $\mathcal{S}_c$ of length~$c_0$.
    By construction, $|\mathcal{S}_1|=|B(\ell_v)|\leq t + 1$. By induction, $|\mathcal{S}_c|\leq (t+1)^c$ for each integer $c \geqslant 1$.
    Similarly, $|\mathcal{S}'_c|\leq (t+1)^{c+1}$.
    There are at most $\sum_{i=1}^{s_2-1}|\mathcal{S}'_i|$ vertices $w \in Z$ with $|Z_{w}| < s_2$.
    Since $|Z|\geq s_{1}(t + 1)^{s_{2}} + 2(t + 1)^{s_{2}} > s_1|\mathcal{S}_{s_2}|+\sum_{i=1}^{s_2-1}|\mathcal{S}'_i|$, there are at least $s_1|\mathcal{S}_{s_{2}}|$ vertices $w \in Z$ such that $|Z_w|\geq s_2$. Therefore there is some set $S_{2} := (v_1,\dots,v_{s_{2}})\in \mathcal{S}_{s_{2}}$ and some set $S_1 \subseteq Z \subseteq V(H)$ of size at least $s_1$ such that for each $w\in S_1$, we have $(z_{w,1}, z_{w, 2}, \dots,  z_{w, s_{2}})=S_{2}$ and hence $S_{2} \subseteq X_w$, as desired.
    \end{proof}

    For an integer $t \geqslant 1$, a \defn{$t$-tree} is an edge-maximal graph of treewidth $t$. Let $T$ be a rooted tree. For each node $x \in V(T)$, define
$$T_{x} := T[\{y \in V(T) : y \text{ is a } T\text{-descendant of }x\}]$$
to be the maximal subtree of $T$ rooted at $x$. We make use of the following well-known normalisation lemma (see \citep[Lemma~8]{DMW23} for a proof).

\begin{lem}\label{lem:normaltd}
    For every graph $H$, there is a rooted tree $T$ with $V(T)=V(H)$ and a tree decomposition $(T,B)$ of width $\tw(H
    )$ such that:
    \begin{enumerate}
        \item $\{v\}\subseteq \{w\in V(T):v\in B(w)\}\subseteq V(T_v)$ for every vertex $v\in V(H)$, and consequently
        \item for every edge $vw\in E(H)$, one of $v$ or $w$ is a strict $T$-ancestor of the other.
    \end{enumerate}
\end{lem}
A tree decomposition as in \cref{lem:normaltd} is said to be \defn{normal}.

\begin{lem}\label{lem:normalttree}
Let $t$ be a positive integer, let $H$ be a $t$-tree, let $(T,B)$ be a normal tree decomposition of $H$, and let $P$ be a vertical path in $T$. 
 Then:
\begin{enumerate}[(a)]
    \item\label{noramlttree1} for the function $B_P:V(P)\to 2^{V(P)}$ where $B_P(w):=B(w)\cap V(P)$ for all $w\in V(P)$, $(P,B_P)$ is a path decomposition of $H[V(P)]$,
    \item\label{noramlttree2} for every $v\in V(P)$ and every connected subgraph $H'\subseteq H[V(T_v)]$, the subgraph of $H$ induced by $V(P)\cap V(H')$ is connected, and
    \item\label{noramlttree3} for every connected subgraph $H'\subseteq H$ and every vertex $v$ such that $v$ has a strict $T$-ancestor and a $T$-descendant in $H'$, $H[V(H')\cup \{v\}]$ is connected.
\end{enumerate}
\end{lem}
\begin{proof}
    To prove \ref{noramlttree1},  first observe that since $T[\{h \in V(T):w\in B(h)\}]$ is connected for each $w\in V(P)$, the graph $P[\{h\in V(P):w\in B_P(h)\}]=T[V(P)\cap \{h\in V(T):w\in B(h)\}]$ is also connected.
    Now consider an edge $vw\in E(H[V(P)])$. Since $(T,B)$ is normal, we can assume without loss of generality that $w$ is a strict $T$-ancestor of $v$.
    Let $t_0\in V(T)$ be such that $v, w\in B(t_0)$. Since $(T,B)$ is normal, $t_0\in V(T_v)\cap V(T_w)=V(T_v)$.
    Since $w\in B(w)$ and $T[\{h\in V(T):w\in B(h)\}]$ is connected, we have $w\in B(v)$, and so $\{v,w\}\subseteq B(v)\cap V(P)=B_P(v)$, which completes the proof of~\ref{noramlttree1}.

    To prove \ref{noramlttree2}, suppose for the sake of contradiction that there is some $v\in V(P)$ and some connected subgraph $H'\subseteq H[V(T_v)]$ such that $H[V(P)\cap V(H')]$ is not connected.
    Thus there is a path $Q$ in $H'$ between distinct vertices $u$ and $w$ in $V(P)\cap V(H')$ with no internal vertices in $P$, such that $w$ is a strict $T$-ancestor of $u$ and $uw\notin E(H)$.
    Let $E^*$ be the set of edges of $Q$ whose endpoints lie in distinct components of $T-E(P)$.
    Consider $u'w'\in E^*$ with $w'$ a strict $T$-ancestor of $u'$.
    Since $u'$ and $w'$ are in distinct components of $T-E(P)$, $w'$ is also a $T$-ancestor of a vertex in $P$. 
    Since $w'\in V(Q)\subseteq V(T_v)$, $w'$ is a $T$-descendant of $v$, and hence $w'\in V(P)$.
    Since $Q$ has no internal vetex in $P$, we have $w'\in \{u,w\}$.

    Now consider an edge $u''w''$ in $E(Q)$ with exactly one endpoint $u''$ in $T_u$. Such an edge must exist since $Q$ has exactly one endpoint in $T_u$. By definition, $u''w'' \in E^*$.
    Since $(T,B)$ is normal and $u''w'' \in E^* \subseteq E(H)$, $w''$ is a strict $T$-ancestor of $u''$.
    Thus $w''\in\{u,w\}\setminus V(T_u)$ by the argument in the previous paragraph, implying $w''=w$.
    Let $u^*$ be a vertex such that $w, u''\in B(u^*)$. Since $(T,B)$ is normal, $u^*$ is a $T$-descendant of $u''$ and hence of $u$.
    Thus $w\in B(u^*)\cap B(w)$ and so $u, w \in B(u)$.
    Hence $(T,B)$ is a tree decomposition of the graph obtained by adding $uw$ to $H$, contradicting the fact that $H$ is a $t$-tree.

    To prove \ref{noramlttree3}, consider a connected subgraph $H'\subseteq H$ and a vertex $v$ such that $V(H')$ contains both a strict $T$-ancestor and a $T$-descendant of $v$.
    In particular, $H'$ contains an edge $uw$ such that $u\in V(T_v)$ and $w\notin V(T_v)$.
    Since $(T,B)$ is normal, $w$ is a strict $T$-ancestor of $u$ and $w, u \in B(u')$ for some $u'\in V(T_u)\subseteq V(T_v)$.
    Additionally, $w\in B(w)$ and $v\in B(v)$, and so $w, v \in B(v)$.
    Since $H$ is a $t$-tree, $vw\in E(H)$, and hence $H[V(H')\cup \{v\}]$ is connected.
\end{proof}

    \begin{lem}\label{lem:touchgraphtw2}
        Let $t$ and $z$ be positive integers and let $(T,B)$ be a normal tree decomposition of a $t$-tree $H$. For each $i\in \{1, \dots, z\}$, let $X_i\subseteq V(H)$ be a set of vertices such that $H[X_i]$ is connected.
        Let $H^*$ be the graph with vertex set $V(H)$ such that distinct vertices $v$ and $w$ are adjacent in $H^*$ if and only if there exist $i,j\in \{1, \dots, z\}$ such that $\{v\}\subseteq X_i\subseteq V(T_v)$, $\{w\}\subseteq X_j\subseteq V(T_w)$, and $H[X_i\cup X_j]$ is connected. Then for any integers $s_1 \geqslant 1$ and $s_2 \geqslant 2$ at least one of the following holds:
        \begin{enumerate}
            \item $H^*$ has treewidth at most $(s_1+2)(t+1)^{s_2} - 2$, or
            \item there are subsets $S_1$ and $S_2$ of $V(H)$ such that $|S_1| \geqslant s_1$, $|S_2| \geqslant s_2$, and for each $v\in S_1$ there exists $i\in \{1, \dots, z\}$ such that $\{v\}\subseteq X_i\subseteq V(T_v)$ and $S_2 \subseteq X_i$.
        \end{enumerate}
        \end{lem}
 \begin{proof} Since $(T,B)$ is normal, $H$ is a spanning subgraph of the closure of $T$. For every edge $vw\in E(H^*)$, we have that some vertex in $T_v$ is either in $T_w$ or adjacent in $H$ to a vertex in $T_w$.
    It follows that one of $v$ or $w$ is a strict $T$-ancestor of the other, meaning $H^*$ is also a spanning subgraph of the closure of $T$.
    Define $B^{*} : V(T) \rightarrow 2^{V(H^{*})}$ as follows. For each $v \in V(T)$, let $B^*(v)$ be the set consisting of $v$ and all vertices $w\in V(H^*)$ such that $w$ is a strict $T$-ancestor of $v$ and $wx \in E(H^*)$ for some $T$-descendant $x$ of $v$. By \cref{lem:closure}, $(T,B^{*})$ is a tree decomposition of~$H^*$.

    If every bag of $(T, B^*)$ has size at most $(s_1+2)(t+1)^{s_2} - 1$, then the first outcome of the lemma is satisfied. Otherwise, there exists $v\in V(H)$ such that $|B^*(v)| \geqslant (s_1+2)(t+1)^{s_2}$. Let $P$ be the vertical path in $T$ from $v$ to the root of $T$, let $H':=H[V(P)]$ and let $(P,B_P)$ be the path decomposition of $H'$ described in Lemma~\ref{lem:normalttree}\ref{noramlttree1}. Note that $(P,B_P)$ has width at most $t$.

        For each $i\in \{1, \dots, z\}$, let $X_{v,i}:=X_i\cap V(P)$. Since $H[X_i]$ is connected, $H'[X_{v,i}]$ is connected by Lemma~\ref{lem:normalttree}\ref{noramlttree2}.
        Now, consider a vertex $w\in B^*(v) \setminus \{v\}$. By definition of $B^*$, $w$ is a strict $T$-ancestor of $v$ (and so $w \in V(H')$) and $wx \in E(H^*)$ for some $T$-descendant $x$ of $v$. By definition of $H^*$, there exist $i_w, j_w\in \{1, \dots, z\}$ such that $\{w\}\subseteq X_{i_w}\subseteq V(T_w)$, $\{x\}\subseteq X_{j_w}\subseteq V(T_x) \subseteq V(T_v)$, and $H[X_{i_w}\cup X_{j_w}]$ is connected.
        Let $X'_w:=X_{v,i_w}$, so $H'[X'_w]$ is connected. 
        By Lemma~\ref{lem:normalttree}\ref{noramlttree3}, since $x\in V(T_v)$ and $w$ is a strict $T$-ancestor of $v$, we have that $H[X_{i_w}\cup X_{j_w}\cup \{v\}]$ is connected.
        Since $X_{j_w}\subseteq V(T_v)$, we have $X_{j_w}\cap V(P)\subseteq \{v\}$ and so $(X_{i_w}\cup X_{j_w}\cup \{v\})\cap V(P)=X'_w\cup \{v\}$.
        Hence $H'[X'_w\cup \{v\}]$ is connected by \cref{lem:normalttree}\ref{noramlttree2}.

        For every other vertex $w$ of $H'$ (that is, for every $w \in (V(H') \setminus B^*(v)) \cup \{v\}$), define $X'_w:=\{w\}$. We wish to apply \cref{lem:pathwidth}. Let $n := |V(H')|$. Recall that $P_n$ is the graph defined before the statement of \cref{lem:pathwidth}. Associate every vertex $x$ of $P$ to a positive integer $\dist_{T}(x, r) + 1 \in \{1, \dots, n\}$, where $r$ is the root of $T$. Let $\tilde{B} : \{1, \dots, n\} \rightarrow 2^{V(H')}$ be a bag assignment obtained from $B_{P}$ using this association. So $(P_n, \tilde{B})$ is a path decomposition of $H'$ of width at most $t$. We now check the conditions of \cref{lem:pathwidth} for the path decomposition $(P_n, \tilde{B})$ of $H'$ and the collection of sets $(X'_{u} : u \in V(H'))$. For every $u \in V(H')$, $u \in X'_{u}$ and $H'[X'_{u}]$ is connected, and hence conditions (i) and (iii) of \cref{lem:pathwidth} are satisfied. By definition, every vertex $w' \in V(X'_{u})$ is a $T$-descendant of $u$. Then, since $(T,B)$ is normal, condition (ii) is satisfied.

        By definition of $B^*$, all the vertices $w$ of $B^*(v)$ are $T$-ancestors of $v$. For every such vertex $w$, the graph $H'[X'_w \cup X'_v]$ is connected because $H'[X'_w \cup \{v\}]$ is connected and $\{v\} = X'_v$. Recall that $|B^{*}(v)| \geqslant (s_{1} + 2)(t + 1)^{s_2}$. Now, by \cref{lem:pathwidth} applied to the path decomposition $(P_n, \tilde{B})$ of $H'$ and the collection of sets $(X'_{u} : u \in V(H'))$, there are subsets $S_1$ and $S_2$ of $V(H')$ such that $|S_1| \geqslant s_1$, $|S_2| \geqslant s_2$, and for each $w\in S_1$, we have $S_2\subseteq X'_w$. If $w \in B^*(v) \setminus \{v\}$ then $\{w\} \subseteq X_{i_w} \subseteq V(T_w)$ and $S_2 \subseteq X'_w \subseteq X_{i_w}$. Otherwise, $w \in (V(H') \setminus B^*(v)) \cup \{v\}$ and $S_2 \subseteq \{w\}$, but this is impossible because $|S_2| \geqslant s_2 \geqslant 2$. Thus at least one of the outcomes of the lemma is satisfied.
    \end{proof}

For integers $t \geqslant 1$ and $y \geqslant 0$, a graph $J$ is \defn{$(t,y)$-good} if there is graph $H$ of treewidth at most $t$ and a path $P$ such that there is a subgraph $J'$ of $H\boxtimes P$ isomorphic to $J$, and for all but at most $y$ vertices $v$ of $H$, $J'[(\{v\}\times V(P))\cap V(J')]$ is a non-empty path.

We now show that, under certain conditions, weak shallow minors inherit product structure.

\begin{lem}\label{lem:RTWtechnical}
Let $r$ and $y$ be non-negative integers and $t, a, b$ and $c$ be positive integers. Let $J$ be a $K_{a, b}$-minor-free $(t,y)$-good graph. If $G$ is a weak $r$-shallow minor of $J\boxtimes K_c$, then $$\rtw(G)\leqslant (4r + 1)c(((8r+1)c(a - 1)+3)(t+1)^{y+(2ra+1)(a + b-1)+1} - 1) - 1.$$ 
\end{lem}
\begin{proof}
By the definition of $(t,y)$-good, there is a graph $H$ of treewidth at most $t$, a path $P$, and a subgraph $J'$ of $H\boxtimes P$ isomorphic to $J$ such that for all but at most $y$ vertices $v$ of $H$, the set $(\{v\}\times V(P))\cap V(J')$ induces a non-empty path of $J'$. We may assume that $H$ is a $t$-tree. By \cref{lem:normaltd}, there exists a normal tree decomposition $(T,B)$ of $H$.
Let $\mu$ be a weak $r$-shallow model of $G$ in $J'\boxtimes K_c$, and let $g_1:V(G)\to V(H)$, $g_2:V(G) \rightarrow V(P)$ and $g_3:V(G)\to K_c$ be functions such that for all $v\in V(G)$ we have $(g_1(v),g_2(v),g_3(v))\in \mu(v)$ and $\mu(v)\subseteq V(T_{g_1(v)})\times V(P)\times V(K_c)$.
For each $v\in V(G)$, define $X_v$ to be the projection of $\mu(v)$ to $V(H)$.
Note that $\{g_{1}(v)\} \subseteq X_v\subseteq V(T_{g_1(v)})$ and $H[X_v]$ is connected because the subgraph of $J'\boxtimes K_c$ induced by $\mu(v)$ is connected.

Consider an edge $vw\in E(G)$. Since $\mu$ is a model, $(J'\boxtimes K_c)(\mu(v)\cup \mu(w))$ is connected. Hence, $H[X_v\cup X_w]$ is connected. Observe that $\dist_{P}(g_2(v), g_2(w)) \leqslant \dist_{J' \boxtimes K_c}((g_1(v),g_2(v),g_3(v)), (g_1(w),g_2(w),g_3(w))) \leqslant 4r + 1$; the second inequality holds because $\mu$ is a weak $r$-shallow model. Thus $g_2(v)g_2(w)\in E(P^{4r+1})$.

Define $H^*$ to be the graph with vertex set $V(H)$ such that distinct vertices $v'$ and $w'$ are adjacent in $H^*$ if and only if there are vertices $v,w\in V(G)$ such that $g_1(v)=v'$, $g_1(w)=w'$ and $H[X_v\cup X_w]$ is connected. It follows from the above observations that the map $v \rightarrow (g_1(v), g_2(v), g_3(v))$ is an injective homomorphism from $G$ to $H^*\boxtimes P^{4r+1} \boxtimes K_c$. Thus $G$ is contained in $H^*\boxtimes P^{4r+1} \boxtimes K_c$. Since $P^{4r+1}$ is contained in $P\boxtimes K_{4r+1}$, $\rtw(G) \leqslant \tw(H^*\boxtimes K_{4r+1}\boxtimes K_c)\leq (4r+1)c(\tw(H^*)+1)-1$.

Let $s_1 := (8r+1)c(a - 1) + 1$ and $s_2 := y+(2ra+1)(a + b-1)+1$. Note that $s_2 \geqslant 2$. By \cref{lem:touchgraphtw2}, at least one of the following holds: (i) $\tw(H^*) \leqslant (s_1+2)(t+1)^{s_2} - 2$, or (ii) there is a set $S_1\subseteq V(H)$ of size at least $s_1$ and a set $S_2\subseteq V(H)$ of size at least $s_2$ such that for each $v' \in S_1$, there is some $v\in V(G)$ such that $g_1(v)=v'$ and $S_2\subseteq X_v$. If (i) holds, then we are done.

Our goal is to show that the outcome (ii) does not hold. Assume for the sake of contradiction that such sets $S_1$ and $S_2$ exist. Let $S'_1$ be a minimal subset of $V(G)$ such that $g_1(S'_1)=S_1$ and $S_2\subseteq X_v$ for all $v\in S'_1$. So $|S'_1| = |S_1| \geqslant s_1$. Fix an arbitrary vertex $s_0\in S_2$. For each $v\in S'_1$, define $\ell(v):=(\ell_1(v),\ell_2(v))\in V(P \boxtimes K_c)$ such that $(s_0,\ell_1(v),\ell_2(v))\in \mu(v)$. Since $\mu$ is a model, the map $\ell$ is injective.
Since $\mu$ is a weak $r$-shallow model,
for each $v\in S'_1$ there is a tree $U_v\subseteq J'$ of radius at most $r$ such that $\mu(v)\subseteq V(U_v)\times V(K_c)$. If $v,w\in S'_1$ are two vertices such that $U_v$ and $U_w$ intersect, then $\dist_{P}(\ell_1(v), \ell_1(w)) \leqslant 4r$. Since $\ell$ is an injection, for each $v \in S'_1$, there are at most $(8r + 1)c - 1$ vertices $w \in S'_1$ such that $U_v \cap U_w \neq \emptyset$ and $w \neq v$. By a greedy algorithm, there is a set $I_1\subseteq S'_1$ of size at least $\lceil \frac{s_1}{(8r + 1)c} \rceil \geqslant a$ such that the trees in $\{U_v:v\in I_1\}$ are pairwise vertex-disjoint. 
By the definition of $(t,y)$-good, there is a set $I_2\subseteq S_2$ of size at least $s_2 - y = (2ra+1)(a + b-1)+1$ such that for each $v\in I_2$ the set $(\{v\}\times V(P))\cap V(J')$ induces a non-empty path $Q_v$ of $J'$.
Thus $\{U_v:v\in I_1\}$ is a collection of pairwise disjoint connected subgraphs of $J'$, each with radius at most $r$, and $\{Q_v:v\in I_2\}$ is a collection of pairwise disjoint connected subgraphs of $J'$. For each $v\in I_1$ and $w\in I_2$, we have $w \in S_2 \subseteq X_v$. So $Q_{w}$ hits the projection of $\mu(v)$ to $H \boxtimes P$. Since $\mu(v)\subseteq V(U_v)\times V(K_c)$, the projection of $\mu(v)$ to $H \boxtimes P$ lies in $V(U_v)$. Thus $V(U_v \cap Q_w)\neq\emptyset$.
By \cref{FindKst}, $K_{a,b}$ is a minor of $J'$, a contradiction.
\end{proof}

If $r$ is a vertex in a connected graph $G$ and $V_{i}:= \{v \in V (G) : \dist_{G}(r, v) = i\}$ for all $i \geqslant 0$, then $(V_0, V_1, \dots)$ is called a \defn{BFS layering} of $G$ \defn{rooted} at $r$. Associated with a BFS layering is a \defn{BFS spanning tree} $T$ obtained by choosing, for each non-root vertex $v \in V_i$ with $i \geqslant 1$, a neighbour $w$ in $V_{i - 1}$, and adding the edge $vw$ to $T$. Thus $\dist_{T}(r, v) = \dist_{G}(r, v)$ for each vertex $v$ of $G$. For a partition $\mathcal{P}$ of a graph $G$, the \defn{quotient} of $\mathcal{P}$ is the graph, denoted by \defn{$G / \mathcal{P}$}, with vertex set $\mathcal{P}$ where distinct parts $A, B \in \mathcal{P}$ are adjacent in $G / \mathcal{P}$ if and only if some vertex in $A$ is adjacent in $G$ to some vertex in $B$.

To complete the proof of \cref{thm:RTWmain}, we use the following results due to \citet[Lemma~21]{DJMMUW20} and \citet[Corollary~6]{UWY22}.

\begin{lem}[\cite{DJMMUW20}]
\label{lem:genusPS}
Let $G$ be a connected graph with Euler genus $g$.
For every BFS spanning tree $T$ of $G$ rooted at some vertex $r$ with corresponding BFS layering $(V_0,V_1,\dots)$, 
there is a subgraph $Z\subseteq G$ with at most $2g$ vertices in each layer $V_i$, such that $Z$ is connected and $G-V(Z)$ is planar. 
Moreover, there is a connected planar graph $G^+$ containing $G-V(Z)$ as a subgraph, and 
there is a BFS spanning tree $T^+$ of $G^+$ rooted at some vertex $r^+$ 
with corresponding BFS layering $(W_0,W_1,\dots)$ of $G^+$, 
such that  $W_{i} \cap (V(G) \setminus V(Z)) = V_i  \setminus V(Z)$ for all $i\geq 0$, and
$P \cap (V(G) \setminus V(Z))$ is a vertical path in $T$ for every vertical path $P$ in $T^+$. 
\end{lem}

\begin{thm}[\citep{UWY22}]\label{thm:goodplanarPS} Let $T$ be a rooted spanning tree in a connected planar graph $G$. Then $G$ has a partition $\mathcal{P}$ into vertical paths in $T$ such that $\tw(G/\mathcal{P})\leq 6$.
\end{thm}

Products and partitions are inherently related, as observed by \citet[Observation~35]{DJMMUW20}.

\begin{obs}[\citep{DJMMUW20}]
\label{PartitionsProduct}
For a graph $H$, a graph $G$ is contained in $H\boxtimes P$ for some path $P$ if and only if there is a partition $\PP$ of $G$ and there is a layering $(V_0,V_1,\dots)$ of $G$, such that $G/\PP$ is contained in $H$ and $|X\cap V_i|\leq 1$ for each $X\in\PP$ and $i\geq 0$.
\end{obs}

\begin{cor}\label{cor:goodgenusPS}
    Every graph of Euler genus $g$ is $(2g+6,2g)$-good.
\end{cor}
\begin{proof}
    The class of $(2g+6,2g)$-good graphs is subgraph-closed, so it suffices to consider an arbitrary connected graph $G$ of Euler genus $g$. 
    Let $T$ be a BFS spanning tree of $G$ and $(V_0,V_1,\dots)$ be the corresponding BFS layering. Then there exist $Z$, $G^+$, $T^+$ and $(W_0, W_1, \dots)$ that satisfy all the properties given by \cref{lem:genusPS}. By \cref{thm:goodplanarPS}, there exists a partition $\mathcal{P}$ of $G^+$ such that:
    \begin{enumerate}
        \item the graph $H:=G^+/\mathcal{P}$ has treewidth at most $6$, and
        \item\label{item2} for each $S\in \mathcal{P}$, $S \cap (V(G)\setminus V(Z))$ induces a vertical path in $T$.
    \end{enumerate} 
    Recall that $|V(Z)\cap V_i|\leq 2g$ for each layer $V_i$, so there is a partition $\mathcal{Z}$ of $V(Z)$ with at most $2g$ parts so that each part contains at most one vertex in each layer.
   Define $\mathcal{P}':=\mathcal{Z}\cup \{S\cap (V(G)\setminus V(Z)):S\in \mathcal{P}\}$.
   By \ref{item2} and the definition of $\mathcal{Z}$, we have $|S\cap V_i| \leqslant 1$ for each $S\in \mathcal{P}'$ and each layer $V_i$.
   Observe that $(G/\mathcal{P}')-\mathcal{Z}\subseteq G^+/\mathcal{P}$, and so $G/\mathcal{P}'$ has treewidth at most $6+|\mathcal{Z}|\leq 6+2g$.
   Thus the result follows from \cref{PartitionsProduct}. 
\end{proof}

\begin{proof}[Proof of \cref{thm:RTWmain}]
As an easy consequence of Euler's formula, for $n\geq 3$ the maximum number of edges of an $n$-vertex bipartite graph of Euler genus $g$ is $2(n+g-2)$, and so $K_{3,2g+3}$ has Euler genus at least $g+1$.
Thus every graph of Euler genus $g$ is $K_{3,2g+3}$-minor-free, and so the result follows from \cref{lem:RTWtechnical} and Corollary~\ref{cor:goodgenusPS}.
\end{proof}

\section{Putting It All Together} 
\label{FinalSection}

This section combines the tools and results of \cref{tools,SectionColouring,WeakShallowMinorsSection} to prove \cref{introRTW,introLTW}, which establish bounds on the row treewidth and layered treewidth of certain topological $k$-matching-planar graphs.

We use the Coloured Planarisation Lemma (\cref{generallemma}) and the Distance Lemma (\cref{distancelemma}) to show that certain topological graphs have bounded row treewidth and layered treewidth.

\begin{lem} \label{generallayeredtreewidth} Suppose that a topological graph $G$ has a transparent ordered $c$-edge-colouring $\phi$ such that:

\begin{itemize}

\item for any $i, j \in \{1,\dots,c\}$ with $i < j$, for any edge $e$ of colour $i$ and for any fragment $\gamma$ of $e$, the matching number of the set of edges of colour $j$ that cross $\gamma$ is at most $m$,

\item for any $e \in E(G)$, the vertex cover number of the set of edges of colour less than $\phi(e)$ that cross $e$ is at most $k$.

\end{itemize}

Then

\begin{enumerate}
    \item $G$ is a weak $r$-shallow minor of $G^{\phi} \boxtimes K_{t}$ where $r := \frac{2^{c + 1}k^{c} - 2k - 1}{2k - 1}$ and $t := 1 + 5(c - 1)m$, and

    \item $\ltw(G) \leqslant 3t(4r + 1)$ and $\rtw(G) \leqslant (4r + 1)t((2(8r+1)t+3)7^{30r+6} - 1) - 1$.
\end{enumerate}

\end{lem}

\begin{proof}

By the Coloured Planarisation Lemma (\cref{generallemma}\ref{CPLa}), there exists a model $\mu$ of $G$ in $G^{\phi} \boxtimes K_{t}$. Let $v \in V(G)$ and $x \in V(G^{\phi})$ be two vertices such that $(x, i) \in \mu(v)$ for some $i \in \{1, \dots, t\}$. By \cref{generallemma}\ref{CPLc}, $x \in W_{vw} \setminus \{w\}$ for some edge $vw \in E(G)$ or $x = v$. By the Distance Lemma (\cref{distancelemma}), $\dist_{G^{\phi}}(v, x) \leqslant r$. So $\dist_{G^{\phi} \boxtimes K_{t}}((v, 1), (x, i)) \leqslant r$. As such, $\mu(v)$ has weak radius at most $r$ in $G^{\phi} \boxtimes K_{t}$. Thus $G$ is a weak $r$-shallow minor of $G^{\phi} \boxtimes K_{t}$.

By \cref{layeredtreewidthplanargraphs}, $\ltw(G^{\phi}) \leqslant 3$. By \cref{ltwWeakShallowMinor}, $\ltw(G) \leqslant 3t(4r + 1)$. By \cref{thm:RTWmain}, $\rtw(G) \leqslant (4r + 1)t((2(8r+1)t+3)7^{30r+6} - 1) - 1$.
\end{proof}

\cref{generallayeredtreewidth} implies that topological $k$-matching-planar graphs with bounded topological thickness have bounded row treewidth and layered treewidth.

\begin{lem} \label{MatchingPlanarOne} Let $G$ be a topological $k$-matching-planar graph with topological thickness $c$. Then

\begin{enumerate}
    \item $G$ is a weak $r$-shallow minor of $H \boxtimes K_{t}$ for some planar graph $H$ where $r := \frac{2^{2c + 1}k^{c} - 4k - 1}{4k - 1}$ and $t := 1 + 5(c - 1)k$, and

    \item $\ltw(G) \leqslant 3t(4r + 1)$ and $\rtw(G) \leqslant (4r + 1)t((2(8r+1)t+3)7^{30r+6} - 1) - 1$.
\end{enumerate}
    
\end{lem}

\begin{proof} Let $\phi$ be a transparent ordered $c$-edge-colouring of $G$, where we ‘order' the colours arbitrarily. By \cref{folklore}, for any $e \in E(G)$, the vertex cover number of the set of edges of colour less than $\phi(e)$ that cross $e$ is at most $2k$. The result follows from \cref{generallayeredtreewidth}.
\end{proof}

We now apply the main results of \cref{SectionColouring} to prove \cref{introRTW,introLTW}. Applying \cref{colouringkcoverplanar} with \cref{MatchingPlanarOne}, we obtain the following.

\begin{thm} \label{rcolouring} Let $G$ be a topological $k$-matching-planar graph such that for every vertex $v \in V(G)$, the set of edges incident to $v$ can be coloured with at most $s$ colours such that monochromatic edges do not cross. Then

\begin{enumerate}
    \item $G$ is a weak $r$-shallow minor of $H \boxtimes K_{\ell}$ for some planar graph $H$ where 
    $r \in 2^{\mathcal{O}(sk^{3}\log^2 k)}$
    and $\ell \in \mathcal{O}(sk^{4}\log k)$, and 
    \item $\ltw(G) \in 2^{\mathcal{O}(sk^{3}\log^2 k)}$
    and $\rtw(G) \in 2^{2^{\mathcal{O}(sk^{3}\log^2 k)}}$.
\end{enumerate}
\end{thm}

\cref{rcolouring} implies that simple topological $k$-matching-planar graphs have layered treewidth $2^{\mathcal{O}(k^{3}\log^{2} k)}$ and row treewidth $2^{2^{\mathcal{O}(k^{3}\log^{2} k)}}$.  The following result, which implies \cref{introRTW,introLTW}, is an immediate corollary of \cref{MatchingPlanarOne} and \cref{colouringtheorem}. 

\begin{thm} \label{thm:final} Let $G$ be a topological $k$-matching-planar graph with no $t$ pairwise crossing edges incident to a common vertex. Then 
     $G$ is a weak $r$-shallow minor of $H \boxtimes K_{\ell}$ for some planar graph $H$ where $r \in 2^{(k + 1)^{3}\log_{2}^{2}(k + 2)2^{\mathcal{O}(2^{(t - 1)(t - 2) / 2})}}$ and $\ell \in (k + 1)^{4}\log_{2}(k + 2)2^{\mathcal{O}(2^{(t - 1)(t - 2) / 2})}$. Moreover,
 \begin{align*}
\ltw(G) \in 2^{(k + 1)^{3}\log_{2}^{2}(k + 2)\cdot2^{\mathcal{O}(2^{(t - 1)(t - 2) / 2})}}
\quad\text{and}\quad
\rtw(G) \in 2^{2^{(k + 1)^{3}\log_{2}^{2}(k + 2)\cdot2^{\mathcal{O}(2^{(t - 1)(t - 2) / 2})}}}.
\end{align*}
\end{thm}

\section{Open Problems} \label{OpenProblems}

We conclude with four inter-related  open problems.

\begin{question} \label{PerfectRTW} 
Do $k$-matching-planar graphs have row treewidth at most some function $f(k)$, independent of the maximum number of pairwise crossing edges incident to a common vertex?
\end{question}

\begin{question} \label{FinalOpen} 
Does there exist a function $f$ such that every $k$-matching-planar graph is isomorphic to a topological $f(k)$-matching-planar graph with no $f(k)$ pairwise crossing edges incident to a common vertex?
\end{question}

Note that a positive answer to \cref{FinalOpen}, combined with \cref{thm:final}, would imply a positive answer to \cref{PerfectRTW}.

Given a beyond planar graph class $\mathcal{G}$, it is natural to ask if a graph in $\mathcal{G}$ can be redrawn in a `simple' way, maintaining the property of the class. Such redrawings are investigated by the graph drawing community. In particular, \cref{SimpleFan} due to \citet{KKRS23} says that every fan-planar graph is isomorphic to a simple topological fan-planar graph. \citet{PRTT-DCG06} proved that every $k$-planar graph for $k \leqslant 3$ is isomorphic to a simple topological $k$-planar graph. This ceases to be true for $k \geqslant 4$, as pointed out by \citet{Schaefer18}. On the other hand, \citet*{HLRT20} proved that every $k$-planar graph is isomorphic to a simple topological $f(k)$-planar graph, for some function $f$.  
\citet{HL24} constructed min-$2$-planar graphs that are not isomorphic to a simple topological min-$k$-planar graph for any fixed $k$.

Consider the analogous  question for $k$-matching-planar graphs. 

\begin{question} \label{SimpleQuestion} Does there exist a function $f$ such that every $k$-matching-planar graph is isomorphic to a simple topological $f(k)$-matching-planar graph?
\end{question}

Note that a positive answer to \cref{SimpleQuestion} would imply positive answers to \cref{FinalOpen,PerfectRTW}.

As discussed in \cref{sectionbackground}, there exists no function $f$ such that every topological $k$-matching-planar graph is topological $f(k)$-quasi-planar (because there might be an unbounded number of pairwise crossing edges incident to a common vertex). On the other hand, $k$-matching-planar graphs might be redrawn as $f(k)$-quasi-planar drawings. This leads to the following question.

\begin{question} \label{Quasi} Does there exist a function $f$ such that every $k$-matching-planar graph is $f(k)$-quasi-planar?
\end{question}

As discussed in \cref{sectionbackground}, topological $k$-matching-planar graphs with no $t$ pairwise crossing edges incident to a common vertex are topological $(2kt + 2)$-quasi-planar. So a positive answer to \cref{FinalOpen} (or \cref{SimpleQuestion}) would imply a positive answer to \cref{Quasi}. And a positive answer to \cref{SimpleQuestion} would imply a positive answer to all the open problems listed in this section.

\subsection*{Acknowledgements}

Thanks to Robert Hickingbotham, Torsten Ueckerdt and Jung Hon Yip for helpful comments.

{
\fontsize{10pt}{11pt}
\selectfont
\bibliographystyle{DavidNatbibStyle}
\bibliography{DavidBibliography}
}
\end{document}